\documentclass[a4paper,english,10pt]{article}

\makeatletter
\def\thm@space@setup{%
  \thm@preskip=\parskip \thm@postskip=0pt
}
\makeatother

\usepackage[left=.9in, right=.9in, bottom=1.2in, top=1in,dvips]{geometry}
\makeatletter

\AtBeginDocument{\setlength\abovedisplayskip{5pt plus 2pt minus 1pt}
\setlength\belowdisplayskip{5pt plus 2pt minus 1pt}}
\edef\restoreparindent{\parindent=\the\parindent\relax}
\usepackage{parskip}
\let\svpar\par
\edef\svparskip{\the\parskip}
\def\revertpar{\svpar\setlength\parskip{\svparskip}\let\par\svpar}
\def\noparskip{\leavevmode\setlength\parskip{0pt}%
  \def\par{\svpar\let\par\revertpar}}
\restoreparindent
\usepackage[utf8]{inputenc}
\usepackage{lmodern}
\usepackage{indentfirst}
\usepackage[affil-it]{authblk}
\usepackage{rotating}
\usepackage[nottoc,notlot,notlof]{tocbibind}
\usepackage{titlesec}
 \titleformat{\subparagraph}[hang]{\normalfont}{\thesubparagraph}{0pt}{\underline}
 \titleformat{\paragraph}[hang]{\normalfont}{\theparagraph}{0pt}{\myuline}
\usepackage{titletoc}

\usepackage[textsize=scriptsize]{todonotes} 
\makeatletter
\renewcommand{\todo}[2][]{\tikzexternaldisable\@todo[#1]{#2}\tikzexternalenable}
\makeatother
\usepackage{comment}

\usepackage[hidelinks]{hyperref}
\usepackage{enumitem}
\setlist[enumerate]{itemsep=2.0pt plus 1.0 pt minus 0.5pt, topsep=4.0pt plus 2.0 pt minus 1.0pt}
\setlist[itemize]{itemsep=2.0pt plus 1.0 pt minus 0.5pt, topsep=4.0pt plus 2.0 pt minus 1.0pt}
\usepackage{array}
\newcolumntype{M}[1]{>{\centering\arraybackslash}m{#1}}
\usepackage{multirow}
\usepackage{tikz}
\usetikzlibrary{external}
\usepackage{calc}
\usepackage{pgfplots}
\pgfplotsset{compat=1.16}
\usepackage{graphicx}
\usepackage{subcaption}
\usepackage[labelfont=bf]{caption}
\usepackage{epsfig}
\usepackage{color, soul}

\usepackage{contour}
\usepackage[normalem]{ulem}

\contourlength{1pt}

\newcommand{\myuline}[1]{%
  \uline{\phantom{#1}}%
  \llap{\contour{white}{#1}}%
}

\usepackage{marginnote}

\usepackage{color}
\usepackage{amsmath, amsthm, slashed}
\usepackage[capitalize]{cleveref}
\usepackage{amssymb}
\usepackage{bbm}
\usepackage{bm}
\usepackage{stackengine}
\usepackage{marvosym}
\usepackage{mathtools}
\usepackage{mathrsfs}
\usepackage{upgreek}
\usepackage{accents}
\usepackage{textcomp}
\usepackage{slashed} %for gradient on spheres
\usepackage{centernot}
\usepackage{cancel}
\usepackage{harmony} %for musical symbols
\usepackage{wasysym} %for clock
\wasyfamily
\DeclareFontShape{U}{wasy}{b}{n}{ <-10> ssub * wasy/m/n
 <10> <10.95> <12> <14.4> <17.28> <20.74> <24.88>wasyb10 }{}
\DeclareFontShape{U}{wasy}{b}{it}{ <-10> ssub * wasy/m/n
 <10> <10.95> <12> <14.4> <17.28> <20.74> <24.88>wasyb10 }{}
\DeclareMathAlphabet\mathbfcal{OMS}{cmsy}{b}{n}
\newcommand\numberthis{\addtocounter{equation}{1}\tag{\theequation}}
\renewcommand{\Re}{\operatorname{Re}}
\renewcommand{\Im}{\operatorname{Im}}

\allowdisplaybreaks

\numberwithin{equation}{section}

\theoremstyle{plain}

\newtheorem*{namedthm}{\namedthmname}
\newcounter{namedthm}


\newtheorem{theorem}{Theorem}[section]
\newtheorem*{theorem*}{Theorem}
\newtheorem{conjecture}[theorem]{Conjecture}
\newtheorem*{conjecture*}{Conjecture}

\newtheorem*{corollary*}{Corollary}
\newtheorem{proposition}[theorem]{Proposition}
\newtheorem{lemma}[theorem]{Lemma}

\theoremstyle{definition}

\newtheorem{definition}[theorem]{Definition}
\newtheorem{remark}[theorem]{Remark}
\newtheorem*{remark*}{Remark}

\usepackage{csquotes}
\let\OLDthebibliography\thebibliography
\renewcommand\thebibliography[1]{
  \OLDthebibliography{#1}
  \setlength{\parskip}{0pt}
  \setlength{\itemsep}{2pt plus 0.3ex}
  \setlength{\labelsep}{0em} %Change 5 em to taste
}

\newcommand{\mc}[1]{\mathcal{#1}}

\newcommand{\p}{\partial}
\newcommand{\lp}{\left}
\newcommand{\rp}{\right}

\newlength{\dhatheight}

\newcommand{\nablasl}{\slashed{\nabla}}

\DeclareMathOperator{\divsl}{\slashed{\mathrm{div}}}

\DeclareMathOperator{\curlsl}{\slashed{\mathrm{curl}}}

\DeclareMathOperator{\tr}{\mathrm{tr}}

\newcommand{\ueta}{\underline{\eta}}
\newcommand{\uchi}{\underline{\chi}}
\newcommand{\ualpha}{\underline{\alpha}}
\newcommand{\trchi}{\tr\chi}
\newcommand{\truchi}{\tr\uchi}

\newcommand{\epschi}{(\slashed{\varepsilon}\cdot {\chi})}
\newcommand{\epsuchi}{(\slashed{\varepsilon}\cdot \uchi)}
\newcommand{\uL}{\underline{L}}

\newcommand{\swei}[2]{#1^{[{#2}]}}

\newcommand{\bsy}{\ensuremath{\boldsymbol}}

\newcommand{\chih}{\widehat{\chi}}
\newcommand{\omegah}{\hat{\omega}}

\newcommand{\eo}{\bsy{e}_{\bsy{1}}}
\newcommand{\etw}{\bsy{e}_{\bsy{2}}}
\newcommand{\et}{\bsy{e}_{\bsy{3}}}
\newcommand{\ef}{\bsy{e}_{\bsy{4}}}

\newcommand{\chibh}{\widehat{\underline{\chi}}}
\newcommand{\omegabh}{\hat{\underline{\omega}}}
\newcommand{\etab}{\underline{\eta}}

\newcommand{\yb}{\underline{\xi}}

\newcommand{\alphab}{\underline{\alpha}}

\newcommand{\chib}{\underline{\chi}}
\newcommand{\xib}{\underline{\xi}}

\usepackage{totcount}

\newtotcounter{citnum} %From the package documentation
\begin{document}
 \title{\LARGE \textbf{The Maxwell equations on full sub-extremal and extremal Kerr spacetimes}}

\author[1]{{\Large Gabriele Benomio}}
\author[2]{{\Large  Rita \mbox{Teixeira da Costa}\vspace{0.4cm}}}

\affil[1]{\small  Gran Sasso Science Institute, Viale Francesco Crispi 7, L'Aquila (AQ), 67100,
Italy }
\affil[2]{\small 
University of Cambridge, Department of Pure Mathematics and Mathematical Statistics, Wilberforce Road, Cambridge CB3 0WA, United Kingdom}

\date{\today}

\maketitle

\begin{abstract}
We study the Cauchy problem for the Maxwell equations in the exterior region of Kerr black hole spacetimes. The equations are formulated for components of the Maxwell field relative to the algebraically special frame of Kerr, with the unknowns treated as tensorial quantities associated with a non-integrable horizontal distribution. The extremal Maxwell components \emph{decouple} into Teukolsky equations, whereas the middle Maxwell components form a \emph{coupled} system of transport and elliptic equations. Assuming control over the extremal components, we prove uniform boundedness (without loss of derivatives) and decay estimates for the middle components in the full $|a|\leq M$ range of spacetime parameters. Our analysis relies on (i) deriving a \emph{decoupled} system of transport and elliptic equations for two \emph{modified} middle Maxwell components and (ii) decomposing general solutions into a dynamical and stationary part, the latter determined by two real (electric and magnetic) charges which are entirely read off from the initial data at the event horizon. 

In the sub-extremal $|a|<M$ case, works of Shlapentokh-Rothman and the second author provide the necessary control over the extremal components, yielding \textit{unconditional} boundedness and decay results for all the unknowns in the equations.

In the extremal $|a|=M$ case, we formulate a conjectural boundedness and decay statement for the extremal components, motivated by work of Casals, Gralla and Zimmerman on fixed azimuthal mode solutions compactly supported away from the event horizon. Our boundedness and decay results for all the unknowns in the equations remain, therefore, \emph{conditional}. We show that the complicated dynamics of the extremal components at the event horizon is inherited by the middle components; in particular, we uncover novel conservation laws for the middle components of axisymmetric solutions.
\end{abstract}

\bigskip

\setcounter{tocdepth}{3}
\tableofcontents
\setcounter{tocdepth}{3}

%\newpage

\section{Introduction}

The Kerr black hole family is, conjecturally, the unique family of regular stationary black hole solutions to the Einstein vacuum equations
\begin{align}
    \mathrm{Ric}(g)=0. \label{eq:EVE}
\end{align}
It is is parametrized by two real numbers, $M>0$ and $|a|\leq M$. In this paper, we study the dynamics of electromagnetic waves on any fixed member of this family. Concretely, we consider the Maxwell equations
\begin{align}
    \mathrm{d}\, \bsy F&=0 \, , & \mathrm{d}\star \bsy F&=0 \, , \label{eq:Maxwell-eqs-intro}
\end{align}
where $\mathrm{d}$ and $\star$ represent the exterior derivative and Hodge dual, respectively, and where $\bsy F$ is the Faraday electromagnetic tensor, as a system of evolutionary PDEs on a fixed Kerr black hole background. Our first result concerns the sub-extremal $|a|<M$ Kerr sub-family:

\begin{theorem}  \label{thm:main-subextremal} Fix $M>0$. For any $|a|< M$, solutions of the Maxwell equations \eqref{eq:Maxwell-eqs-intro}  on the Kerr exterior arising from suitably regular initial data 
\begin{enumerate}[label=(\alph*),noitemsep]
    \item remain uniformly bounded in time in terms of initial data and,
    \item decay in time along a suitable spacelike foliation, after subtracting a  stationary solution itself determined by initial data.
\end{enumerate}
\end{theorem}

The boundedness statement we obtain \textit{does not lose derivatives}: we show that a coercive $L^2$-based energy flux of the Maxwell field through a suitable foliation of the Kerr exterior is controlled by an energy flux, at the same level of differentiability, of the data. In that sense, our result contains not only a suitable asymptotic stability statement in Theorem~\ref{thm:main-subextremal}(b), but also the orbital stability for \eqref{eq:Maxwell-eqs-intro} in Theorem~\ref{thm:main-subextremal}(a). Our results were previously known for the Schwarzschild ($a=0$) and very slowly rotating Kerr ($|a|\ll M$) sub-families by work of other authors \cite{Blue2008,Andersson2015a,Pasqualotto2016}, see also \cite{Ghanem2024}. In the full sub-extremal $|a|<M$ range, Theorem~\ref{thm:main-subextremal} contains, to our knowledge, the first orbital stability for \eqref{eq:Maxwell-eqs-intro}. Prior to this, \cite{Ma2022} obtained a uniform sharp decay result for the Maxwell equations \eqref{eq:Maxwell-eqs-intro} on Kerr in the full sub-extremal $|a|<M$ range, i.e.\ the fastest possible decay for nontrivial compactly supported initial data. The decay rates in Theorem~\ref{thm:main-subextremal}(b) are weaker than those of \cite{Ma2022} as, motivated by nonlinear applications, we allow for less regular data.

As we shall describe in Section \ref{sec_intro_proof}, our strategy to prove Theorem~\ref{thm:main-subextremal} relies on energy boundedness and integrated local energy decay for the so-called spin $\pm 1$ Teukolsky equations \cite{Teukolsky1973}. The spin $\pm 1$ Teukolsky equations are wave-type equations for certain components of $\bsy F$, known as \emph{extremal} components, which are implied by \eqref{eq:Maxwell-eqs-intro} and decouple from the system of Maxwell equations. The analysis of the Teukolsky equations has been recently carried out in the full sub-extremal range in \cite{SRTdC2020,SRTdC2023}, and we shall make use of those results here.\footnote{It should be noted that the earlier results of \cite{Ma2022} for the full sub-extremal range $|a|\leq M$ are \textit{conditional} on energy boundedness and integrated local energy decay for the spin $\pm 1$ Teukolsky equations. In the case $|a|\ll M$, the latter was shown by the same author in \cite{Ma2017}.} Indeed, as in the physics tradition (see e.g.\ \cite{Chandrasekhar}), we interpret the Maxwell equations \eqref{eq:Maxwell-eqs-intro} as \emph{transport} equations for the non-extremal, or \textit{middle}, Maxwell components with forcing terms given by the extremal (Teukolsky) components. This should be contrasted with earlier works \cite{Blue2008,Andersson2015a} for $|a|\ll M$, which are based on the Fackerell--Ipser equations \cite{Fackerell1972}, and precede the introduction in \cite{Dafermos2016a} of the toolbox required to tackle the Teukolsky equations. We also mention an alternative approach in the $a=0$ case in \cite{Andersson2016}, based on a superenergy tensor.

Before turning to a careful description of the above strategy, let us remark that our analysis of the middle Maxwell components is implemented so as to extend to the $|a|=M$ extremal Kerr case. Indeed, we show:
\begin{theorem}\label{thm:main-extremal} Fix $M>0$. For any $|a|\leq  M$, solutions of the Maxwell equations on \eqref{eq:Maxwell-eqs-intro} on the Kerr exterior arising from suitably regular initial data {\normalfont supported on a fixed azimuthal mode}
\begin{enumerate}[label=(\alph*),noitemsep]
    \item remain uniformly bounded in time, {\normalfont away from the event horizon}, in terms of initial data and,
    \item decay in time, {\normalfont away from the event horizon}, along a suitable spacelike foliation, after subtracting a stationary solution itself determined by initial data,
\end{enumerate}
{\normalfont assuming a conjectural statement for solutions of the spin $\pm 1$ Teukolsky equations (see already Conjecture~\ref{conj:teukolsky-EB-ILED-ext}).}
\end{theorem}

The added emphasis allows one to easily compare the $|a|\leq M$ case with the $|a|<M$ case in the previous Theorem~\ref{thm:main-subextremal}; evidently, Theorem \ref{thm:main-extremal} is much weaker. Indeed, even though our approach in Theorem~\ref{thm:main-subextremal} to deduce estimates for Maxwell \eqref{eq:Maxwell-eqs-intro} from estimates  for Teukolsky fully generalizes to the entire Kerr black hole family $|a|\leq M$, the approach of \cite{SRTdC2020,SRTdC2023} to the Teukolsky equations \textit{does not}: in view of the growing body of work concerning waves on extremal Kerr black holes and their spherically symmetric charged analogues, we expect the behavior of solutions to the Teukolsky equation to be far worse for $|a|=M$ than for $|a|<M$. The conjectured statement on which we rely in Theorem~\ref{thm:main-extremal} is based on this literature, in particular on the heuristics in \cite{Casals2016,Casals2019a,Gralla2018} and on the rigorous results in \cite{Gajic2023}. We direct the reader to Section~\ref{sec:teukoslsky-extremal} for further details and a more comprehensive comparison with the literature.

\subsection{Overview of the proofs} \label{sec_intro_proof}

Let us now turn to an overview of the main ideas involved in our proofs of Theorems~\ref{thm:main-subextremal} and \ref{thm:main-extremal}. As mentioned above, following the physics tradition, we base our analysis of the Maxwell equations \eqref{eq:Maxwell-eqs-intro} on a black-box understanding of the spin $\pm 1$ Teukolsky equations. To do so, one would like to rewrite \eqref{eq:Maxwell-eqs-intro} in a geometric formalism where the Teukolsky equations naturally emerge and can be directly exploited. 

In \cite{Benomio2020,Benomio2022}, the first author develops a formalism for casting geometric equations, such as \eqref{eq:EVE} or \eqref{eq:Maxwell-eqs-intro}, on $1+3$-dimensional Lorentzian manifolds in terms of local frames $\mathcal{N}=(e_1,e_2,e_3,e_4)$ where $(e_3,e_4)$ are null vectors and the associated horizontal distribution $\mathfrak D_{\mc N}=\langle e_3, e_4\rangle^\bot$, for which $(e_1,e_2)$ form a local basis, is allowed to be non-integrable. This is inspired by the mathematical literature, in particular the seminal works \cite{Christodoulou1990, Christodoulou1993}, but also connects naturally to the physics literature. Indeed, non-integrable local frames are often used in classical physics textbooks such as \cite{Chandrasekhar} in the study of equations \eqref{eq:EVE} or \eqref{eq:Maxwell-eqs-intro} on black hole spacetimes. More specifically, the \textit{algebraically special frame} $\mc N_{\rm as}$ of Kerr, i.e.\ the choice $(e_3,e_4)=(e_3^{\rm as},e_4^{\rm as})$ of principal null vectors, is traditionally used in the study of equations \eqref{eq:EVE} (typically its linearisation) and \eqref{eq:Maxwell-eqs-intro} on Kerr spacetimes and is well-known to yield a non-integrable horizontal distribution $\mathfrak D_{\mc N_{\rm as}}=\langle e_3^{\rm as}, e_4^{\rm as}\rangle^\bot$. The complex $s$-spin-weighted functions arising in the Newman--Penrose and Geroch--Held--Penrose formalism \cite{Newman1962, Geroch1973} in which these works are formulated, see e.g.~the seminal papers \cite{Fackerell1972,Teukolsky1973}, are in $1$-to-$1$ correspondence with the tensors over $\mathfrak D_{\mc N_{\rm as}}$ introduced in \cite{Benomio2020,Benomio2022}.

The formalism of \cite{Benomio2020,Benomio2022} showcases the remarkable hidden structure of \eqref{eq:EVE} and \eqref{eq:Maxwell-eqs-intro} around Kerr spacetimes which are displayed by classical physics works while rewriting them in the terms of modern mathematical literature. The Maxwell equations~\eqref{eq:Maxwell-eqs-intro} are cast into a system of first order partial differential equations for 4 unknowns---two 1-tensors, $\bsy{\alpha}$ and $\bsy\alphab$, and two scalar functions, ${\bsy{\rho}}$ and ${\bsy{\sigma}}$, on $\mathfrak D_{\mc N_{\rm as}}$---which are frame components of the Faraday tensor $\bsy F$ relative to the frame $\mc N_{\rm as}$ and make up the entirety of the degrees of freedom of $\bsy F$. We show that the 1-tensors $\bsy{\alpha}$ and $\bsy\alphab$ satisfy the tensorial analogue of the spin $\pm 1$ Teukolsky equations, as well as the tensorial analogue of the so-called Teukolsky--Starobinsky identities; schematically, we obtain
\begin{align*}
   \text{\eqref{eq:Maxwell-eqs-intro}} &\implies (\Box_g +\partial)\bsy\alpha=0, \quad (\Box_g -\partial)\bsy \alphab=0, \quad \partial^2_{\rm null} (\bsy \alpha,\bsy \alphab) = \partial^2_{\rm angular} (\bsy \alphab, \bsy \alpha) \, .
\end{align*}
We demonstrate that, after a very simple linear transformation $(\bsy \rho, \bsy\sigma)\mapsto (\widehat{\bsy{\rho}},\widehat{\bsy{\sigma}})$, the Maxwell equations are equivalent to the statement that all derivatives of the scalars $(\widehat{\bsy{\rho}},\widehat{\bsy{\sigma}})$ are decoupled from each other (as opposed to those of the original quantities $(\bsy \rho, \bsy\sigma)$) and are fully determined by $\bsy{\alpha}$ and $\bsy\alphab$ and their derivatives; schematically, we have
\begin{align}
   \text{\eqref{eq:Maxwell-eqs-intro}} \iff \partial \widehat{\bsy{\rho}}=\partial^{\leq 1}\bsy\alpha+\partial^{\leq 1}\bsy\alphab, \quad  \partial \widehat{\bsy{\sigma}}=\partial^{\leq 1}\bsy\alpha+\partial^{\leq 1}\bsy\alphab. \label{eq:Maxwell-schematic}
\end{align}
These observations, with the input of \cite{SRTdC2020,SRTdC2023} for $|a|<M$, and the conjectural statement for the dynamics $\bsy{\alpha}$ and $\bsy\alphab$ for $|a|\approx M$, already imply boundedness and decay properties for (suitably weighted) derivatives of $\bsy F$ as per Theorems~\ref{thm:main-subextremal} and~\ref{thm:main-extremal}.

The system of equations~\eqref{eq:Maxwell-eqs-intro} admits stationary (thus non-decaying) solutions $\bsy F_{\textup{stat}}$. In general, one expects that dynamical solutions to \eqref{eq:Maxwell-eqs-intro} converge to a stationary solution $\bsy F_{\textup{stat}}$, and therefore decay \emph{to zero} of $\bsy F$ itself can only hold after subtracting a suitable $\bsy F_{\textup{stat}}$. Stationary solutions can be fully parametrized by two topological charges\footnote{By standard topological arguments, for any (possibly dynamical) solution $\bsy F$ to \eqref{eq:Maxwell-eqs-intro}, the integrals $\int_{\mathbb{S}^2} \star \bsy F$ and $\int_{\mathbb{S}^2} \bsy F$ are independent of the foliation $\mathbb{S}^2$-spheres on which are computed, and therefore always constant.}
\begin{align*}
  \mathfrak  q_E&= \frac{1}{4\pi}\int_{\mathbb{S}^2} \star \bsy F_{\textup{stat}} \, ,  &     \mathfrak q_B&= \frac{1}{4\pi}\int_{\mathbb{S}^2} \bsy F_{\textup{stat}}
\end{align*}
computed over (any) foliation $\mathbb{S}^2$-spheres, see e.g.\ \cite[Section 3.1]{Andersson2015a} and \cite{Ma2022} for a discussion and use of these charges. For our system of equations \eqref{eq:Maxwell-eqs-intro}, with unknowns $(\bsy{\alpha},\bsy\alphab,\widehat{\bsy{\rho}},\widehat{\bsy{\sigma}})$, stationary solutions correspond to the family of solutions such that
\begin{align} \label{intro_stationary_solns}
    \bsy{\alpha_{\textup{stat}}}&=0 \, , & \bsy{\alphab}_{\textup{stat}}&=0 \, , & \widehat{\bsy{\rho}}_{\textup{stat}}&=\mathfrak{q}_1 \, , & \widehat{\bsy{\sigma}}_{\textup{stat}}&=\mathfrak{q}_2 
\end{align}
on the entire Kerr exterior spacetime, where the real constants $\mathfrak{q}_1,\mathfrak{q}_2$ are interpreted as the two charges of the stationary solution.%

It is natural to choose the stationary charges $\mathfrak{q}_1,\mathfrak{q}_2$ so as to coincide with the zero modes
\begin{align} \label{intro_zero_modes}
    &\int_{\mathbb{S}^2}\widehat{\bsy{\rho}} \, , & &\int_{\mathbb{S}^2}\widehat{\bsy{\sigma}} \, .
\end{align}
However, contrary to the integrals of $\bsy{F}$ and $\star \bsy F$ over the foliation $\mathbb{S}^2$-spheres, the integrals \eqref{intro_zero_modes} are, for $|a|>0$, dynamical quantities,\footnote{The zero modes of $\widehat{\bsy{\rho}}$ and $\widehat{\bsy{\sigma}}$ are constant for $a=0$.} thus requiring a \emph{choice of sphere} in \eqref{intro_zero_modes} where the charges are computed. 

Given a general solution $(\bsy{\alpha},\bsy\alphab,\widehat{\bsy{\rho}},\widehat{\bsy{\sigma}})$ arising from initial data prescribed on a spacelike (hyperboloidal) hypersurface $\Sigma_0$, in order to subtract the non-zero constant to which the general $\widehat{\bsy{\rho}}$ and $\widehat{\bsy{\sigma}}$ are converging (i.e.\ proving decay to zero), one may compute the charges on the \emph{final horizon sphere}, that is 
\begin{align*}
{\mathfrak q}_1&=\lim_{\tau\to \infty} \int_{\mathbb S^2_{\tau,r_+}}\widehat{\bsy\rho}\, d\Omega\,, & {\mathfrak q}_2&=\lim_{\tau\to \infty} \int_{\mathbb S^2_{\tau,r_+}}\widehat{\bsy\rho}\, d\Omega\,,
\end{align*}
where $\mathbb{S}^2_{\tau,r_+}=\mathcal{H}^+\cap\Sigma_\tau$ for spacelike hyperboloidal foliation $\Sigma_{\tau\geq 0}$. Such a choice of stationary charges immediately implies that the zero modes of $(\widehat{\bsy{\rho}}-\widehat{\bsy{\rho}}_{\textup{stat}},\widehat{\bsy{\sigma}}-\widehat{\bsy{\sigma}}_{\textup{stat}})$ decay to zero along the event horizon, and thus $(\widehat{\bsy{\rho}}-\widehat{\bsy{\rho}}_{\textup{stat}},\widehat{\bsy{\sigma}}-\widehat{\bsy{\sigma}}_{\textup{stat}})$ themselves decay to zero along the event horizon in the full $|a|\leq M$ range. The a-priori control over $\widehat{\bsy\rho}$ and $\widehat{\bsy\sigma}$ along the event horizon provided by our approach may be compared to possibly different approaches to the analysis of the Maxwell equations on black hole spacetimes in which control over the middle Maxwell components along the event horizon first requires control of these quantities along hypersurfaces located far from the event horizon (see e.g.~\cite{Ma2022} in the context of establishing sharp decay rates for the Maxwell equations on Kerr spacetimes).

Let us pause to highlight an important byproduct of our choice to identify $(\widehat{\bsy\rho}_{\rm stat},\widehat{\bsy\sigma}_{\rm stat})$ at the event horizon. 
In linear problems, one expects to be able to choose stationary charges entirely at the level of initial data (as opposed to nonlinear problems, where modulation procedures are typically necessary). To that end, we prove that our charges $(\mathfrak q_1,\mathfrak q_2)$, originally identified teleologically, can actually be read off from the initial data for the Maxwell system:
\begin{align*} 
   \mathfrak q_1&= \mathfrak{Q}_1(0) \, , & \mathfrak q_2&= \mathfrak{Q}_2(0) \,,
\end{align*}
where we have introduced scalar functions $\mathfrak{Q}_{1,2}:[0,\infty)\rightarrow\mathbb{R}$ defined through
\begin{align*}
\mathfrak{Q}_1(\tau) &=\frac{1}{I_{a,M}}\int_{\mathbb{S}^2_{\tau,r_+}}\left( \frac{r_+^2-a^2\cos^2\theta}{r_+^2+a^2\cos^2\theta}\,\widehat{\bsy{\rho}}+\frac{2ar_+\cos\theta}{r_+^2+a^2\cos^2\theta}\, \widehat{\bsy{\sigma}}-\frac{1}{2}(r_+^2+a^2\cos^2\theta) \left(\mathfrak k,\widetilde{\bsy{\alpha}}\right)\right)d\Omega\,, \\
 \mathfrak{Q}_2(\tau) &=\frac{1}{I_{a,M}}\int_{\mathbb{S}^2_{\tau,r_+}}\left( \frac{r_+^2-a^2\cos^2\theta}{r_+^2+a^2\cos^2\theta}\, \widehat{\bsy{\sigma}}-\frac{2ar_+\cos\theta}{r_+^2+a^2\cos^2\theta}\,\widehat{\bsy{\rho}}+\frac{1}{2}(r_+^2+a^2\cos^2\theta) \mathfrak k\wedge \widetilde{\bsy{\alpha}}\right)d\Omega\,.
\end{align*}
Here $I_{a,M}\neq 0$ is a constant and $\mathfrak k$ is an explicit $\mathfrak D_{\mc N_{\rm as}}$ one-form depending only $a,M$, and such that $\mathfrak k=0$ for $a=0$; $\widetilde{\bsy\alpha}$ is a $\mathbb S^2_{\tau,r_+}$ tensor obtained as the projection of $\bsy\alpha$ over the $\mathbb S^2_{\tau,r_+}$-spheres via the $\mathbb{S}^2$-projection procedure of \cite{Benomio2022}. Note that, for $a=0$, $\mathfrak Q_1$ and $\mathfrak Q_2$ reduce to the  zero modes of $\widehat{\bsy\rho}$ and $\widehat{\bsy\sigma}$. The a-posteriori retracing of our charges to the initial data relies on two horizon conservation laws 
\begin{align*}
    \p_\tau\mathfrak{Q}_{1}&=0 \, , & \p_\tau\mathfrak{Q}_{2}&=0 \,.
\end{align*}
That these hold crucially depends on the enhanced properties of the algebraically special frame and 
$\mathbb{S}^2$-projection procedure at the event horizon (see Sections~\ref{sec:algebraically-special-prelims} and \ref{sec:projection_formulae} below): no such conservation laws hold away from the event horizon.

In the sub-extremal $|a|<M$ case, the horizon decay of $\widehat{\bsy\rho}$ and $\widehat{\bsy\sigma}$ can be extended to decay off the horizon by a standard integration of the transport equations \eqref{eq:Maxwell-schematic} in the direction transversal to $\mathcal{H}^+$. This extension exploits the decay of \emph{both} $\bsy\alpha$ and $\bsy\alphab$ (and their derivatives) along $\mathcal{H}^+$. On the other hand, in the extremal $|a|=M$ case, one does \emph{not} expect decay of both $\bsy\alpha$ and $\bsy\alphab$ along $\mathcal{H}^+$. In particular, $\bsy\alphab$ is expected to grow in time along $\mc H^+$ (see e.g.\ \cite{Casals2016,Gralla2018,Gajic2023}), thus preventing one from extending decay of $\widehat{\bsy\rho}$ and $\widehat{\bsy\sigma}$ off $\mc H^+$ via the standard integration procedure adopted for the $|a|<M$ case.

To circumvent this obstruction, we develop a strategy to integrate the transport equations \eqref{eq:Maxwell-schematic} which only requires \emph{degenerate}-at-the-horizon control of $\bsy\alphab$. The crucial observation is that the transport equations \eqref{eq:Maxwell-schematic} can be combined with the Teukolsky equations to obtain, schematically,
\begin{align*}
\p_r (\widehat{\bsy\rho},\widehat{\bsy\sigma}) = \p_r \underline{\bsy\Phi}+ \text{faster decaying terms}, \numberthis\label{eq:crucial-identity-intro}
\end{align*}
at the future event horizon $\mc H^+$ (see already Section~\ref{sec:warm-up-spherically-symmetric} for the precise computation in a simpler setting) and, for $|a|=M$, $\underline{\bsy\Phi}$ denotes the derivative of $\bsy\alphab$ along the null generator of $\mc H^+$. Note that, in view of the work in \cite{Aretakis2012,Lucietti2012,Apetroaie2022} on conservation laws for axisymmetric solutions of the Teukolsky equations along $\mc H^+$, this pointwise identity immediately implies 

\begin{proposition} For axisymmetric solutions of the Maxwell system on extremal Kerr ($a=M$), there are conservation laws for transversal derivatives of $(\widehat{\bsy\rho},\widehat{\bsy\sigma})$ along $\mc H^+$.
\end{proposition}

See Remark~\ref{rmk:extremal-axisym} for a precise statement. On the other hand, for general solutions, we recall  from the heuristics \cite{Casals2016,Gralla2018} that while one expects $\p_r\underline{\bsy\Phi}$ to \textit{grow} along $\mc H^+$, one expects $\underline{\bsy\Phi}$ to decay. Integrating the pointwise identity \eqref{eq:crucial-identity-intro}, we deduce
\begin{align*}
(\widehat{\bsy\rho},\widehat{\bsy\sigma}) 
= (\widehat{\bsy\rho},\widehat{\bsy\sigma}) |_{\mc H^+} + \underline{\bsy\Phi} + \text{faster decaying terms}.
\end{align*} 
Thus, decay of $(\widehat{\bsy\rho},\widehat{\bsy\sigma})$ along $\mc H^+$ is propagated to the entire Kerr exterior manifold. 

\subsection{Related works and outlook}

We conclude this introduction by making contact with related problems in the literature. We have briefly mentioned applications of Theorems~\ref{thm:main-subextremal} and \ref{thm:main-extremal}  to nonlinear problems. The simplest such problems are nonlinear theories of electromagnetism, such as the Maxwell--Born--Infeld theory, on Kerr spacetimes. This was previously considered by \cite{Pasqualotto2017} in the $a=0$ case.

A more challenging source of nonlinear problems comes from coupling electromagnetism to gravity through the Einstein--Maxwell system
\begin{gather*}
\begin{gathered}
\mathrm{Ric}(\bsy g)=2\,\bsy F\cdot \bsy F^{\bsy \sharp_2}-\frac12 |\bsy F|^2_{\bsy g}\, \bsy g\,,\\
\mathrm{d}\, \bsy F=0, \quad \mathrm{d}\star \bsy F=0\,. \label{eq:EM-system}
\end{gathered}
\end{gather*}
Within the Einstein--Maxwell system, Kerr black holes sit as the $Q=0$ sub-family of the wider Kerr--Newman family of black hole solutions, which is parametrized by $(a,M,Q)$ with $Q>0$ and $a^2+Q^2\leq M^2$. While it might be tempting to extrapolate from Theorems~\ref{thm:main-subextremal} and \ref{thm:main-extremal} consequences for the Einstein--Maxwell system \eqref{eq:EM-system}, we caution the reader that the Maxwell equations \eqref{eq:Maxwell-eqs-intro} are \textit{not} the linearisation of \eqref{eq:EM-system} around Kerr(--Newmann). Indeed, by linearising around the Kerr--Newman family, one obtains an intricate system of coupled gravitational and electromagnetic perturbations. In the $a=0$ case, this linearised system has been extensively studied in \cite{Giorgi2019a} ($|Q|<M$) and \cite{Apetroaie2022} ($|Q|=M$, dynamics of the extremal components). For the nonlinear dynamics in the $|Q|=M$ case under spherical symmetry---where, in view of the rigidity of \eqref{eq:EM-system}, one needs to add a scalar field to the system as an additional degree of freedom--- a rather complete picture has been obtained in \cite{Angelopoulos2024}; see also the earlier numerics \cite{Murata2013}.  For $a\neq 0$, though much less is known, we direct the reader to the works \cite{Giorgi2023,Giorgi2024} for substantial recent progress. In this direction, we also mention the works \cite{Sterbenz2014,Metcalfe2017} which consider the Maxwell equations on more general spacetimes motivated by perturbations of black holes.

\subsection{Outline of the paper}

We briefly outline the structure of the paper:
\begin{itemize}
\item \textbf{Sections \ref{sec:Kerr-prelims}.} We introduce the Kerr exterior manifold, the algebraically special frame and the differentiable structures which will be useful for the sequel.

\item \textbf{Sections \ref{sec:maxwell_eqns_head}.} We derive the Maxwell equations on a general spacetime. The equations are formulated for components of the Maxwell field relative to a general (possibly non-integrable) null frame. We then specialize the equations to the Kerr exterior manifold and discuss the notion of seed initial data and prove well-posedness of the Cauchy problem. 

\item \textbf{Section \ref{sec:energy_norms_head}.} We introduce the energy flux norms which will be used in the analysis.

\item \textbf{Section \ref{sec:statement_theorems}.} We state the main theorems of the paper.

\item \textbf{Section \ref{sec:analysis-middle}.} This is the core of the paper. We prove boundedness and decay estimates on the middle Maxwell components in terms of suitable energy flux norms for the extremal Maxwell components. We also derive a precise version of the identity \eqref{eq:crucial-identity-intro} valid in the full $|a|\leq M$ range. As a direct application, we obtain conservation laws along the event horizon for axisymmetric solutions to the Maxwell equations in the extremal $|a|=M$ case.

\item \textbf{Section \ref{sec:analysis-teukolsky}.} We state the boundedness and decay estimates for the extremal Maxwell components. In the $|a|<M$ case, these are taken from \cite{SRTdC2020,SRTdC2023} as a black box. In the $|a|=M$ case, these were previously formulated as a conjecture in Section \ref{sec:statement_theorems} that we now motivate.

\item \textbf{Section \ref{sec:proof_main_theorems}.} We combine the results of Sections \ref{sec:analysis-middle}--\ref{sec:analysis-teukolsky} to obtain the proof of our main theorems stated in Section \ref{sec:statement_theorems}.

\item \textbf{Appendix \ref{sec:appendix_head}.} We derive several identities from the Maxwell equations, such as the tensorial Teukolsky--Starobinsky identities, the tensorial Teukolsky equations and the tensorial Fackerell--Ipser equations. We also rewrite the Maxwell equations for transformed spin-weighted quantities.
\end{itemize}

\subsection*{Acknowledgments}

RTdC acknowledges support from the National Science Foundation (US) award DMS2103173, as well as the Engineering and Physical Sciences Research Council (UK) award UKRI1793. The authors thank D.\ Gajic and R.\ Unger for comments on a preliminary version of this manuscript.

%\newpage  

%-------------------------------------------------------------------
%
% THE KERR SPACETIME
%
%-------------------------------------------------------------------
\section{The Kerr exterior manifold}
\label{sec:Kerr-prelims}

In this section, we introduce the Kerr exterior manifold, its algebraically special frame and the spacetime foliations which will be necessary in our analysis.

\subsection{Definition}

We define the manifold-with-boundary 
\begin{equation}  \label{ambient_manifold}
\mathcal{M}:= (-\infty,\infty) \times [0,\infty) \times \mathbb{S}^2 
\end{equation}
with coordinates $\bar{t}\in (-\infty,\infty)$, $y\in [0,\infty)$ and standard (local) spherical coordinates $(\bar{\theta},\bar{\phi})\in\mathbb{S}^2$. We define the \emph{(future) event horizon} as the boundary of $\mathcal{M}$,
\begin{equation*}
\mathcal{H}^+:=\partial\mathcal{M} = \{y=0\}\,,  
\end{equation*} 
 and \emph{(future) null infinity} as the formal hypersurface
\begin{equation*}
\mathcal{I}^+ :=\left\lbrace \bar{t}=\infty \right\rbrace \, .
\end{equation*}

We define the vector fields
\begin{align}
T& :=\partial_{\bar{t}} \, ,  &  Z& :=\partial_{\bar{\phi}} \label{eq:TZ-def}
\end{align}
on $\mathcal{M}$. We will denote by $\widetilde{Z}$ the smooth (degenerate) extension of $Z$ to a global vector field on $\mathbb{S}^2$.

Given real parameters $a$ and $M$, with $|a|\leq M$, we define the positive constants
\begin{equation*}
r_{\pm}:=M\pm\sqrt{M^2-a^2}
\end{equation*}
and a new coordinate $\bar{r}=\bar{r}(y)$ such that
\begin{equation*}
\bar{r}:[0,\infty)\rightarrow [r_+,\infty) \, ,
\end{equation*}
with $\bar{r}(3M)=y(3M)$. The refer to the coordinates $$(\bar{t},\bar{r},\bar{\theta},\bar{\phi})$$ as \emph{Kerr coordinates}. We note that $\mathcal{H}^+=\left\lbrace \bar{r}=r_+ \right\rbrace$ in Kerr coordinates.

For fixed $|a|\leq M$, we define the smooth scalar functions
\begin{align*}
\Delta(\bar{r}) &:= (\bar{r}-r_+)(\bar{r}-r_-) \, , &
\Sigma(\bar{r},\bar{\theta}) &:= \bar{r}^2+a^2\cos^2\bar{\theta} 
\end{align*}
on $\mathcal{M}$ and define the \emph{Kerr family of metrics} as the two-parameter family of Lorentzian metrics $g_{a,M}$ on $\mathcal{M}$ such that 
\begin{align}
g_{a,M}= &-\left(1-\frac{2M\bar{r}}{\Sigma}\right){d\bar{t}}^2+2\,d\bar{t}\, d\bar{r}+\Sigma\, {d\bar{\theta}}{}^2 +\frac{(\bar{r}^2+a^2)^2-\Delta\, a^2\sin^2\bar{\theta}}{\Sigma}\,\sin^2\bar{\theta} \, {d\bar{\phi}}{}^{2}   \label{def_kerr_metric} \\
&-2\,a\sin^2\bar{\theta} \, d\bar{r} \,d\bar{\phi}-\frac{4aM\bar{r}}{\Sigma}\,\sin^2\bar{\theta} \, d\bar{t}\, d\bar{\phi} \, .   \nonumber
\end{align}
The metric $g_{a,M}$ in \eqref{def_kerr_metric} is manifestly \emph{smooth} on $\mathcal{M}$, including on $\mathcal{H}^+$. The event horizon $\mathcal{H}^+$ can be checked to be a \emph{null} hypersurface relative to $g_{a,M}$, whereas the level sets of $r_+<\bar{r}<\infty$ are \emph{timelike} hypersurfaces. The smooth Lorentzian manifold $(\mathcal{M},g_{a,M})$ will be referred to as the \emph{Kerr exterior manifold}. One can check that $(\mathcal{M},g_{a,M})$ solves the vacuum Einstein equations, i.e.\ $g_{a,M}$ is Ricci-flat. In the sequel, the Kerr metric will be denoted by $g$.

\subsubsection{The Kerr-star and Boyer--Lindquist differentiable structures} \label{sec_Kerr_star_coords}

For fixed $|a|\leq M$, we consider a smooth scalar function $\iota=\iota(\bar{r})$ such that $$\iota:[r_+,\infty)\rightarrow [0,1]$$ and
\begin{equation}
\iota(\bar{r})=
\begin{cases}
1, \quad r_+\leq \bar{r} \leq 17M/8 \\
0, \quad \bar{r}\geq 9M/4 
\end{cases} \, ,
\end{equation}
with
\begin{gather*}
\frac{\bar{r}^2+a^2}{\Delta}-\iota(\bar{r})\left(\frac{\bar{r}^2+a^2}{\Delta}-1 \right) >0 \, , \\
2-\left( 1-\frac{2M\bar{r}}{\Sigma}\right) \left( \frac{\bar{r}^2+a^2}{\Delta}-\iota(\bar{r})\left(\frac{\bar{r}^2+a^2}{\Delta}-1 \right) \right) >0 \, .
\end{gather*}

We define \emph{Kerr-star coordinates}\footnote{With a slight abuse of notation, we do not rename the coordinates $(\bar{r},\bar{\theta})$.}
\begin{equation} \label{kerr_star_coords}
(t^*,\bar{r},\bar{\theta},\phi^*)
\end{equation}
on $\mathcal{M}$ such that
\begin{align*}
t^*&= \bar{t}-\int_{r_0}^{\bar{r}}\lp[(1-\iota(\bar{r}{}^{\prime}))\,\frac{{\bar{r}{}^{\prime}}{}^2+a^2}{\Delta(\bar{r}{}^{\prime})}+\iota(\bar{r}{}^{\prime})\rp]d\bar{r}{}^{\prime} \, , \\
\phi^*&=\bar{\phi}-\int_{r_0}^{\bar{r}}(1-\iota(\bar{r}{}^{\prime}))\,\frac{a}{\Delta(\bar{r}{}^{\prime})}\,d\bar{r}{}^{\prime} \quad \text{mod} \, 2\pi \, ,
\end{align*}
with $r_0=9M/4$. It follows that, for the vector fields introduced in \eqref{eq:TZ-def}, we have
\begin{align*}
T& =\partial_{{t}^*} \, ,  &  Z& =\partial_{{\phi}^*} \, .
\end{align*}
The Kerr-star coordinates \eqref{kerr_star_coords} are \emph{global} (modulo the usual degeneration of the angular coordinates $(\bar{\theta},\phi^*)\in\mathbb{S}^2$) coordinates on $\mathcal{M}$, including on $\mathcal{H}^+$. We define the foliation Kerr-star two-spheres $$\mathbb{S}^2_{t^*,\bar{r}} :=\left\lbrace t^*,\bar{r}  \right\rbrace \times \mathbb{S}^2 \, .$$

For any finite $\tau\in\mathbb{R}$, the hypersurface 
\begin{equation} \label{def_S_tau}
\mathcal{S}_{\tau} :=\left\lbrace t^*=\tau \right\rbrace
\end{equation}
is a \emph{spacelike},\footnote{This property can be easily checked and it holds in view of the definition of the function $\iota(\bar{r})$.} asymptotically flat hypersurface intersecting $\mathcal{H}^+$. Future null infinity corresponds to the formal hypersurface $\left\lbrace t^*=\infty \right\rbrace=\mathcal{I}^+ $.

The Kerr-star coordinates \eqref{kerr_star_coords} can be related to \emph{Boyer--Lindquist coordinates} $$(t,r,\theta,\phi)$$ on $\mathcal{M}\setminus\mathcal{H}^+$ by the coordinate transformation
\begin{align*}
t^*&= t+\int_{r_0}^{r}\iota(r^{\prime})\,\left(\frac{{r^{\prime}}{}^2+a^2}{\Delta(r^{\prime})}-1\right)\,dr^{\prime} \, , & \bar{r}&=r \, , \\
\bar{\theta}&=\theta \, , & \phi^*&=\phi+\int_{r_0}^{r}\iota(r^{\prime})\, \frac{a}{\Delta(r^{\prime})}\,dr^{\prime} \quad \textup{mod} \, 2\pi \, .
\end{align*}
In Boyer--Lindquist coordinates, the Kerr metric takes the familiar form
\begin{equation}
g=-\frac{\Delta}{\Sigma}\,(dt-a\sin^2\theta\,d\phi)^2+\frac{\Sigma}{\Delta}\,dr^2+\Sigma \,d\theta^2+\frac{\sin^2\theta}{\Sigma}\,(a\, dt-(r^2+a^2)\,d\phi)^2 \, . \label{kerr_metric_bl}
\end{equation}

\subsubsection{A hyperboloidal differentiable structure} \label{sec_hyperboloidal_coords}

Let $\tilde \zeta\colon [r_+,\infty)\to [0,1]$ is a smooth cut-off function which vanishes for $r\leq 3M$ and is equal to 1 for $r\geq 4M$. We define  \textit{hyperboloidal Kerr-star coordinates}\footnote{With a slight abuse of notation, we do not rename the coordinates $(r,\theta)$.} $$(\tilde{t}{}^*,r,\theta,\tilde\phi{}^*)$$ on $\mathcal{M}$ such that 
\begin{align*}
\tilde{t}{}^*&=t^*-f(r) \,,  &\tilde\phi{}^*&=\phi^*
\end{align*}
for
\begin{align*}
f(r):=\tilde\zeta(r)\lp(r+2M \log r-\frac{3M^2}{r}\rp).
\end{align*}
In hyperboloidal Kerr-star coordinates, we have
\begin{align*}
T&=\p_{\tilde t{}^*}, & Z&=\p_{\tilde\phi{}^*}.
\end{align*}
For future convenience, we define the vector field
\begin{equation}
\tilde X:=\p_{r}|_{(\tilde{t}{}^*,r,\theta,\tilde\phi{}^*)} \, ,  \label{eq:tilde-X}
\end{equation}
for which one can check that $\tilde X=\p_{r}|_{(t,r,\theta,\phi)} + \frac{df}{dr}T$ for $r\geq 9M/4$. We also define the foliation hyperboloidal two-spheres $$\mathbb S^2_{\tilde t{}^*, r}:=\{\tilde t{}^*,r\}\times \mathbb{S}^2 \, .$$ 

For any finite $\tau\in \mathbb R$, the hypersurface
\begin{equation} \label{def_hyperboloidal_hypers}
\Sigma_\tau:=\{\tilde t{}^*=\tau \}
\end{equation}
is a spacelike hypersurface intersecting $\mc H^+$ which becomes \emph{asymptotically null} as it approaches $\mc I^+$.

\begin{figure}[htbp]
\centering
\includegraphics[scale=1.1]{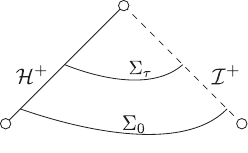}
\caption{$\Sigma_\tau$ are asymptotically null, smooth and foliate $\mc M$.}
\label{fig:hyperboloidal}
\end{figure}

\subsection{The algebraically special horizontal distribution}
\label{sec:algebraically-special-prelims}

\subsubsection{The algebraically special frame}   \label{sec_Kerr_algebr_special_frame}

In Boyer--Lindquist coordinates, we define the \emph{algebraically special} null vector fields
\begin{align}  
e_4^{\text{as}}&=\frac{r^2+a^2}{\Sigma}\,\partial_{t} +\frac{\Delta}{\Sigma}\,\partial_r+\frac{a}{\Sigma}\,\partial_{\phi} \, ,  \label{as_null_frame_vectors_1} \\  
e_3^{\text{as}}&=\frac{r^2+a^2}{\Delta}\,\partial_{t}-\partial_r +\frac{a}{\Delta}\,\partial_{\phi} \, . \label{as_null_frame_vectors_2}
\end{align}
The vector fields $e_4^{\text{as}}$ and $e_3^{\text{as}}$ are regular and non-degenerate vector fields on the entire manifold $\mathcal{M}$, including on $\mathcal{H}^+$. We also define the induced (global, regular) horizontal distribution
\begin{equation*}
\mathfrak{D}_{\mathcal{N}_{\text{as}}}:=\left\langle e^{\text{as}}_3,e^{\text{as}}_4 \right\rangle^{\perp} 
\end{equation*}
on $\mathcal{M}$. One can complete the null vector fields \eqref{as_null_frame_vectors_1}--\eqref{as_null_frame_vectors_2} with a local orthonormal basis
\begin{equation*}
(e^{\text{as}}_1,e^{\text{as}}_2)
\end{equation*}
of $\mathfrak{D}_{\mathcal{N}_{\text{as}}}$ to form the local null frame $$\mathcal{N}_{\text{as}}=(e^{\text{as}}_1,e^{\text{as}}_2,e^{\text{as}}_3,e^{\text{as}}_4) \, ,$$ called the \emph{algebraically special frame} of the Kerr exterior manifold $(\mathcal{M},g)$.

The frame $\mathcal{N}_{\text{as}}$ is \emph{non-integrable} for $|a|>0$, with $\mathfrak{D}_{\mathcal{N}_{\text{as}}}$ a \emph{non-integrable} distribution, see for instance \cite[Definition 4.7]{Benomio2022}.  Along the event horizon $\mathcal{H}^+$, the frame $\mathcal{N}_{\text{as}}$ has some noteworthy properties:
\begin{itemize}[noitemsep]
\item We have
\begin{equation} \label{intro_e4_horizon_killing}
e_4^{\text{as}}|_{\mathcal{H}^+} = 2\,\frac{r_+^2+a^2}{\Sigma(r_+)}\,T +2\,\frac{a}{\Sigma(r_+)}\,\widetilde{Z} \, , 
\end{equation}
i.e.\ the vector field $e_4^{\text{as}}$ is tangent to $\mathcal{H}^+$ and spanned by the Killing vector fields $T$ and $\widetilde{Z}$. 
\item We have
\begin{equation*}
\mathfrak{D}_{\mathcal{N}_{\text{as}}}|_{\mathcal{H}^+}\subset T\mathcal{H}^+ \, ,
\end{equation*}
i.e.\ the arbitrary local frame $(e^{\text{as}}_1,e^{\text{as}}_2)$ of $\mathfrak{D}_{\mathcal{N}_{\text{as}}}$ is tangent to $\mathcal{H}^+$. Thus, the vector fields $$(e^{\text{as}}_4,e^{\text{as}}_1,e^{\text{as}}_2)$$ form a local basis of the \emph{integrable} distribution $T\mathcal{H}^+\subset T\mathcal{M}$.
\end{itemize}

For future convenience, we note a possible explicit choice of $(e^{\text{as}}_1,e^{\text{as}}_2)$ which reads
\begin{align} \label{explicit_e1_e2}
e_1^{\text{as}}&= \frac{a^2\sin\theta\cos\theta}{\Sigma}\,\partial_{t}+\frac{r}{\Sigma}\,\partial_{\theta}+\frac{a\cos\theta}{\Sigma\sin\theta}\,\partial_{\phi} \, , & e_2^{\text{as}}&= \frac{ar\sin \theta}{\Sigma}\,\partial_{t}-\frac{a\cos\theta}{\Sigma}\,\partial_{\theta}+\frac{r}{\Sigma \sin\theta}\,\partial_{\phi}
\end{align}
and we define
\begin{align}
L &=\frac{\Sigma}{r^2+a^2}e_4^{\rm as} \, , & \uL &=\frac{\Delta}{r^2+a^2}e_3^{\rm as} \, . \label{def_L_Lbar}
\end{align}

\subsubsection{Induced metric and connection coefficients} 
\label{sec_Kerr_connection_coeff_curv_comps} 

The induced metric $\slashed{g}$ (together with its inverse $\slashed{g}^{-1}$ and associated standard volume form $\slashed{\varepsilon}$) over $\mathfrak{D}_{\mathcal{N}_{\text{as}}}$ and the connection coefficients of $g$ relative to $\mathcal{N}_{\text{as}}$ are defined as \emph{$\mathfrak{D}_{\mathcal{N}_{\text{as}}}$ tensors} (see \cite[Definition 4.9]{Benomio2022}). The $\mathfrak{D}_{\mathcal{N}_{\text{as}}}$ tensors are defined globally on $\mathcal{M}$ and are regular quantities on the entire manifold $\mathcal{M}$, including on $\mathcal{H}^+$.

In Boyer--Lindquist coordinates, the induced metric reads\footnote{Strictly speaking, the quantities \eqref{kerr_induced_metric}--\eqref{kerr_zeta} below are canonical extensions of the induced metric and connection coefficients (see \cite[Definition 4.17]{Benomio2022}).}
\begin{equation} 
\slashed{g}=  \frac{a^2}{\Sigma}\,\sin^2\theta\, dt^2 -2\,\frac{a\,(r^2+a^2)}{\Sigma}\,\sin^2\theta\, dt\,d\phi +\Sigma\,d\theta^2+\frac{(r^2+a^2)^2}{\Sigma}\,\sin^2\theta\,d\phi^2 \label{kerr_induced_metric}
\end{equation}
and the connection coefficients $\eta$, $\etab$ and $\zeta$ of $g$ relative to $\mathcal{N}_{\text{as}}$ read
\begin{align}
\eta=& \, -\frac{a^2r}{\Sigma^2}\,\sin^2\theta\, dt 
-\frac{a^2\sin(2\theta)}{2\,\Sigma}\,d\theta+\frac{ar(r^2+a^2)}{\Sigma^2}\,\sin^2\theta\,d\phi \, , \label{kerr_eta}\\
\etab=& \, \frac{a^2r}{\Sigma^2}\,\sin^2\theta\, dt -\frac{a^2\sin(2\theta)}{2\,\Sigma}\,d\theta- \frac{ar(r^2+a^2)}{\Sigma^2}\,\sin^2\theta\,d\phi \, , \label{kerr_etab}\\
\zeta=& \, \eta \, . \label{kerr_zeta}
\end{align}
The remaining connection coefficients of $g$ relative to $\mathcal{N}_{\text{as}}$ read
\begin{align*}
{\chih} &=0 \, , & {\chibh} &= 0 \, ,\\
(\slashed{\varepsilon}\cdot\chi) &= \frac{2a \Delta\cos \theta }{\Sigma^2} \, ,   & (\slashed{\varepsilon}\cdot\chib)&= \frac{ 2a \cos \theta}{\Sigma} \, , \\
(\text{tr}\chi) &= \frac{2r \Delta}{\Sigma^2} \, ,  &  (\text{tr}\chib) &= -\frac{2 r }{\Sigma} 
\end{align*}
and
\begin{align*}
\hat{\omega} &= -\frac{ 2\left(a^2   (M-r)\cos^2  \theta +a^2 r- M r^2\right)}{\Sigma^2} \, , &
\omegabh &= 0 \, , \\
\xi &=0 \, , &  \yb &=0 \, .
\end{align*}

The Levi-Civita connection $\nabla$ of $g$ induces a connection $\nablasl$ over $\mathfrak{D}_{\mathcal{N}_{\text{as}}}$ (see \cite[Section 4.4.1]{Benomio2022} for its rigorous definition). The induced connection $\nablasl_{a,M}$ is compatible with the induced metric $\slashed{g}$, however it fails to be torsion-free (and thus Levi-Civita relative to $\slashed{g}$) because the distribution $\mathfrak{D}_{\mathcal{N}_{\text{as}}}$ is non-integrable.

\subsubsection{Tensor products and horizontal differential operators}

In this section, we recall the definition of contractions and products of $\mathfrak{D}_{\mathcal{N}_{\text{as}}}$ tensors and the differential operators acting on $\mathfrak{D}_{\mathcal{N}_{\text{as}}}$ tensors. In our notation, the Latin capital letters are horizontal frame indices.

Given a $\mathfrak{D}_{\mathcal{N}_{\text{as}}}$ one-form $\varsigma$ and $\mathfrak{D}_{\mathcal{N}_{\text{as}}}$ two-tensor $\theta$, we introduce the notation $\varsigma^{\sharp}$, $\theta^{\sharp_1}$ and $\theta^{\sharp_2}$ such that
\begin{align*}
{\varsigma}^{\sharp}{}^{A}&=\slashed{g}^{AB}\varsigma_B \, , & {\theta}^{\sharp_1}{}_{A}^B&=\slashed{g}^{BC}\theta_{CA} \, , & {\theta}^{\sharp_2}{}_{A}^B&=\slashed{g}^{BC}\theta_{AC} \, .
\end{align*}
We define the duality relations
\begin{align*}
{^{\star}\varsigma}_{A}&={\slashed{\varepsilon}^{\sharp_2}}{}^B_A\,\varsigma_B \, , &  {^{\star}\theta}_{AB}&={\slashed{\varepsilon}^{\sharp_2}}{}^C_A\,\theta_{BC} \, .
\end{align*} 
By definition, we have ${{^{\star}}{^{\star}}\varsigma}=-\varsigma$.

Given the $\mathfrak{D}_{\mathcal{N}_{\text{as}}}$ one-forms $\varsigma$ and $\tilde{\varsigma}$, we define the products
\begin{align*}
\varsigma\cdot\tilde{\varsigma}=(\varsigma,\tilde{\varsigma})&=\slashed{g}^{AB}\varsigma_A\,\tilde{\varsigma}_B \, , & \varsigma \wedge \tilde{\varsigma}&=\slashed{\varepsilon}^{AB}\varsigma_A\,\tilde{\varsigma}_B \, .
\end{align*}

The  connection $\nablasl$  can be employed to introduce a set of differential operators acting on $\mathfrak{D}_{\mathcal{N}_{\text{as}}}$ tensors, as follows. For any smooth scalar function $f$, $\mathfrak{D}_{\mathcal{N}_{\text{as}}}$ one-form $\varsigma$ and $\mathfrak{D}_{\mathcal{N}_{\text{as}}}$ two-tensor $\theta$, we define the divergence operator 
\begin{align*}
\slashed{\text{div}}\,\varsigma&=\slashed{g}^{AB}(\nablasl_{A}\varsigma)_{B}  \, , &
(\slashed{\text{div}}\,\theta)_{A}&=\slashed{g}^{CB}(\nablasl_C \theta)_{AB} \, , 
\end{align*}
the curl operator 
\begin{equation*}
\slashed{\text{curl}}\,\varsigma=\slashed{\varepsilon}^{AB}(\nablasl_A \varsigma)_{B} 
\end{equation*}
and the horizontal scalar and covariant Laplacian
\begin{align*}
    \slashed{\Delta}f&=\slashed{\text{div}}\nablasl f \, , & \slashed{\Delta}\varsigma&=\slashed{\text{div}}\nablasl \varsigma \, .
\end{align*}
The identities $\slashed{\Delta}f=-\slashed{\text{curl}}{}^{\star}\nablasl f $ and $\slashed{\Delta}\varsigma=-\slashed{\text{curl}}{}^{\star}\nablasl \varsigma$ hold.

\subsubsection{Notable geometric identities}
\label{sec:algebraically-special-props}

In this section, we collect some geometric identities which will be useful for the sequel.

On $\mathcal{H}^+$, we have 
\begin{equation*}
\nablasl_4\Gamma|_{\mathcal{H}^+}=0 
\end{equation*}
for any connection coefficient $\Gamma$ of $g$ relative to $\mathcal{N}_{\text{as}}$. In particular,
\begin{equation} \label{omegah_extremal}
\omegah|_{\mathcal{H}^+}=0 \quad\quad \text{when $|a|=M$.}
\end{equation}
The identity \eqref{omegah_extremal} encodes the fact that the surface gravity of the extremal Kerr event horizon is zero.

On $\mathcal{M}$, by using the explicit formulas of Section~\ref{sec_Kerr_connection_coeff_curv_comps}, we can compute 
\begin{gather*}
\nablasl\Sigma = \Sigma (\eta+\etab)\,,\qquad \nablasl(r^2-a^2\cos^2\theta) = -\nablasl\Sigma\,, \numberthis\label{eq:ang-derivative-Sigma}\\
 \nablasl_4 \lp(\Sigma^{-1}\nablasl\Sigma\rp) = -\frac12\tr\chi(\eta+\etab)+\frac12\epschi{}^\star(\eta+\etab)\,,\\
\nablasl\tr \chib = -\epsuchi{}^\star (\eta-\etab)\,,\qquad
\nablasl\tr \chi= \nablasl\lp(-\frac{\Delta}{\Sigma}\tr \chib\rp) = \epsuchi{}^\star (\eta-\etab)-\tr\chi(\eta+\etab)\,. 
\end{gather*}
The Codazzi equations \cite[Eqs.\ (96)--(97)]{Benomio2022} yield
\begin{alignat}{3}
\nablasl \epsuchi&={}^\star\nablasl (-\tr \chib) -\epsuchi (2\etab-\eta) +\tr \chib{}^\star\eta &&= \tr \chib{}^\star \eta-\epsuchi\etab \, , \label{formula_nablasl_atrchib}\\
\nablasl \epschi&= {}^\star\nablasl(- \tr \chi) -3\epschi \eta -\tr \chi{}^{\star}\eta &&= \tr \chi {}^\star\etab -\epschi(2\eta+\etab) \, .
\end{alignat}
The transport equations  \cite[Eqs.\ (90), (91) and (93)]{Benomio2022} for $\eta$ and $\etab$ give
\begin{align*}
\nablasl_3(\eta-\etab) &= \frac{1}{2}(\text{tr}\chib)(-3\eta+\etab)-\frac{1}{2}(\slashed{\varepsilon}\cdot\chib)({}^{\star}\eta+{}^{\star}\etab) \, , \\
\nablasl_4 \eta &=-\frac12 \tr\chi (\eta-\etab)+\frac12 \epschi {}^\star(\eta-\etab) \, ,
\end{align*}
from which we obtain by direct computation
\begin{align*}
\nablasl_4 \etab = - \nablasl_4 \eta+\nablasl_4 \lp(\Sigma^{-1}\nablasl\Sigma\rp)=- \tr\chi \etab-\epschi {}^\star\etab = -2\chi^{\sharp_2}\cdot \etab \, .
\end{align*}
The transport equations  \cite[Eqs.\ (84), (85), (88) and (89)]{Benomio2022} for  $\tr\chib$, $\tr\chi$, $\epsuchi$ and $\epschi$ yield
\begin{align*}
2\curlsl \eta&= \nablasl_3\epschi +2\sigma_G \, ,\\
2\divsl \eta&= \nablasl_3\tr\chi +\frac12\tr\chi\tr\chib-\frac12\epschi\epsuchi-2|\eta|^2-2\rho_G  \, ,
\end{align*}
where 
\begin{align*}
\rho_G&= \frac{2Mr}{\Sigma^3}(3a^2\cos^2\theta-r^2), & \sigma_G&= \frac{2Ma\cos\theta}{\Sigma^3}(3r^2-a^2\cos^2\theta) 
\end{align*}
are components of the Weyl tensor of $g$ \cite[Section 5.4]{Benomio2022}, and
\begin{alignat*}{3}
\curlsl (\eta+\etab)&= \frac12\lp[\nablasl_3\epschi+\nablasl_4\epsuchi+\hat\omega \epsuchi\rp] &&=0 \, ,\\
\divsl (\eta-\etab)&= \frac12\lp[\nablasl_3\tr\chi-\nablasl_4\tr\chib-\hat\omega \tr\chib\rp] &&=0 \, , \numberthis \label{eq:div-eta-etab}
\end{alignat*}
where the final equality is obtained by direct computation. Finally, we note the identity
\begin{equation*}
(\eta,\eta)-(\etab,\etab)=0 \, ,
\end{equation*}
which can be checked by direct computation.

For the vector field \eqref{eq:tilde-X}, an easy computation yields
\begin{align}
\tilde X=& \, \frac12 \frac{\Sigma}{\Delta} e_4^{\rm as}-\frac12\lp(1-\frac{\Delta}{\Sigma}\rp)e_3^{\rm as} -\frac{2r}{\Sigma}(\eta-\etab)\cdot \nablasl \label{eq:tilde-X_bis}\\
&+\lp(\frac{df}{dr}-1\rp)\lp[\frac12 \frac{\Delta}{\Sigma}e_3^{\rm as}+\frac12e_4^{\rm as}-\frac{2r}{\Sigma}(\eta-\etab)\cdot \nablasl\rp]  \nonumber 
\end{align}
for $r\geq 9M/4$.

\subsection{Induced structure and norms over foliation spheres}

In this section, we introduce angular operators and norms over the foliation spheres $\mathbb{S}^2_{t^*,\bar{r}}$ and $\mathbb{S}^2_{\tilde t^*,r}$. We take $\bigcirc$ to be a placeholder for $(t^*,\bar{r})$ or $(\tilde t^*,r)$, particularizing to each case only when relevant.

We denote by $(\slashed{g}_{\mathbb{S}^2_{\bigcirc}},\nablasl_{\mathbb{S}^2_{\bigcirc}})$ the metric and associated Levi-Civita connection induced by the Kerr metric and Levi-Civita connection $(g,\nabla)$ over the foliation spheres $\mathbb{S}^2_{\bigcirc}$. We associate the standard volume forms 
\begin{align*}
\slashed{\varepsilon}_{\mathbb{S}^2_{t^*,\bar{r}}}&=\sqrt{\det\slashed{g}_{\mathbb{S}^2_{t^*,\bar{r}}}}d\bar{\theta} \, d\phi^*, & \slashed{\varepsilon}_{\mathbb{S}^2_{\tilde t^*,r}}&=\sqrt{\det\slashed{g}_{\mathbb{S}^2_{\tilde t^*,r}}}d\theta \, d\tilde\phi^*.
\end{align*}
For later convenience, we also define the rescaled volume forms
\begin{align*}
\mathring{\slashed{\varepsilon}}_{\mathbb{S}^2_{\bigcirc}} &:= \frac{4\pi}{\text{Area}(\mathbb{S}^2_{\bigcirc})}\, \slashed{\varepsilon}_{\mathbb{S}^2_{\bigcirc}} \, .
\end{align*}

For any $\mathbb{S}^2_{\bigcirc}$ $k$-tensor $\varsigma$,\footnote{The notion of a $\mathbb{S}^2$ $k$-tensor is here understood in the sense of \cite{Christodoulou1993} (i.e.~as a smooth section of $\otimes_k(T\mathbb{S}^2)^{\star}$).} we define the pointwise norm
\begin{align*}
|\varsigma|_{{\slashed{g}}_{\mathbb{S}^2_{\bigcirc}}}^2 &:= {\slashed{g}}_{\mathbb{S}^2_{\bigcirc}}^{A_1B_1}\cdots{\slashed{g}}_{\mathbb{S}^2_{\bigcirc}}^{A_k B_k}\varsigma_{A_1\cdots A_k}\varsigma_{B_1\cdots B_k} \, ,
\end{align*}
where the contracting indices are frame indices relative to an orthonormal (relative to ${\slashed{g}}_{\mathbb{S}^2_{\bigcirc}}$) frame of $T\mathbb{S}^2_{\bigcirc}$, and the $L^2(\mathbb{S}^2_{\bigcirc})$-norm
\begin{equation}
\left\lVert \varsigma \right\rVert^2_{L^2(\mathbb{S}^2_{\bigcirc})} :=\int_{\mathbb{S}^2_{\bigcirc}} |\varsigma|_{{\slashed{g}}_{\mathbb{S}^2_{\bigcirc}}}^2    \mathring{\slashed\varepsilon}_{\mathbb{S}^2_{\bigcirc}} \, . \label{L2_norm_sphere_foll}
\end{equation}
For any $\mathfrak D_{\mc N_{\rm as}}$ $k$-tensor $\varpi$, we define the pointwise norm
\begin{align*}
|\varpi|_{{\slashed{g}}}^2 &:= {\slashed{g}}^{A_1B_1}\cdots{\slashed{g}}^{A_k B_k}\varpi_{A_1\cdots A_k}\varpi_{B_1\cdots B_k} \, ,
\end{align*}
where the contracting indices are frame indices relative to an orthonormal (relative to $\slashed{g}$) frame of $\mathfrak D_{\mc N_{\rm as}}$, and the $L^2(\mathbb{S}^2_{\bigcirc})$-norm
\begin{equation}
    \left\lVert \varpi \right\rVert^2_{L^2(\mathbb{S}^2_{\bigcirc})} :=\int_{\mathbb{S}^2_{\bigcirc}} |\varpi|_{{\slashed{g}}}^2 \,   \mathring{\slashed\varepsilon}_{\mathbb{S}^2_{\bigcirc}} \, . \label{L2_norm_sphere_nonint}
\end{equation}
We note that the norms \eqref{L2_norm_sphere_foll}--\eqref{L2_norm_sphere_nonint} do not introduce any additional (with respect to the pointwise norm inside the integral) weight in $\bar{r},r$ as $\bar{r},r\rightarrow\infty$. Indeed, if we let  $d\Omega_{\mathbb{S}^2_{\bigcirc}}$ denote the volume form induced by the round metric over the spheres $\mathbb{S}_{\bigcirc}^2$ once the round metric is rescaled by the inverse square of the (area-)radius of the sphere (informally, consider $d\Omega_{\mathbb{S}^2_{\bigcirc}}$ to be the volume form induced by the round metric on the \textit{unit} sphere), then there exist universal constants $C,c>0$ such that
\begin{align}
c \int_{\mathbb{S}^2_{\bigcirc}} |\varsigma|_{\slashed{g}_{\mathbb{S}^2_{\bigcirc}}}^2   d\Omega_{\mathbb{S}^2_{\bigcirc}} \leq \left\lVert \varsigma \right\rVert^2_{L^2(\mathbb{S}^2_{\bigcirc})} &\leq C  \int_{\mathbb{S}^2_{\bigcirc}} |\varsigma|_{\slashed{g}_{\mathbb{S}^2_{\bigcirc}}}^2   d\Omega_{\mathbb{S}^2_{\bigcirc}} \, , \label{equivalence_norm_spheres} \\
c \int_{\mathbb{S}^2_{\bigcirc}} |\varpi|_{\slashed{g}}^2    \, d\Omega_{\mathbb{S}^2_{\bigcirc}} \leq \left\lVert \varpi \right\rVert^2_{L^2(\mathbb{S}^2_{\bigcirc})} &\leq C  \int_{\mathbb{S}^2_{\bigcirc}}|\varpi|_{\slashed{g}}^2  \,  d\Omega_{\mathbb{S}^2_{\bigcirc}} \, . \label{equivalence_norm_spheres_bis}
\end{align}
In particular, we note that 
\begin{align*}
d\Omega_{\mathbb{S}^2_{t^*,\bar{r}}}&= \sin\bar{\theta} \, d\bar{\theta} d\phi^*, & d\Omega_{\mathbb{S}^2_{\tilde{t}{}^*,r}}&= \sin\theta \, d\theta d\tilde{\phi}^*.
\end{align*}
For later convenience, one can check that, at the event horizon, the identity
\begin{equation} \label{id_event_horizon_volume_forms}
    \slashed{\varepsilon}_{\mathbb{S}^2_{t^*,r_+}}=(r_+^2+a^2) \, d\Omega_{\mathbb{S}^2_{t^*,r_+}}
\end{equation}
holds. In what follows, we will often drop the subscripts from the pointwise norms and volume forms whenever doing so does not cause ambiguity.

\subsection{\texorpdfstring{$\mathbb{S}^2$}{Foliation sphere}-projection on the event horizon} \label{sec:projection_formulae}

In the previous section, we introduced the metric and associated Levi-Civita connection $(\slashed g_{\mathbb S^2_{t^*,\bar{r}}},\nablasl_{\mathbb S^2_{t^*,\bar{r}}})$ induced by $(g,\nabla)$ over the foliation spheres $\mathbb{S}^2_{t^*,\bar{r}}$. We now let $\check{\slashed g}$ be the metric induced by $\slashed g$ over the foliation spheres $\mathbb{S}^2_{t^*,\bar{r}}$ via the projection procedure in \cite[Section 7]{Benomio2022}, and denote by $\slashed{\varepsilon}_{\check{\slashed{g}}}$ its associated standard volume form. We remark that
\begin{equation*}
    \check{\slashed g}\neq\slashed g_{\mathbb S^2_{t^*,\bar{r}}}
\end{equation*}
on $\mathcal{M}$. We denote the Levi-Civita connection of $\check{\slashed g}$ by $\check{\slashed\nabla}$. We quote from \cite[Section 7]{Benomio2022} the properties of $(\check{\slashed g},\check{\slashed\nabla})$ which will be exploited in our analysis:
\begin{proposition}[$\mathbb{S}^2$-projection formulae]\label{prop:proj-formulas} For any $\mathfrak D_{\mc N_{\rm as}}$ one-form $\varsigma$, let $\tilde\varsigma$ be the associated $\mathbb{S}^2_{t^*,\bar{r}}$ one-form such that
\begin{align*}
\tilde\varsigma(e_A^{\textup{ad}})&=g(e_A^{\textup{ad}},e_B^{\textup{as}})\varsigma^B \, ,  &  A&=\left\lbrace 1,2 \right\rbrace \, ,
\end{align*}
with $(e_1^{\textup{ad}},e_2^{\textup{ad}})$ an arbitrary local frame of $\mathbb{S}^2_{t^*,\bar{r}}$. Then, $\tilde\varsigma$ satisfies the property 
\begin{equation}
|\tilde\varsigma|_{\check{\slashed g}} = |\varsigma|_{\slashed g} \, . \label{equality_proj_norms} 
\end{equation}
Furthermore, \underline{along $\mc H^+$}, we have 
\begin{equation} \label{equality_proj_metrics_horizon}
    \check{\slashed g}=\slashed g_{\mathbb S^2_{t^*,r_+}} 
\end{equation}
and the projection formulae 
\begin{align*}
\widetilde{\nablasl_4 f} &=\check{\nablasl}_4f\,, & \widetilde{\nablasl_3 f} &=\check{\nablasl}_3f\,, & \widetilde{\nablasl f} &=\check{\nablasl}f+\frac12 (\check{\nablasl}_4f)\mathfrak k
\end{align*}
for any smooth scalar function $f$ and 
\begin{align*}
\widetilde{\nablasl \varsigma} &=  \check{\nablasl}\, \widetilde{\varsigma}+\frac{1}{2}\, \mathfrak{k}\otimes \check{\nablasl}_4\,\widetilde{\varsigma}\implies 
\begin{dcases}
\widetilde{\slashed{\textup{div}}\, \varsigma} &= \check{\slashed{\textup{div}}}\, \widetilde{\varsigma} +\frac{1}{2}\, \left(\mathfrak{k}, \check{\nablasl}_4\,\widetilde{\varsigma}\right)_{\check{\slashed{g}}} \, , \\
\widetilde{\slashed{\textup{curl}}\, \varsigma} &= \check{\slashed{\textup{curl}}}\, \widetilde{\varsigma} +\frac{1}{2}\, \mathfrak{k}\wedge_{\check{\slashed{g}}} \check{\nablasl}_4\,\widetilde{\varsigma} \,,
\end{dcases}
\end{align*}
for any $\mathfrak D_{\mc N_{\rm as}}$ one-form $\varsigma$. Here, $\mathfrak{k}$ is a $\mathbb{S}^2_{t^*,\bar{r}}$ one-form such that
\begin{align*}
\mathfrak{k}(e_A^{\rm ad})&= g(e_3^{\rm as},e_A^{\rm ad})\, ,   &  A&=\left\lbrace 1,2 \right\rbrace \, ,
\end{align*}
with $(e_1^{\textup{ad}},e_2^{\textup{ad}})$ an arbitrary local frame of $\mathbb{S}^2_{t^*,\bar{r}}$. It can be checked that $\nablasl_4\mathfrak k|_{\mc H^+}=0$. One has $\mathfrak{k}=0$ when $|a|=0$.
\end{proposition}

We remark that, in view of the identity \eqref{equality_proj_metrics_horizon}, one can replace the connection $\check{\nablasl}$ with $\nablasl_{\mathbb S^2_{t^*,\bar{r}}}$ (and similarly for the associated angular operators) and the tensor products relative to $\check{\slashed{g}}$ with tensor products relative to $\slashed g_{\mathbb S^2_{t^*,\bar{r}}}$ in the horizon projection formulae stated in Proposition \ref{prop:proj-formulas}. Combining the identities \eqref{id_event_horizon_volume_forms} and \eqref{equality_proj_metrics_horizon}, it also follows that
\begin{equation} \label{identity_event_horizon_volume_forms_bis}
    \slashed{\varepsilon}_{\check{\slashed{g}}}|_{\mathcal{H}^+}=(r_+^2+a^2) \, d\Omega_{\mathbb{S}^2_{t^*,r_+}} \, .
\end{equation}

More complicated projection formulae hold off of $\mc H^+$, see e.g.\ \cite{Benomio2022,BTdC2025}. On  $\mc H^+$, the properties of $\mc N_{\rm as}|_{\mc H^+}$ noted in Section~\ref{sec_Kerr_algebr_special_frame} lead to the simpler formulae stated in Proposition \ref{prop:proj-formulas}.

%-------------------------------------------------------------------
%
% THE SYSTEM OF MAXWELL EQUATIONS
%
%-------------------------------------------------------------------
%\newpage

\section{The Maxwell equations} \label{sec:maxwell_eqns_head}

In this section, we first derive the Maxwell equations on a general spacetime. We then specialise the system of equations to the Kerr exterior manifold and discuss special classes of solutions, the notion of initial data and its well-posedness.

%-------------------------------------------------------------------
% Full system: equations, initial data, well-posedness
%-------------------------------------------------------------------
\subsection{The equations on a general spacetime} \label{sec:maxwell_general}

In this section, we derive the (source-free) Maxwell equations for the Maxwell field components relative to a general null frame of a general Lorentzian manifold. In particular, the frame is not assumed to be integrable. To derive the equations, we apply the formalism for non-integrable horizontal distributions from \cite[Section 4]{Benomio2022}. Here, we adopt the boldface notation to denote both the unknowns and the background quantities in the system of equations.

Let $(\bsy{\mathcal{M}},\bsy{g})$ be a $(3+1)$-dimensional, smooth, time-oriented Lorentzian manifold endowed with a local null frame $\bsy{\mathcal{N}}=(\eo,\etw,\et,\ef)$, with associated horizontal distribution $\bsy{\mathfrak{D}_{\mathcal{N}}}$. We denote by $\bsy{\nabla}$ the Levi-Civita connection of $\bsy{g}$ on $\bsy{\mathcal{M}}$. Let $\bsy{F}\in\Lambda^2(\bsy{\mathcal{M}})$. Consider the $\bsy{\mathfrak{D}_{\mathcal{N}}}$ one-forms $\bsy{\alpha},\bsy{\alphab}$ and the smooth scalar functions $\bsy{\rho},\bsy{\sigma}$ on $\bsy{\mathcal{M}}$ defined by
\begin{align}
\bsy{\alpha}(\bsy{e_A}) &:= \bsy{F}(\bsy{e_A},\bsy{e_4}) ,\label{eq:def-alpha}\\
\bsy{\alphab}(\bsy{e_A}) &:= \bsy{F}(\bsy{e_A},\bsy{e_3}) , \label{eq:def-alphab}\\
\bsy{\rho} &:= \frac{1}{2}\,\bsy{F}(\bsy{e_3},\bsy{e_4}) , \label{eq:def-rho}\\
\bsy{\sigma} &:=\frac{1}{2}\,\bsy{\slashed{\varepsilon}^{AB}}\bsy{F}(\bsy{e_A},\bsy{e_B})    \label{eq:def-sigma}
\end{align}
with $\bsy{A}=\left\lbrace \bsy{1},\bsy{2}\right\rbrace$, where $\bsy{\slashed{\varepsilon}}$ is the standard volume form associated to the metric $\bsy{\slashed{g}}$ induced by $\bsy{g}$ over $\bsy{\mathfrak{D}_{\mathcal{N}}}$. We note that, from \eqref{eq:def-sigma}, we have $\bsy{F}_{\bsy{AB}}=\bsy{\sigma}\bsy{\slashed{\varepsilon}_{AB}}$. We will refer to $\bsy{\alpha}, \bsy{\alphab}$ as the \emph{extremal} Maxwell components and to $\bsy{\rho},\bsy{\sigma}$ as the \emph{middle} Maxwell components.

Making use of the formulae in \cite[Section 4.6]{Benomio2022}, we compute the following identities\footnote{In our notation, the Latin capital letters are frame indices, with the first letters of the alphabet running over the horizontal indices only. Note also that one needs to keep track of the order of the indices of the second fundamental forms, which are not in general symmetric tensors.}
\begin{align*}
(\bsy{\nabla_A}\bsy{F})_{\bsy{B4}} &= (\bsy{\nablasl_A}\bsy{\alpha})_{\bsy{B}}+\bsy{\zeta_A}\bsy{\alpha_B}-\bsy{\chi_{AB}}\bsy{\rho}-\bsy{{}^{\star}\chi_{BA}}\bsy{\sigma}\, ,\\
(\bsy{\nabla_A}\bsy{F})_{\bsy{B3}} &= (\bsy{\nablasl_A}\bsy{\alphab})_{\bsy{B}}-\bsy{\zeta_A}\bsy{\alphab_B}+\bsy{\chib_{AB}}\bsy{\rho}-\bsy{{}^{\star}\chib_{BA}}\bsy{\sigma} \, , \\
(\bsy{\nabla_4}\bsy{F})_{\bsy{34}} &= 2\bsy{\nablasl_4}\bsy{\rho}-2\, \bsy{\etab^A}\bsy{\alpha_A}+2\, \bsy{\xi^A}\bsy{\alphab_A} \, , \\
(\bsy{\nabla_3}\bsy{F})_{\bsy{43}} &= -2\bsy{\nablasl_3}\bsy{\rho}-2\, \bsy{\eta^A}\bsy{\alphab_A}+2\, \bsy{\yb^A}\bsy{\alpha_A} \, ,\\
(\bsy{\nabla_4}\bsy{F})_{\bsy{AB}} &= (\bsy{\nablasl_4}\bsy{\sigma})\bsy{\slashed{\varepsilon}_{AB}}+\bsy{\etab_A}\bsy{\alpha_B}-\bsy{\etab_B}\bsy{\alpha_A} +\bsy{\xi_A}\bsy{\alphab_B} -\bsy{\xi_B}\bsy{\alphab_A} \, , \\
(\bsy{\nabla_3}\bsy{F})_{\bsy{AB}} &= (\bsy{\nablasl_3}\bsy{\sigma})\bsy{\slashed{\varepsilon}_{AB}}+\bsy{\eta_A}\bsy{\alphab_B}-\bsy{\eta_B}\bsy{\alphab_A}+\bsy{\yb_A}\bsy{\alpha_B} -\bsy{\yb_B}\bsy{\alpha_A} \, , \\
(\bsy{\nabla_A}\bsy{F})_{\bsy{BC}} &= (\bsy{\nablasl_A}\bsy{\sigma})\bsy{\slashed{\varepsilon}_{BC}} +\frac{1}{2}\,\bsy{\chi_{AB}}\bsy{\alphab_C}+\frac{1}{2}\,\bsy{\chib_{AB}}\bsy{\alpha_C}-\frac{1}{2}\,\bsy{\chi_{AC}}\bsy{\alphab_B}-\frac{1}{2}\,\bsy{\chib_{AC}}\bsy{\alpha_B} \, , \\
(\bsy{\nabla_4}\bsy{F})_{\bsy{3A}} &= -(\bsy{\nablasl_4}\bsy{\alphab})_{\bsy{A}}-\hat{\bsy{\omega}}\bsy{\alphab}_{\bsy{A}} -2\, \bsy{\etab_A}\bsy{\rho}+2\, {}^{\bsy{\star}}\bsy{\etab_A}\bsy{\sigma} \, , \\
(\bsy{\nabla_3}\bsy{F})_{\bsy{A4}} &= (\bsy{\nablasl_3}\bsy{\alpha})_{\bsy{A}}+\underline{\hat{\bsy{\omega}}}\bsy{\alpha}_{\bsy{A}}-2\, \bsy{\eta_A}\bsy{\rho}-2 \, \bsy{{}^{\star}\eta_A} \bsy{\sigma} \, , \\
(\bsy{\nabla_A}\bsy{F})_{\bsy{34}} &= 2\bsy{\nablasl_A}\bsy{\rho}-\bsy{\chib^{\sharp_2}{}_A^B}\bsy{\alpha_B}+\bsy{\chi^{\sharp_2}{}_A^B}\bsy{\alphab_B} \, , 
\end{align*}
where ${}^{\bsy{\star}}$ denotes the horizontal Hodge dual as defined in \cite[Section 4.3]{Benomio2022}. We recall that $\bsy{\nablasl}$ is defined as the connection induced by $\bsy{\nabla}$ over $\bsy{\mathfrak{D}_{\mathcal{N}}}$. In general, $\bsy{\nablasl}$ is not the Levi-Civita connection of $\bsy{\slashed{g}}$. Nonetheless, one has $\bsy{\nablasl_X\slashed{g}}=\bsy{\nablasl_X\slashed{\varepsilon}}=0$ for any $\bsy{X}\in \Gamma(T\bsy{\mathcal{M}})$. These metric compatibility properties have been used repeatedly in the computations above.

We now assume that $\bsy{F}$ satisfies the (source-free) Maxwell equations 
\begin{align}
(\bsy{\nabla_{I}}\bsy{F})_{\bsy{J K}}+(\bsy{\nabla_{J}}\bsy{F})_{\bsy{K I}}+(\bsy{\nabla_{K}}\bsy{F})_{\bsy{I J}} &=0 \, , & \bsy{g^{JK}}(\bsy{\nabla_J}\bsy{F})_{\bsy{K I}} &=0  \label{eq:Maxwell}
\end{align}
on $\bsy{\mathcal{M}}$, with $\bsy{I},\bsy{J},\bsy{K}\in\left\lbrace \bsy{1},\bsy{2},\bsy{3},\bsy{4}\right\rbrace$. We can expand the second set of equations in \eqref{eq:Maxwell} as 
\begin{equation}
\bsy{\slashed{g}^{AB}}(\bsy{\nabla_A}\bsy{F})_{\bsy{BI}}-\frac{1}{2}\,(\bsy{\nabla_3}\bsy{F})_{\bsy{4I}}-\frac{1}{2}\,(\bsy{\nabla_4}\bsy{F})_{\bsy{3I}}=0 \, . \label{eq:Maxwell-2}
\end{equation}
Taking $\bsy{I}\in\left\lbrace \bsy{3},\bsy{4} \right\rbrace$ in \eqref{eq:Maxwell-2}, we obtain
\begin{align*}
0&=\bsy{\slashed{g}^{AB}}(\bsy{\nabla_A}\bsy{F})_{\bsy{B4}}-\frac{1}{2}\,(\bsy{\nabla_4}\bsy{F})_{\bsy{34}} \, ,  \\
0&=\bsy{\slashed{g}^{AB}}(\bsy{\nabla_A}\bsy{F})_{\bsy{B3}}-\frac{1}{2}\,(\bsy{\nabla_3}\bsy{F})_{\bsy{43}} \, ,
\end{align*}
which, noting that $(\bsy{\textbf{tr}{}^{\star}\chi})=(\bsy{\slashed{\varepsilon}\cdot\chi})$ and $(\textbf{tr}\bsy{{}^{\star}\chib})=(\bsy{\slashed{\varepsilon}\cdot\chib})$, can be written as 
\begin{align} 
\bsy{\nablasl_4}\bsy{\rho}+(\bsy{\textbf{tr}\chi})\bsy{\rho} &=\slashed{\textbf{div}}\,\bsy{\alpha}+(\bsy{\zeta}+\bsy{\etab},\bsy{\alpha})-(\bsy{\xi},\bsy{\alphab})-(\bsy{\slashed{\varepsilon}\cdot\chi})\bsy{\sigma} \, , \label{4_rho}
\\
\bsy{\nablasl_3}\bsy{\rho}+(\textbf{tr}\bsy{\chib})\bsy{\rho} &=-\slashed{\textbf{div}}\,\bsy{\alphab}+(\bsy{\zeta}-\bsy{\eta},\bsy{\alphab})+(\bsy{\xib},\bsy{\alpha})+(\bsy{\slashed{\varepsilon}\cdot\chib})\bsy{\sigma} \, . \label{3_rho}
\end{align}
We now consider the first set of equations in \eqref{eq:Maxwell}. Taking $\bsy{I}=\bsy{A}$, $\bsy{J}=\bsy{B}$ and $\bsy{K}\in \left\lbrace \bsy{3},\bsy{4} \right\rbrace$, we obtain
\begin{align*}
0&=(\bsy{\nabla_4}\bsy{F})_{\bsy{AB}}+(\bsy{\nabla_A}\bsy{F})_{\bsy{B4}}-(\bsy{\nabla_B}\bsy{F})_{\bsy{A4}} \,,\\
0&=(\bsy{\nabla_3}\bsy{F})_{\bsy{AB}}+(\bsy{\nabla_A}\bsy{F})_{\bsy{B3}}-(\bsy{\nabla_B}\bsy{F})_{\bsy{A3}}\,, 
\end{align*}
which, noting that $(\bsy{\slashed{\varepsilon}\cdot{}^{\star}\chi})=(\textbf{tr}\bsy{\chi})$ and $(\bsy{\slashed{\varepsilon}\cdot{}^{\star}\chib})=(\textbf{tr}\bsy{\chib})$, can be rewritten as
\begin{align} 
\bsy{\nablasl_4}\bsy{\sigma}+(\textbf{tr}\bsy{\chi})\bsy{\sigma} &=-\bsy{\slashed{\textbf{curl}}}\,\bsy{\alpha}-(\bsy{\zeta}+\bsy{\etab})\bsy{\wedge}\bsy{\alpha}-\bsy{\xi} \bsy{\wedge} \bsy{\alphab}+(\bsy{\slashed{\varepsilon}\cdot\chi})\bsy{\rho} \, , \label{4_sigma} \\
\bsy{\nablasl_3}\bsy{\sigma}+(\textbf{tr}\bsy{\chib})\bsy{\sigma} &=-\bsy{\slashed{\textbf{curl}}}\,\bsy{\alphab}+(\bsy{\zeta}-\bsy{\eta})\bsy{\wedge}\bsy{\alphab}-\bsy{\yb} \bsy{\wedge} \bsy{\alpha}-(\bsy{\slashed{\varepsilon}\cdot\chib})\bsy{\rho} \, . \label{3_sigma}
\end{align}
Finally, we consider the first set of equations in \eqref{eq:Maxwell} with $\bsy{I}=\bsy{4}$, $\bsy{J}= \bsy{3}$ and $\bsy{K}=\bsy{A}$ and the second set of equations in \eqref{eq:Maxwell} with $\bsy{I}=\bsy{A}$ to obtain
\begin{align*}
0&=(\bsy{\nabla_4}\bsy{F})_{\bsy{3A}}-(\bsy{\nabla_A}\bsy{F})_{\bsy{34}}+(\bsy{\nabla_3}\bsy{F})_{\bsy{A4}} \, , \\
0&=\bsy{\slashed{g}^{BC}}(\bsy{\nabla_B}\bsy{F})_{\bsy{CA}}+\frac{1}{2}\,(\bsy{\nabla_3}\bsy{F})_{\bsy{A4}}-\frac{1}{2}\,(\bsy{\nabla_4}\bsy{F})_{\bsy{3A}} \, .
\end{align*}
By summing these two equations, and repeating the procedure for the conjugate equations under the exchange of the frame indices $\bsy{3}$ and $\bsy{4}$, we obtain
\begin{align*}
0 &=\bsy{\slashed{g}^{BC}}(\bsy{\nabla_B}\bsy{F})_{\bsy{CA}}-(\bsy{\nabla_4}\bsy{F})_{\bsy{3A}}+\frac{1}{2}\,(\bsy{\nabla_A}\bsy{F})_{\bsy{34}} \, ,\\
0 &=\bsy{\slashed{g}^{BC}}(\bsy{\nabla_B}\bsy{F})_{\bsy{CA}}+(\bsy{\nabla_3}\bsy{F})_{\bsy{A4}}-\frac{1}{2}\,(\bsy{\nabla_A}\bsy{F})_{\bsy{34}} \, ,
\end{align*}
which can be written as
\begin{align}
\bsy{\nablasl_4}\bsy{\alphab}+\bsy{\chi^{\sharp_2}\cdot}\bsy{\alphab}-\hat{\bsy{\chi}}{}^{\bsy{\sharp}} \bsy{\cdot} \bsy{\alphab}+\hat{\bsy{\omega}}\,\bsy{\alphab} &= \hat{\bsy{\chib}}{}^{\bsy{\sharp}}\bsy{\cdot} \bsy{\alpha}-\bsy{\nablasl}\bsy{\rho} +{}^{\bsy{\star}}\bsy{\nablasl}\bsy{\sigma} -2\,\bsy{\etab}\,\bsy{\rho}+2\, {}^{\bsy{\star}}\bsy{\etab}\,\bsy{\sigma} \, , \label{4_alphab}\\
\bsy{\nablasl_3}\bsy{\alpha}+\bsy{\chib}^{\bsy{\sharp_2}}\bsy{\cdot}\bsy{\alpha}-\hat{\bsy{\chib}}{}^{\bsy{\sharp}} \bsy{\cdot} \bsy{\alpha}+\underline{\hat{\bsy{\omega}}}\,\bsy{\alpha} &= \hat{\bsy{\chi}}{}^{\bsy{\sharp}}\bsy{\cdot} \bsy{\alphab} +\bsy{\nablasl}\bsy{\rho}+{}^{\bsy{\star}}\bsy{\nablasl}\bsy{\sigma} +2\,\bsy{\eta}\,\bsy{\rho}+2\,{}^{\bsy{\star}}\bsy{\eta}\,\bsy{\sigma} \, . \label{3_alpha}
\end{align}
The equations \eqref{4_rho}--\eqref{3_alpha} constitute the system of Maxwell equations on $(\bsy{\mathcal{M}},\bsy{g})$ relative to $\bsy{\mathcal{N}}$.

\subsection{The equations on the Kerr spacetime relative to the algebraically special frame}
\label{sec:Maxwell-Kerr}

In this section, we specialise the general system of Maxwell equations \eqref{4_rho}--\eqref{3_alpha} derived in Section \ref{sec:maxwell_general} by identifying $(\bsy{\mathcal{M}},\bsy{g})$ with the Kerr exterior manifold $(\mathcal{M},g)$ and $\bsy{\mathcal{N}}$ with the algebraically special frame $\mathcal{N}_{\text{as}}$ of $(\mathcal{M},g)$ (see Section~\ref{sec:Kerr-prelims}). In the notation, the unknowns in the equations remain boldfaced, whereas the background quantities are now unbolded.

The equations \eqref{4_rho}--\eqref{3_alpha} take the form
\begin{align}
\nablasl_3 \bsy\sigma +(\truchi)\bsy\sigma &= -\curlsl\bsy\ualpha-\epsuchi \bsy\rho\,, \label{eq:del3-sigma}\\
\nablasl_3 \bsy\rho +(\truchi) \bsy\rho &= -\divsl\bsy\ualpha+\epsuchi \bsy\sigma\,, \label{eq:del3-rho}
\end{align}
\begin{align}
\nablasl_4 \bsy\sigma +(\trchi)\bsy\sigma &= -\curlsl\bsy\alpha-(\eta+\ueta)\wedge \bsy\alpha+\epschi \bsy\rho\,, \label{eq:del4-sigma}\\
\nablasl_4 \bsy\rho +(\trchi) \bsy\rho &= \divsl\bsy\alpha+(\eta+\ueta,\bsy\alpha)-\epschi \bsy\sigma\,, \label{eq:del4-rho}
\end{align}
\begin{align}
\nablasl_3 \bsy\alpha +\uchi^{\sharp_2}\cdot \bsy\alpha &= \nablasl\bsy\rho +{}^\star \nablasl \bsy\sigma +2\, \eta \, \bsy\rho+2\,{}^\star\eta \, \bsy\sigma\,, \label{eq:del3-alpha}\\
\nablasl_4 \bsy\ualpha +\chi^{\sharp_2}\cdot \bsy\ualpha +\omegah \, \bsy\ualpha  &= -\nablasl\bsy\rho +{}^\star \nablasl \bsy\sigma -2\,\ueta \, \bsy\rho+2 \, {}^\star\ueta \, \bsy\sigma\,. \label{eq:del4-ualpha}
\end{align}
The equations \eqref{eq:del3-sigma}--\eqref{eq:del4-ualpha} constitute the system of Maxwell equations on $(\mathcal{M},g)$ relative to $\mathcal{N}_{\text{as}}$. For future convenience, we note that, in view of the identities of Section~\ref{sec:algebraically-special-props}, the system can be rewritten in the form 
\begin{align}
\nablasl_3(\Sigma\bsy{\sigma})&=-\curlsl(\Sigma\bsy{\alphab})+(\eta+\etab)\wedge(\Sigma\bsy{\alphab})-(\slashed{\varepsilon}\cdot\chib)\Sigma\bsy{\rho} \, , \label{eq:del3-sigma'}\\
\nablasl_3(\Sigma\bsy{\rho})&=-\divsl(\Sigma\bsy{\alphab})+(\eta+\etab,\Sigma\bsy{\alphab})+(\slashed{\varepsilon}\cdot\chib)\Sigma\bsy{\sigma} \, , \label{eq:del3-rho'}
\end{align}
\begin{align}
\nablasl_4 (\Sigma\bsy\sigma) &= -\curlsl(\Sigma\bsy\alpha)+\epschi \Sigma \bsy\rho \, , \label{eq:del4-sigma'}\\
\nablasl_4 (\Sigma \bsy\rho) &= \divsl(\Sigma\bsy\alpha)-\epschi \Sigma \bsy\sigma \, , \label{eq:del4-rho'}
\end{align}
\begin{align}
\Sigma(\nablasl_3\bsy{\alpha}+\chib^{\sharp_2}\cdot\,\bsy{\alpha}) &= \nablasl(\Sigma\bsy{\rho})+{}^{\star}\nablasl(\Sigma\bsy{\sigma})+(\eta-\etab)\Sigma\bsy{\rho}+{}^{\star}(\eta-\etab)\Sigma\bsy{\sigma} \, , \label{eq:del3-alpha'} \\
\Sigma(\nablasl_4\bsy{\alphab}+\chi^{\sharp_2}\cdot\,\bsy{\alphab}+\omegah \, \bsy\alphab) &= -\nablasl(\Sigma\bsy{\rho})+{}^{\star}\nablasl(\Sigma\bsy{\sigma})+(\eta-\etab)\Sigma\bsy{\rho}-{}^{\star}(\eta-\etab)\Sigma\bsy{\sigma} \, . \label{eq:del4-ualpha'} 
\end{align}

\subsubsection{The Teukolsky equations}
\label{sec:Teukolsky-equations}

From the Maxwell equations \eqref{eq:del3-sigma}--\eqref{eq:del4-ualpha}, one can show that the extremal Maxwell components $\bsy \alpha$ and $\bsy \alphab$ each satisfy a decoupled wave equation:
\begin{proposition}[Tensorial Teukolsky equations] \label{prop:Teukolsky-tensor} We have the identities
\begin{align}
\begin{split}
&\Sigma \lp(\nablasl_4+\frac32 \tr \chi -\frac12 \epschi{}^\star \rp)\!\lp(\nablasl_3+\frac12 \tr \chib +\frac12 \epsuchi{}^\star \rp)\bsy\alpha\\
&\qquad-  (\nablasl+2\eta)\divsl(\Sigma\bsy\alpha)+{}^\star (\nablasl+2\eta)\curlsl (\Sigma\bsy\alpha)=0\,,
\end{split}\label{eq:Teukolsky-plus}\\
\begin{split}
&\lp(\nablasl_3+\frac32 \tr \chib -\frac12 \epsuchi{}^\star \rp)\!\lp(\nablasl_4+\frac12 \tr \chi +\frac12 \epschi{}^\star +\hat\omega\rp)\bsy\alphab\\
&\qquad-  (\nablasl+2\etab)\divsl\bsy\alphab +{}^\star (\nablasl+2\etab)\curlsl \bsy\alphab =0\,.
\end{split}\label{eq:Teukolsky-minus}
\end{align}
\end{proposition}
Equations~\eqref{eq:Teukolsky-plus}--\eqref{eq:Teukolsky-minus} are henceforth known as the (tensorial) Teukolsky equations. We direct the reader to Appendix~\ref{app:Teukolsky-equations} for a detailed derivation.\footnote{As in this paper we are not concerned with an in-depth analysis of the Teukolsky equations, either in tensorial or spinorial form, this derivation is included only as a reference for the reader.} We remark that the decoupling of the extremal Maxwell components is intimately tied to the fact that the frame components of the Maxwell field (i.e.~the unknowns of the Maxwell system) are defined relative to the algebraically special frame.

In the above,  $\bsy\alpha$ and $\bsy\alphab$ are one-forms over the horizontal distribution $\mathfrak D_{\mc N_{\rm as}}$, hence \eqref{eq:Teukolsky-plus}--\eqref{eq:Teukolsky-minus} are equations for smooth sections $\Gamma((\mathfrak D_{\mc N_{\rm as}})^\star)$. In contrast, the original derivation of the Teukolsky equations, due to Teukolsky \cite{Teukolsky1973}, in the Newman--Penrose formalism \cite{Newman1962} yielded equations
\begin{align}
\begin{split}
\bigg[\Box_{g} &+\frac{2(\pm 1)}{\Sigma^2}(r-M)\p_r +\frac{2(\pm 1)}{\Sigma^2}\lp(\frac{a(r-M)}{\Delta}+i\frac{\cos\theta}{\sin^2\theta}\rp)\p_\phi\\
&+\frac{2(\pm 1)}{\Sigma^2}\lp(\frac{M(r^2-a^2)}{\Delta}-r -ia\cos\theta\rp)\p_t +\frac{1}{\Sigma^2}\lp(\pm 1-\cot^2\theta\rp)\bigg]\upalpha^{[\pm 1]}  =0
\end{split}\label{eq:Teukolsky-spin}
\end{align} 
for complex-valued functions $\swei{\upalpha}{\pm 1}$ such that $\Delta^{\pm 1}\swei{\upalpha}{\pm 1}$ are  smooth, $(\pm 1)-$spin-weighted functions on $\mc M$.  We denote this function space by $\mathscr S^{[\pm 1]}_\infty(\mc M)$, and direct the reader to e.g.~\cite[Section 3.1.1]{SRTdC2023} for a precise definition. In \eqref{eq:Teukolsky-spin}, $\Box_{g}$ denotes the covariant wave operator (i.e.\ the Laplace--Beltrami operator) for scalar functions on the Kerr manifold $(\mc M,g)$.

Let $$\kappa(r,\theta):=r-ia\cos\theta$$ be a smooth complex-valued function on $\mathcal{M}$. The tensorial and spin-weighted quantities can be suitably identified:

\begin{lemma} \label{lemma:isomorphism} The maps $\mathrm{i}^{[\pm 1]}\colon \Gamma((\mathfrak D_{\mc N_{\rm as}})^{\star})\to \mathscr S^{[\pm 1]}_\infty(\mc M)$ given by
\begin{align}
\begin{split}
\bsy\alpha &\xmapsto{\mathrm{i}^{[+1]}}\Delta\swei{\upalpha}{+1}:=-\frac{\Sigma}{\sqrt{2}}\lp[\bsy\alpha(e_1^{\rm as})+i\bsy\alpha(e_2^{\rm as})\rp], \\ \bsy\alphab &\xmapsto{\mathrm{i}^{[-1]}} \Delta^{-1}\swei{\upalpha}{-1}:= \frac{1}{\sqrt{2}\Sigma}\kappa^2\lp[\bsy\alphab(e_1^{\rm as})-i\bsy\alphab(e_2^{\rm as})\rp],
\end{split} \label{eq:map-B-to-NP}
\end{align}
where $(e_1^{\rm as},e_2^{\rm as})$ are chosen to be given by the explicit formula \eqref{explicit_e1_e2}, define an isomorphism between $\Gamma((\mathfrak D_{\mc N_{\rm as}})^{\star})$ and $\mathscr S^{[\pm 1]}_\infty(\mc M)$ with the properties 
\begin{align*}
|\Delta\swei{\upalpha}{+1}|^2 &=\frac12 \lvert\Sigma \bsy\alpha\rvert_{\slashed g}^2, &
|(r^2+a^2)^{1/2}\Delta^{-1}\swei{\upalpha}{-1}|^2 
%= \frac12|\kappa^{-2}\Sigma|^2 \lvert (r^2+a^2)^{1/2}\bsy\alphab\rvert_{\slashed g}^2 
&= \frac12 \lvert (r^2+a^2)^{1/2}\bsy\alphab\rvert_{\slashed g}^2.
\end{align*}
The isomorphisms $\mathrm{i}^{[\pm 1]}$ descend to isomorphisms between the space of solutions to the tensorial Teukolsky equations \eqref{eq:Teukolsky-plus}--\eqref{eq:Teukolsky-minus} and, respectively, the space of solutions to the $(\pm 1)$-spin-weighted Teukolsky equations \eqref{eq:Teukolsky-spin}.
\end{lemma}
\begin{proof} It is not hard to check that, for $\bsy{\alpha}, \bsy{\alphab}\in \Gamma((\mathfrak D_{\mc N_{\rm as}})^{\star})$, the expressions
\begin{align*}
&\bsy\alpha(e_1^{\rm as})+i\bsy\alpha(e_2^{\rm as}), & &\bsy\alphab(e_1^{\rm as})-i\bsy\alphab(e_2^{\rm as})
\end{align*}
with $(e_1^{\rm as},e_2^{\rm as})$ given in \eqref{explicit_e1_e2} indeed define, respectively, smooth $(\pm 1)$-spin-weighted functions on $\mc M$ as defined in e.g.\ \cite[Section 3.1.1]{SRTdC2023}. Furthermore, the maps \eqref{eq:map-B-to-NP} are clearly invertible: 
\begin{align*}
\swei{\upalpha}{+ 1}\equiv 0 \Leftrightarrow 
\lp\{\begin{array}{l}
\Re\swei{\upalpha}{+1} \equiv 0 \equiv\bsy\alpha(e_1^{\rm as})\\  \Im\swei{\upalpha}{+1}\equiv 0\equiv \bsy\alpha(e_2^{\rm as})
\end{array}\rp\}
 \Leftrightarrow \bsy\alpha\equiv 0,
\end{align*}
and similarly for the spin $-1$ and $\bsy\alphab$ case. 
\end{proof}

\subsubsection{The modified Maxwell equations}
\label{sec:reduced-Maxwell-system}

We introduce the modified quantities
\begin{align}
\widehat{\bsy{\rho}}&:=\frac14\Sigma[(\truchi)^2-\epsuchi^2]\,\Sigma\bsy{\rho}+\frac12\Sigma\epsuchi(\truchi)\, \Sigma\bsy{\sigma} \,,\label{eq:def-rho-hat}\\  
\widehat{\bsy{\sigma}}&:=-\frac12\Sigma\epsuchi(\truchi)\, \Sigma \bsy{\rho}+\frac14\Sigma[(\truchi)^2-\epsuchi^2]\,\Sigma\bsy\sigma \,,\label{eq:def-sigma-hat}
\end{align}
which, in Boyer--Lindquist coordinates, read
\begin{align*}
    \widehat{\bsy{\rho}}&=(r^2-a^2\cos^2\theta)\,\bsy{\rho}-2ar\cos\theta\, \bsy{\sigma}  \,, &
\widehat{\bsy{\sigma}}&= 2ar\cos\theta\,\bsy{\rho}+(r^2-a^2\cos^2\theta)\,\bsy{\sigma}\, .
\end{align*}
We note the identity
\begin{equation}
\widehat{\bsy{\rho}}^2+\widehat{\bsy{\sigma}}^2=(\Sigma\bsy{\rho})^2+(\Sigma\bsy{\sigma})^2 \, . \label{eq:middle-component-squares}
\end{equation}

In this section, we will show:

\begin{proposition}[Modified Maxwell equations] \label{prop:reduced-system} We have the identities
\begin{align}
\nablasl_3 \widehat{\bsy{\sigma}} =& \, \frac12\Sigma^2\epsuchi(\truchi) \divsl\bsy\ualpha- \frac14\Sigma^2[(\truchi)^2-\epsuchi^2]\curlsl\bsy\ualpha \, , \label{D3_sigma} \\
\nablasl_3 \widehat{\bsy{\rho}} =& \, -\frac14\Sigma^2[(\truchi)^2-\epsuchi^2]\,\divsl\,\bsy{\alphab}-\frac12\Sigma^2\epsuchi(\truchi) \curlsl\,\bsy{\alphab} \, , \label{D3_rho} 
\end{align}
\begin{align}
\nablasl_4 \widehat{\bsy{\sigma}} =& \, -\frac12\Sigma\epsuchi(\truchi) \divsl(\Sigma\bsy\alpha)-\frac14\Sigma[(\truchi)^2-\epsuchi^2] \curlsl(\Sigma\bsy\alpha) \, , \label{D4_sigma} \\
\nablasl_4 \widehat{\bsy{\rho}} =& \,  \frac14\Sigma[(\truchi)^2-\epsuchi^2]\,\divsl(\Sigma\bsy{\alpha})-\frac12\Sigma\epsuchi(\truchi)\,\curlsl(\Sigma\bsy{\alpha}) \, , \label{D4_rho} 
\end{align}
\begin{align}
  \nablasl \widehat{\bsy{\sigma}} =& -\frac12\Sigma\epsuchi(\truchi)\,(\mathfrak{T}(\bsy{\alpha})-\underline{\mathfrak{T}}(\bsy{\alphab}))-\frac14\Sigma[(\truchi)^2-\epsuchi^2]\,({}^{\star}\mathfrak{T}(\bsy{\alpha})+{}^{\star}\underline{\mathfrak{T}}(\bsy{\alphab})) \, , \label{DA_sigma} \\
  \nablasl \widehat{\bsy{\rho}} =& \frac14\Sigma[(\truchi)^2-\epsuchi^2]\,(\mathfrak{T}(\bsy{\alpha})-\underline{\mathfrak{T}}(\bsy{\alphab}))-\frac12\Sigma\epsuchi(\truchi)\,({}^{\star}\mathfrak{T}(\bsy{\alpha})+{}^{\star}\underline{\mathfrak{T}}(\bsy{\alphab}))  \, , \label{DA_rho}
\end{align}
%EXPLICIT FORMULAS -- DO NOT ERASE
%\begin{align}
%\nablasl_3 \widehat{\bsy{\sigma}} =& \, -2ar\cos\theta \divsl\bsy\ualpha- (r^2-a^2\cos^2\theta)\curlsl\bsy\ualpha \, , \label{D3_sigma} \\
%\nablasl_3 \widehat{\bsy{\rho}} =& \, -(r^2-a^2\cos^2\theta)\,\divsl\,\bsy{\alphab}+2ar\cos\theta\, \curlsl\,\bsy{\alphab} \, , \label{D3_rho} \\
%\nablasl_4 \widehat{\bsy{\sigma}} =& \, \frac{2ar\cos\theta}{\Sigma} \divsl(\Sigma\bsy\alpha)-\frac{r^2-a^2\cos^2\theta}{\Sigma} \curlsl(\Sigma\bsy\alpha) \, , \label{D4_sigma} \\
%\nablasl_4 \widehat{\bsy{\rho}} =& \,  \frac{r^2-a^2\cos^2\theta}{\Sigma}\,\divsl(\Sigma\bsy{\alpha})+\frac{2ar\cos\theta}{\Sigma}\,\curlsl(\Sigma\bsy{\alpha}) \, , \label{D4_rho} \\
%  \nablasl \widehat{\bsy{\sigma}} =& \frac{2ar\cos\theta}{\Sigma}\,(\mathfrak{T}(\bsy{\alpha})-\mathfrak{S}(\bsy{\alphab}))-\frac{r^2-a^2\cos^2\theta}{\Sigma}\,({}^{\star}\mathfrak{T}(\bsy{\alpha})+{}^{\star}\mathfrak{S}(\bsy{\alphab})) \, , \label{DA_sigma} \\
%  \nablasl \widehat{\bsy{\rho}} =& \frac{r^2-a^2\cos^2\theta}{\Sigma}\,(\mathfrak{T}(\bsy{\alpha})-\mathfrak{S}(\bsy{\alphab}))+\frac{2ar\cos\theta}{\Sigma}\,({}^{\star}\mathfrak{T}(\bsy{\alpha})+{}^{\star}\mathfrak{S}(\bsy{\alphab}))  \, , \label{DA_rho}
%\end{align}
where we have set
\begin{align}
\mathfrak{T}(\bsy{\alpha})&:=\frac{1}{2}\,\Sigma(\nablasl_3\bsy{\alpha}+\chib^{\sharp_2}\cdot\,\bsy{\alpha}) \, , &
\underline{\mathfrak{T}}(\bsy{\alphab})&:=\frac{1}{2}\,\Sigma(\nablasl_4\bsy{\alphab}+\chi^{\sharp_2}\cdot\bsy{\alphab}+\omegah\,\bsy{\alphab}) \, . \label{def_quantities_3_alpha_4_alphab}
\end{align}
\end{proposition}

The equations \eqref{D3_sigma}--\eqref{DA_rho} constitute the \emph{modified system of Maxwell equations}. Crucially, in this system, the equations \eqref{D3_rho}, \eqref{D4_rho} and \eqref{DA_rho} for $\widehat{\bsy{\rho}}$ are \emph{decoupled} from the equations \eqref{D3_sigma}, \eqref{D4_sigma} and \eqref{DA_sigma} for $\widehat{\bsy{\sigma}}$. 

\begin{remark} \label{rmk:spin-weighted} The modified system of Maxwell equations \eqref{D3_sigma}--\eqref{DA_rho} has an analogue in the language of the spin-weighted functions introduced in the previous section. Indeed, with $\swei{\upalpha}{\pm 1}$ as defined in \eqref{eq:map-B-to-NP}, where we have fixed $(e_1^{\rm as},e_2^{\rm as})$ to be given by \eqref{explicit_e1_e2}, and letting
\begin{align*}
\swei{\upupsilon}{0}:=-\widehat{\bsy \rho}+i\widehat{\bsy \sigma}\numberthis\label{eq:def-upupsilon}
\end{align*}
be a smooth complex-valued function on $\mc M$, the Maxwell equations can be written in the form
\begin{align}
\frac{1}{\sqrt{2}}({e_1^{\rm as}}+i{e_2^{\rm as}})\big(\swei{\upupsilon}{0}\big)&= \frac{\kappa^3}{2\Sigma}e_3^{\rm as}\lp(\kappa^{-1}\Delta\swei{\upalpha}{+1}\rp),\label{eq:Maxwell-m}\\
\frac{1}{\sqrt{2}}({e_1^{\rm as}}-i{e_2^{\rm as}})\big(\swei{\upupsilon}{0}\big)&=\frac{\kappa\Sigma}{2\Delta}e_4^{\rm as}\lp(\kappa^{-1}\swei{\upalpha}{-1}\rp),\label{eq:Maxwell-m-bar}\\
\frac{\Sigma}{\Delta}e_4^{\rm as}\big(\swei{\upupsilon}{0}\big)&=\frac{\kappa^2}{\sqrt{2}}\lp(\kappa({e_1^{\rm as}}-i{e_2^{\rm as}})+\cot\theta\rp)\lp(\kappa^{-1}\swei{\upalpha}{+1}\rp),\label{eq:Maxwell-l}\\
\frac{\Delta}{2\Sigma}e_3^{\rm as}\big(\swei{\upupsilon}{0}\big)&=\frac{\kappa^2}{2\sqrt{2}\Sigma}\lp(\overline{\kappa}({e_1^{\rm as}}+i{e_2^{\rm as}})+\cot\theta\rp)\lp(\kappa^{-1} \swei{\upalpha}{-1}\rp), \label{eq:Maxwell-n}
\end{align}
which can be found from the usual Newman--Penrose formalism for the Kerr metric, see e.g.\  the classical textbook \cite[Section 69, eq.\ (11)]{Chandrasekhar}.
\end{remark}

\begin{proof}[Proof of Proposition~\ref{prop:reduced-system}]
We start by deriving the equations \eqref{D3_sigma}--\eqref{D3_rho} and \eqref{D4_sigma}--\eqref{D4_rho}. We compute
\begin{align*}
\nablasl_3\widehat{\bsy{\sigma}}
%%%%% EXTRA DETAILS
%=& \, \nablasl_3(2ar\cos\theta)\bsy{\rho}+ (2ar\cos\theta)\nablasl_3 \bsy{\rho} + \nablasl_3(r^2-a^2\cos^2\theta)\bsy{\sigma}+(r^2-a^2\cos^2\theta)\nablasl_3\bsy{\sigma} \\
%=& \, -2a\cos\theta\bsy{\rho}-2r\bsy{\sigma} \\
%&+2ar\cos\theta(-(\truchi) \bsy\rho  -\divsl\bsy\ualpha+\epsuchi \bsy\sigma)+(r^2-a^2\cos^2\theta)(-(\truchi)\bsy\sigma -\curlsl\bsy\ualpha-\epsuchi \bsy\rho) \\
%%%%% EXTRA DETAILS
=& \, -(2a\cos\theta+2ar\cos\theta(\truchi)+(r^2-a^2\cos^2\theta)\epsuchi )\bsy{\rho} \\
&+(-2r+2ar\cos\theta\epsuchi-(r^2-a^2\cos^2\theta)(\truchi))\bsy{\sigma}\\
&-2ar\cos\theta \divsl\bsy\ualpha- (r^2-a^2\cos^2\theta)\curlsl\bsy\ualpha 
\end{align*}
and
\begin{align*}
\nablasl_3\widehat{\bsy{\rho}} 
%%%%% EXTRA DETAILS
%=& \, -\nablasl_3(2ar\cos\theta)\bsy{\sigma}- (2ar\cos\theta)\nablasl_3 \bsy{\sigma} + \nablasl_3(r^2-a^2\cos^2\theta)\bsy{\rho}+(r^2-a^2\cos^2\theta)\nablasl_3\bsy{\rho} \\
%=& \, 2a\cos\theta\bsy{\sigma}-2r\bsy{\rho} \\
%&-2ar\cos\theta(-(\truchi)\bsy\sigma -\curlsl\bsy\ualpha-\epsuchi \bsy\rho)+(r^2-a^2\cos^2\theta)(-(\truchi) \bsy\rho  -\divsl\bsy\ualpha+\epsuchi \bsy\sigma) \\
%%%%% EXTRA DETAILS
=& \, (2a\cos\theta+2ar\cos\theta(\truchi)+(r^2-a^2\cos^2\theta)\epsuchi )\bsy{\sigma} \\
&+(-2r+2ar\cos\theta\epsuchi-(r^2-a^2\cos^2\theta)(\truchi))\bsy{\rho}\\
&+2ar\cos\theta \curlsl\bsy\ualpha- (r^2-a^2\cos^2\theta)\divsl\bsy\ualpha \, ,
\end{align*}
where we used the equations \eqref{eq:del3-sigma}--\eqref{eq:del3-rho}. Using Section \ref{sec:algebraically-special-props}, both the brackets multiplying $\bsy{\rho}$ and $\bsy{\sigma}$ can be checked to identically vanish, i.e.
\begin{align*}
2a\cos\theta+2ar\cos\theta(\truchi)+(r^2-a^2\cos^2\theta)\epsuchi&=0 \, , \\
-2r+2ar\cos\theta\epsuchi-(r^2-a^2\cos^2\theta)(\truchi) &= 0 \, .
\end{align*}
Similarly, we compute
\begin{align*}
\nablasl_4\widehat{\bsy{\sigma}}
%%%%% EXTRA DETAILS
%=& \,\nablasl_4(2ar\cos\theta)\bsy{\rho}+ (2ar\cos\theta)\nablasl_4 \bsy{\rho} + \nablasl_4(r^2-a^2\cos^2\theta)\bsy{\sigma}+(r^2-a^2\cos^2\theta)\nablasl_4\bsy{\sigma} \\
%=& \,  \frac{\Delta}{\Sigma}2a\cos\theta\bsy{\rho}+\frac{\Delta}{\Sigma}2r\bsy{\sigma} \\
%&+2ar\cos\theta(-(\trchi) \bsy\rho +\divsl\bsy\alpha+(\eta+\ueta,\bsy\alpha)-\epschi \bsy\sigma) \\ 
%&+(r^2-a^2\cos^2\theta)(-(\trchi)\bsy\sigma -\curlsl\bsy\alpha-(\eta+\ueta)\wedge \bsy\alpha+\epschi \bsy\rho) \\
%%%%% EXTRA DETAILS
=& \, \lp(\frac{\Delta}{\Sigma}2a\cos\theta-2ar\cos\theta(\trchi)+(r^2-a^2\cos^2\theta)\epschi\rp)\bsy{\rho} \\
&+\lp(\frac{\Delta}{\Sigma}2r-2ar\cos\theta\epschi-(r^2-a^2\cos^2\theta)(\trchi)\rp)\bsy{\sigma} \\
&+2ar\cos\theta(\divsl\bsy\alpha+(\eta+\ueta,\bsy\alpha))-(r^2-a^2\cos^2\theta)(\curlsl\bsy\alpha+(\eta+\ueta)\wedge \bsy\alpha)
\end{align*}
and
\begin{align*}
\nablasl_4\widehat{\bsy{\rho}} 
%%%%% EXTRA DETAILS
%=& \, -\nablasl_4(2ar\cos\theta)\bsy{\sigma}- (2ar\cos\theta)\nablasl_4 \bsy{\sigma} + \nablasl_4(r^2-a^2\cos^2\theta)\bsy{\rho}+(r^2-a^2\cos^2\theta)\nablasl_4\bsy{\rho} \\
%=& \,  -\frac{\Delta}{\Sigma}2a\cos\theta\bsy{\sigma}+\frac{\Delta}{\Sigma}2r\bsy{\rho} \\
%&-2ar\cos\theta(-(\trchi)\bsy\sigma -\curlsl\bsy\alpha-(\eta+\ueta)\wedge \bsy\alpha+\epschi \bsy\rho) \\ 
%&+(r^2-a^2\cos^2\theta)(-(\trchi) \bsy\rho +\divsl\bsy\alpha+(\eta+\ueta,\bsy\alpha)-\epschi \bsy\sigma) \\
%%%%% EXTRA DETAILS
=& \, -\lp(\frac{\Delta}{\Sigma}2a\cos\theta-2ar\cos\theta(\trchi)+(r^2-a^2\cos^2\theta)\epschi\rp)\bsy{\sigma} \\
&+\lp(\frac{\Delta}{\Sigma}2r-2ar\cos\theta\epschi-(r^2-a^2\cos^2\theta)(\trchi)\rp)\bsy{\rho} \\
&+2ar\cos\theta(\curlsl\bsy\alpha+(\eta+\ueta)\wedge \bsy\alpha)+(r^2-a^2\cos^2\theta)(\divsl\bsy\alpha+(\eta+\ueta,\bsy\alpha)) \, ,
\end{align*}
where we used the equations \eqref{eq:del4-sigma}--\eqref{eq:del4-rho}. To conclude, we use the fact that, in view of the identities of Section~\ref{sec:algebraically-special-props}, we have 
\begin{align*}
\frac{\Delta}{\Sigma}2a\cos\theta-2ar\cos\theta(\trchi)+(r^2-a^2\cos^2\theta)\epschi &= 0 \, , \\
\frac{\Delta}{\Sigma}2r-2ar\cos\theta\epschi-(r^2-a^2\cos^2\theta)(\trchi) &= 0 \, .
\end{align*}

We now derive the equations \eqref{DA_sigma}--\eqref{DA_rho}. By summing and subtracting the equations \eqref{eq:del3-alpha'} and \eqref{eq:del4-ualpha'} (and taking horizontal Hodge duals), we obtain
\begin{align}
\nablasl(\Sigma\bsy{\sigma})-\Sigma\bsy{\rho}{}^{\star}(\eta-\etab)&=-{}^{\star}\mathfrak{T}(\bsy{\alpha})-{}^{\star}\underline{\mathfrak{T}}(\bsy{\alphab}) \, , \label{id_diff}\\
\nablasl(\Sigma\bsy{\rho})+\Sigma\bsy{\sigma}{}^{\star}(\eta-\etab)&=\mathfrak{T}(\bsy{\alpha})-\underline{\mathfrak{T}}(\bsy{\alphab})\,.  \label{id_sum}
\end{align}
We then compute
\begin{align*}
\nablasl \widehat{\bsy{\sigma}} 
%%%%% EXTRA DETAILS
%=& \, \nablasl(r^2-a^2\cos^2\theta)\bsy{\sigma}+(r^2-a^2\cos^2\theta)\nablasl\bsy{\sigma}+\nablasl(2ar\cos\theta)\bsy{\rho}+ (2ar\cos\theta)\nablasl \bsy{\rho} \\
%%%%% EXTRA DETAILS
&=   \nablasl(r^2-a^2\cos^2\theta)\Sigma\bsy{\sigma}+\frac{r^2-a^2\cos^2\theta}{\Sigma}\nablasl(\Sigma\bsy{\sigma})+(r^2-a^2\cos^2\theta)\nablasl(\Sigma^{-1})\Sigma\bsy{\sigma} \\
&\qquad +\nablasl(r\Sigma(\slashed{\varepsilon}\cdot\chib))\bsy{\rho}+r(\slashed{\varepsilon}\cdot\chib)\nablasl(\Sigma\bsy{\rho})+r(\slashed{\varepsilon}\cdot\chib)\nablasl(\Sigma^{-1})\Sigma\bsy{\rho} \\
&= (-(\eta+\etab)\Sigma-(r^2-a^2\cos^2\theta)(\eta+\etab)-r(\slashed{\varepsilon}\cdot\chib){}^{\star}(\eta-\etab)\Sigma)\bsy{\sigma} \\
&\qquad -\lp(-\frac{r^2-a^2\cos^2\theta}{\Sigma}{}^{\star}(\eta-\etab)\Sigma-\nablasl(r\Sigma(\slashed{\varepsilon}\cdot\chib))+r(\slashed{\varepsilon}\cdot\chib)\Sigma(\eta+\etab)\rp)\bsy{\rho} \\
&\qquad +r(\slashed{\varepsilon}\cdot\chib)(\mathfrak{T}(\bsy{\alpha})-\underline{\mathfrak{T}}(\bsy{\alphab}))-\frac{r^2-a^2\cos^2\theta}{\Sigma}({}^{\star}\mathfrak{T}(\bsy{\alpha})+{}^{\star}\underline{\mathfrak{T}}(\bsy{\alphab}))
\end{align*}
and
\begin{align*}
\nablasl \widehat{\bsy{\rho}} 
%%%%% EXTRA DETAILS
%=& \,  \nablasl(r^2-a^2\cos^2\theta)\bsy{\rho}+(r^2-a^2\cos^2\theta)\nablasl\bsy{\rho}-\nablasl(2ar\cos\theta)\bsy{\sigma}- (2ar\cos\theta)\nablasl \bsy{\sigma} \\
%%%%% EXTRA DETAILS
&=  \nablasl(r^2-a^2\cos^2\theta)\Sigma\bsy{\rho}+\frac{r^2-a^2\cos^2\theta}{\Sigma}\nablasl(\Sigma\bsy{\rho})+(r^2-a^2\cos^2\theta)\nablasl(\Sigma^{-1})\Sigma\bsy{\rho} \\
&\qquad-\nablasl(r\Sigma(\slashed{\varepsilon}\cdot\chib))\bsy{\sigma}-r(\slashed{\varepsilon}\cdot\chib)\nablasl(\Sigma\bsy{\sigma})-r(\slashed{\varepsilon}\cdot\chib)\nablasl(\Sigma^{-1})\Sigma\bsy{\sigma} \\
&= (-(\eta+\etab)\Sigma-(r^2-a^2\cos^2\theta)(\eta+\etab)-r(\slashed{\varepsilon}\cdot\chib){}^{\star}(\eta-\etab)\Sigma)\bsy{\rho} \\
&\qquad+\lp(-\frac{r^2-a^2\cos^2\theta}{\Sigma}{}^{\star}(\eta-\etab)\Sigma-\nablasl(r\Sigma(\slashed{\varepsilon}\cdot\chib))+r(\slashed{\varepsilon}\cdot\chib)\Sigma(\eta+\etab)\rp)\bsy{\sigma} \\
&\qquad+r(\slashed{\varepsilon}\cdot\chib)({}^{\star}\mathfrak{T}(\bsy{\alpha})+{}^{\star}\underline{\mathfrak{T}}(\bsy{\alphab}))+\frac{r^2-a^2\cos^2\theta}{\Sigma}(\mathfrak{T}(\bsy{\alpha})-\underline{\mathfrak{T}}(\bsy{\alphab})) \, ,
\end{align*}
where we used \eqref{eq:ang-derivative-Sigma} and the identities \eqref{id_diff}--\eqref{id_sum}. Both the brackets multiplying $\bsy{\rho}$ and $\bsy{\sigma}$ can be checked to identically vanish, i.e.
\begin{align*}
-(\eta+\etab)\Sigma-(r^2-a^2\cos^2\theta)(\eta+\etab)-r(\slashed{\varepsilon}\cdot\chib){}^{\star}(\eta-\etab)\Sigma &= 0 \, , \\
-(r^2-a^2\cos^2\theta){}^{\star}(\eta-\etab)-r\Sigma(\tr \chib){}^\star \eta+r\Sigma\epsuchi\etab &= 0 \, ,
\end{align*}
where we used the geometric relation \eqref{formula_nablasl_atrchib} and the identities of Section~\ref{sec:algebraically-special-props}. 
\end{proof}

\subsubsection{Notation for solutions}

In the sequel, it will be useful to have some notation to distinguish solutions to the original and modified Maxwell system. We let $$\bsy{\mathfrak{S}}_{\circ}=(\bsy{\alpha}[\bsy{\mathfrak{S}}_{\circ}],\bsy{\alphab}[\bsy{\mathfrak{S}}_{\circ}],\bsy{\rho}[\bsy{\mathfrak{S}}_{\circ}],\bsy{\sigma}[\bsy{\mathfrak{S}}_{\circ}])$$ be a \emph{solution} to the system of Maxwell equations \eqref{eq:del3-sigma}--\eqref{eq:del4-ualpha}. We denote by $$\bsy{\mathfrak S}=(\bsy{\alpha}[\bsy{\mathfrak S}],\bsy{\alphab}[\bsy{\mathfrak S}],\widehat{\bsy{\rho}}[\bsy{\mathfrak S}],\widehat{\bsy{\sigma}}[\bsy{\mathfrak S}])$$ the corresponding \emph{modified solution} (i.e.~solution to the modified system of Maxwell equations \eqref{D3_sigma}--\eqref{DA_rho}), with
\begin{align*}
\bsy{\alpha}[\bsy{\mathfrak S}]&=\bsy{\alpha}[\bsy{\mathfrak S}_{\circ}] \, , & \bsy{\alphab}[\bsy{\mathfrak S}]&=\bsy{\alphab}[\bsy{\mathfrak S}_{\circ}] \, , & \widehat{\bsy{\rho}}[\bsy{\mathfrak S}]&=\widehat{\bsy{\rho}}[\bsy{\mathfrak S}_{\circ}] \, , & \widehat{\bsy{\sigma}}[\bsy{\mathfrak S}]&=\widehat{\bsy{\sigma}}[\bsy{\mathfrak S}_{\circ}] \, .
\end{align*}

\subsection{Special solutions}

In this section, we present two classes of special solutions to the system of Maxwell equations.

\subsubsection{Stationary solutions}
\label{sec:stationary-solutions}

Let $\mathfrak{q}_1,\mathfrak{q}_2\in\mathbb{R}$, to be interpreted as electromagnetic charges. We define the two-parameter family of stationary (smooth) solutions $\bsy{\mathfrak S}_{\circ}{}_{\mathfrak{q}_1,\mathfrak{q}_2}$ such that
\begin{align*}
\bsy{\alpha}[\bsy{\mathfrak S}_{\circ}{}_{\mathfrak{q}_1,\mathfrak{q}_2}]&=0 \, , \\ \bsy{\alphab}[\bsy{\mathfrak S}_{\circ}{}_{\mathfrak{q}_1,\mathfrak{q}_2}]&=0 \, , \\
\bsy{\rho}[\bsy{\mathfrak S}_{\circ}{}_{\mathfrak{q}_1,\mathfrak{q}_2}]&= \frac{1}{\Sigma}\left(\mathfrak{q}_1 \frac{r^2-a^2\cos^2\theta}{\Sigma}+2\,\mathfrak{q}_2\frac{ar\cos\theta}{\Sigma}\right) \, , \\
\bsy{\sigma}[\bsy{\mathfrak S}_{\circ}{}_{\mathfrak{q}_1,\mathfrak{q}_2}]&= \frac{1}{\Sigma}\left(-2\,\mathfrak{q}_1 \frac{ar\cos\theta}{\Sigma}+\mathfrak{q}_2\frac{r^2-a^2\cos^2\theta}{\Sigma}\right) \, ,
\end{align*}
which induces the two-parameter family of modified stationary (smooth) solutions $\bsy{\mathfrak S}_{\mathfrak{q}_1,\mathfrak{q}_2}$ such that
\begin{align*}
\bsy{\alpha}[\bsy{\mathfrak S}_{\mathfrak{q}_1,\mathfrak{q}_2}]&=0 \, , & \bsy{\alphab}[\bsy{\mathfrak S}_{\mathfrak{q}_1,\mathfrak{q}_2}]&=0 \, , \\
\widehat{\bsy{\rho}}[\bsy{\mathfrak S}_{\mathfrak{q}_1,\mathfrak{q}_2}]&= \mathfrak{q}_1 \, , &
\widehat{\bsy{\sigma}}[\bsy{\mathfrak S}_{\mathfrak{q}_1,\mathfrak{q}_2}]&=\mathfrak{q}_2  \, .
\end{align*}
Using the modified system of Maxwell equations \eqref{D3_sigma}--\eqref{DA_rho}, it is immediate to check that, if $\bsy{\alpha}[\bsy{\mathfrak S}]=\bsy{\alphab}[\bsy{\mathfrak S}]=0$ for a given modified solution $\bsy{\mathfrak{S}}$, then $\bsy{\mathfrak{S}}$ is a modified stationary solution.

\subsubsection{Solutions supported on a fixed azimuthal mode}

Let $m\in \mathbb{\mathbb{Z}}$. We say that a solution $\bsy{\mathfrak S}_\circ$ of the system of Maxwell equations is supported on the fixed azimuthal mode $m$ if
\begin{align*}
\widetilde Z\big(\swei{\upalpha}{\pm 1}\big)&=im\,\swei{\upalpha}{\pm 1},
\end{align*}
where $\Delta^{\pm 1}\swei{\upalpha}{\pm 1}\in\mathscr S^{[\pm 1]}_\infty(\mc M)$ are the image of $\bsy{\alpha}[\bsy{\mathfrak S}_\circ]$ and $\bsy{\alphab}[\bsy{\mathfrak S}_\circ]$, respectively, under the invertible maps $\mathrm i^{[\pm 1]}$ (see Lemma \ref{lemma:isomorphism}), and if 
\begin{align*}
\widetilde Z\big(\bsy{\rho}[\bsy{\mathfrak S}_{\circ}]\big)&=im\,\bsy{\rho}[\bsy{\mathfrak S}_{\circ}]\,, & \widetilde Z\big(\bsy{\sigma}[\bsy{\mathfrak S}_{\circ}]\big)&= im\,\bsy{\sigma}[\bsy{\mathfrak S}_{\circ}],
\end{align*}
after promoting $(\bsy{\rho}[\bsy{\mathfrak S}_{\circ}],\bsy{\sigma}[\bsy{\mathfrak S}_{\circ}])$ to complex-valued functions on $\mc M$. This induces a modified solution $\bsy{\mathfrak S}$ supported on  the fixed azimuthal mode $m$, which is such that
\begin{align*}
\widetilde Z\big(\swei{\upalpha}{\pm 1}\big)&=im\,\swei{\upalpha}{\pm 1}\,, &
\widetilde Z\big(\bsy{\rho}[\bsy{\mathfrak S}]\big)&=im\,\bsy{\rho}[\bsy{\mathfrak S}]\,, &
\widetilde Z\big(\bsy{\sigma}[\bsy{\mathfrak S}]\big)&= im\,\bsy{\sigma}[\bsy{\mathfrak S}],
\end{align*}
where $\Delta^{\pm 1}\swei{\upalpha}{\pm 1}\in\mathscr S^{[\pm 1]}_\infty(\mc M)$ are the image of $\bsy{\alpha}[\bsy{\mathfrak S}]=\bsy{\alpha}[\bsy{\mathfrak S}_\circ]$ and $\bsy{\alphab}[\bsy{\mathfrak S}]=\bsy{\alphab}[\bsy{\mathfrak S}_\circ]$, respectively, under the invertible maps $\mathrm i^{[\pm 1]}$.

\subsection{Initial data and well-posedness}

The Maxwell equations \eqref{eq:del3-sigma}--\eqref{eq:del4-ualpha} form an overdetermined system of equations. As a result, the initial data for the system have to satisfy a set of constraint equations. In this section, we discuss how initial data are to be prescribed and briefly comment on the (global) well-posedness of the Maxwell equations \eqref{eq:del3-sigma}--\eqref{eq:del4-ualpha}. In fact, it will be convenient to present our discussion in terms of the modified Maxwell equations \eqref{D3_sigma}--\eqref{DA_rho} and modified solutions $\bsy{\mathfrak{S}}_{\circ}$. Through the identities \eqref{eq:def-rho-hat}--\eqref{eq:def-sigma-hat}, our statements will immediately imply analogous statements for the original Maxwell equations \eqref{eq:del3-sigma}--\eqref{eq:del4-ualpha} and solutions $\bsy{\mathfrak{S}}$.

\emph{In this section, as well as in the remainder of the paper, we always consider smooth solutions to the Maxwell equations.}

\subsubsection{Seed initial data}

The initial data for the modified Maxwell equations \eqref{D3_sigma}--\eqref{DA_rho} are prescribed on the hyperboloidal hypersurface $\Sigma_0$ (recall definition \eqref{def_hyperboloidal_hypers}). The initial data quantities which can be freely prescribed will be referred to as \emph{seed initial data}. Since $\mathfrak{D}_{\mathcal{N}_{\text{as}}}\not\subset T\Sigma_0$, formulating a notion of seed initial data for a system of equations for $\mathfrak{D}_{\mathcal{N}_{\text{as}}}$ tensors, starting from a suitable prescription of quantities on (and intrinsic to) $\Sigma_0$, is slightly subtle.

As a preliminary step, it is useful to relate $\mathfrak{D}_{\mathcal{N}_{\text{as}}}$ one-forms to spacetime one-forms. The former admit canonical extensions to smooth sections
\begin{align*}
\Gamma_{\circ}((T\mathcal{M})^{\star})&:=\{\bsy\varsigma \in \Gamma((T\mathcal{M})^{\star})\colon \bsy\varsigma(e_3^{\text{as}})=\bsy\varsigma(e_4^{\text{as}})=0\} \, ,
\end{align*}
i.e.~spacetime one-forms which identically vanish when evaluated on $e_3^{\text{as}}$ and $e_4^{\text{as}}$ (see e.g.~\cite[Section 7.2]{Benomio2022}). By employing the hyperboloidal differentiable structure of Section \ref{sec_hyperboloidal_coords}, a more explicit characterization of the canonical extension of $\mathfrak{D}_{\mathcal{N}_{\text{as}}}$ one-forms is as follows:

\begin{lemma} \label{lemma:Gamma0-forms} There exist smooth scalar functions $c_{1,2}(r)$ such that the following holds true:~$\bsy{\varsigma}\in \Gamma_{\circ}((T\mathcal{M})^{\star})$ if and only if there exists  $\bsy{\varsigma}_{\mathbb{S}^2_{\tilde{t}{}^*,r}}\in \Gamma((T\mathbb{S}^2_{\tilde{t}{}^*,r})^\star)$ such that 
\begin{equation*}
\bsy{\varsigma}=[c_1(r)\,\bsy{\varsigma}{}_{\mathbb{S}^2_{\tilde{t}{}^*,r}}(Z)]\, d\tilde{t}{}^*+[c_2(r)\,\bsy{\varsigma}{}_{\mathbb{S}^2_{\tilde{t}{}^*,r}}(Z)]\, dr+\bsy{\varsigma}{}_{\mathbb{S}^2_{\tilde{t}{}^*,r}} \, ,
\end{equation*}
and therefore there exists an isomorphism between $\Gamma_{\circ}((T\mathcal{M})^{\star})$ and $\Gamma((T\mathbb{S}^2_{\tilde{t}{}^*,r})^\star)$.
\end{lemma}

\begin{proof} For any $\bsy{\varsigma}\in \Gamma_{\circ}((T\mathcal{M})^{\star})$, we have
\begin{align*}
\bsy{\varsigma}= \bsy{\varsigma}(\p_{\tilde{t}{}^*})d\tilde{t}{}^*+\bsy{\varsigma}(\p_{r})dr + \bsy{\varsigma}(\p_{\theta})d\theta + \bsy{\varsigma}(\p_{\tilde{\phi}{}^*})d\tilde{\phi}{}^*\,.
\end{align*}
The conditions $\bsy{\varsigma}(e_4^{\text{as}})=\bsy{\varsigma}(e_3^{\text{as}})=0$ are equivalent to the constraining identities
\begin{align*} 
\bsy{\varsigma}(\p_{\tilde{t}{}^*})&=c_1(r)\,\bsy{\varsigma}(\p_{\tilde{\phi}{}^*})  \, , & \bsy{\varsigma}(\p_r)&= c_2(r)\,\bsy{\varsigma}(\p_{\tilde{\phi}{}^*}) 
\end{align*}
for some smooth scalar functions $c_{1,2}(r)$. By defining 
\begin{equation*}
    \bsy{\varsigma}{}_{\mathbb{S}^2_{\tilde{t}{}^*,r}}=\bsy{\varsigma}(\p_{\theta})d\theta + \bsy{\varsigma}(\p_{\tilde{\phi}{}^*})d\tilde{\phi}{}^* \, ,
\end{equation*}
one can write $\bsy{\varsigma}$ as in the lemma.
\end{proof}

\begin{remark}
    One has $c_{1,2}(r)\rightarrow 0$ as $r\rightarrow \infty$ and $c_{1,2}(r)$ identically vanishing for $|a|=0$.
\end{remark}

We now state the definition of seed initial data for the modified Maxwell equations \eqref{D3_sigma}--\eqref{DA_rho}. We recall that, by standard Hodge decomposition on $\mathbb{S}^2$, any $\mathbb{S}^2_{\tilde{t}{}^*,r}$ one-form $\bsy{\varsigma}{}_{\mathbb{S}^2_{\tilde{t}{}^*,r}}$ is uniquely represented as
\begin{equation*}
    \bsy{\varsigma}{}_{\mathbb{S}^2_{\tilde{t}{}^*,r}}=-\nablasl_{\mathbb{S}^2_{\tilde{t}{}^*,r}} \bsy{f}+{}^{\star}\nablasl_{\mathbb{S}^2_{\tilde{t}{}^*,r}} \bsy{h} \, ,
\end{equation*}
with $\bsy{f}$ and $\bsy{h}$ smooth scalar functions.

\begin{definition}[Seed initial data] \label{def_seed_data}
Let 
\begin{equation} \label{data_functions}
(\mathfrak{f},\underline{\mathfrak{f}},\mathfrak{h},\underline{\mathfrak{h}})
\end{equation}
be smooth scalar functions on $\Sigma_0$ and 
\begin{equation} \label{data_charges}
(\mathfrak{r},\mathfrak{s})
\end{equation}
be smooth scalar functions on $\mathbb{S}^2_{0,r_+}=\Sigma_0\cap\mathcal{H}^+$. We refer to the scalars \eqref{data_functions}--\eqref{data_charges} as \emph{seed quantities} on $\Sigma_0$. We define \emph{seed initial data} on $\Sigma_0$ for the modified Maxwell equations \eqref{D3_sigma}--\eqref{DA_rho} as the (canonical extensions of) $\mathfrak{D}_{\mathcal{N}_{\text{as}}}$ one-forms 
\begin{equation}
\begin{split}
\bsy{\alpha}|_{\Sigma_0}=& \, [c_1(r)\cdot\,(-\nablasl_{\mathbb{S}^2_{0,r}} \mathfrak{f}+{}^{\star}\nablasl_{\mathbb{S}^2_{0,r}} \mathfrak{h})(Z)]\,d\tilde{t}{}^*+[c_2(r)\cdot\,(-\nablasl_{\mathbb{S}^2_{0,r}} \mathfrak{f}+{}^{\star}\nablasl_{\mathbb{S}^2_{0,r}} \mathfrak{h})(Z)] \, dr \\
&-\nablasl_{\mathbb{S}^2_{0,r}} \mathfrak{f}+{}^{\star}\nablasl_{\mathbb{S}^2_{0,r}} \mathfrak{h} \, , \\
\bsy{\alphab}|_{\Sigma_0}=& \, [c_1(r)\cdot\,(-\nablasl_{\mathbb{S}^2_{0,r}} \underline{\mathfrak{f}}+{}^{\star}\nablasl_{\mathbb{S}^2_{0,r}} \underline{\mathfrak{h}})(Z)]\,d\tilde{t}{}^*+[c_2(r)\cdot\,(-\nablasl_{\mathbb{S}^2_{0,r}} \underline{\mathfrak{f}}+{}^{\star}\nablasl_{\mathbb{S}^2_{0,r}} \underline{\mathfrak{h}})(Z)] \, dr \\
&-\nablasl_{\mathbb{S}^2_{0,r}} \underline{\mathfrak{f}}+{}^{\star}\nablasl_{\mathbb{S}^2_{0,r}} \underline{\mathfrak{h}} 
\end{split}\label{data_alpha_alphab}
\end{equation}
on $\Sigma_0$ and smooth scalar functions
\begin{align}
\widehat{\bsy{\rho}}\, |_{\mathbb{S}^2_{0,r_+}}&= \mathfrak{r} \, , & \widehat{\bsy{\sigma}}\, |_{\mathbb{S}^2_{0,r_+}}&= \mathfrak{s} \label{data_rho_sigma}
\end{align}
on the horizon sphere $\mathbb{S}^2_{0,r_+}$.
\end{definition}

\begin{remark}
The family of modified stationary solutions arises from seed quantities
\begin{align*}
(\mathfrak{f},\underline{\mathfrak{f}},\mathfrak{h},\underline{\mathfrak{h}})&=0 \, ,  & (\mathfrak{r},\mathfrak{s})&=(\mathfrak{q}_1,\mathfrak{q}_2)
\end{align*}
on $\Sigma_0$ and $\mathbb{S}^2_{0,r_+}$ respectively.
\end{remark}

\begin{remark} Modified solutions supported on a fixed azimuthal number $m\in \mathbb Z$ arise from seed quantities satisfying
\begin{align*}
Z(\mathfrak{f},\underline{\mathfrak{f}},\mathfrak{h},\underline{\mathfrak{h}})&=im\, (\mathfrak{f},\underline{\mathfrak{f}},\mathfrak{h},\underline{\mathfrak{h}})\,, &
Z(\mathfrak{r},\mathfrak{s})&=im\, (\mathfrak{r},\mathfrak{s})\,.
\end{align*}
on $\Sigma_0$ and $\mathbb{S}^2_{0,r_+}$ respectively. We refer to these as \emph{seed initial data supported on the fixed azimuthal number} $m\in \mathbb Z$.
\end{remark}

\begin{remark}
    In view of Lemma \ref{lemma:Gamma0-forms}, the space of seed quantities on $\Sigma_0$ is isomorphic to the space of seed initial data on $\Sigma_0$.
\end{remark}

\subsubsection{Well-posedness}

We have the following well-posedness statement for the modified Maxwell equations \eqref{D3_sigma}--\eqref{DA_rho}.

\begin{proposition}[Well-posedness] \label{prop:wp}
For any suitable\footnote{The proof of the proposition relies on solving a constraint (transport) equation along $\Sigma_0$. The seed initial data are to be prescribed so that this equation admits a global solution. For the main results of this work, we consider initial data whose finite energy norm implies that the constraint (transport) equation can indeed be integrated \emph{globally} along $\Sigma_0$.} seed initial data on $\Sigma_0$, there exists a unique smooth solution to the modified Maxwell equations \eqref{D3_sigma}--\eqref{DA_rho} on $J^+(\Sigma_0)\cap\mathcal{M}$ which agrees with the seed initial data on $\Sigma_0$.
\end{proposition}

\begin{proof}
    From any prescription of seed initial data on $\Sigma_0$, one can determine all the unknowns of the modified Maxwell equations \eqref{D3_sigma}--\eqref{DA_rho} along $\Sigma_0$. This fact is more easily seen by exploiting the Maxwell equations \eqref{eq:Maxwell-m}--\eqref{eq:Maxwell-n}. Indeed, by taking the linear combination 
    \begin{align*}
        % -\sqrt{2}\Delta\bar{\kappa}f\cdot\eqref{eq:Maxwel-m-2}+\sqrt{2}\Delta\frac{\bar{\kappa}\bar{\kappa}}{\kappa}\bar{f}\cdot\eqref{eq:Maxwel-m-bar-2}+(r^2+a^2)^2\cdot\eqref{eq:Maxwel-l-2}-(r^2+a^2)^2\cdot\eqref{eq:Maxwel-n-2}\\
        % \iff
        % -\Delta\bar{\kappa}\lp(-\frac{ia\sin\theta }{\kappa}\rp)\cdot\eqref{eq:Maxwel-m-2}+\Delta\frac{\bar{\kappa}\bar{\kappa}}{\kappa}\lp(\frac{ia\sin\theta }{\overline\kappa}\rp)\cdot\eqref{eq:Maxwel-m-bar-2}+(r^2+a^2)^2\cdot\eqref{eq:Maxwel-l-2}-(r^2+a^2)^2\cdot\eqref{eq:Maxwel-n-2}\\
        % \iff
        % ia\sin\theta\frac{\overline\kappa}{\kappa}\frac{\Delta}{(r^2+a^2)^2}\cdot\eqref{eq:Maxwel-m-2}+ia\sin\theta\frac{\overline\kappa}{\kappa}\frac{\Delta}{(r^2+a^2)^2}\cdot\eqref{eq:Maxwel-m-bar-2}+\eqref{eq:Maxwel-l-2}-\eqref{eq:Maxwel-n-2}\\
        % \iff 
        ia\sin\theta\frac{\Delta \overline\kappa^2}{(r^2+a^2)^2 \kappa}\overline \kappa\cdot\eqref{eq:Maxwell-m}+ia\sin\theta\frac{\Delta \overline\kappa}{(r^2+a^2)^2}\kappa\cdot\eqref{eq:Maxwell-m-bar}+\frac{\Delta}{r^2+a^2}\cdot\eqref{eq:Maxwell-l}-\frac{2\Sigma}{r^2+a^2}\cdot\eqref{eq:Maxwell-n}
    \end{align*}

    one can eliminate the $\partial_{\tilde t^*}$-derivatives of all the unknowns and obtain a first order ODE for $\swei{\upupsilon}{0}$ in a direction tangent to $\Sigma_0$ (and transversal to the spheres foliating $\Sigma_0$). The seed initial data for $\bsy{\widehat{\rho}}$ and $\bsy{\widehat{\sigma}}$ prescribed on $\mathcal{H}^+\cap\Sigma_0$ provides the initial datum for the ODE, whereas the seed initial data for $\bsy{\alpha}$ and $\bsy{\alphab}$ prescribed on $\Sigma_0$ completely determine the right hand side of the ODE on the entire $\Sigma_0$. The ODE can then be integrated along $\Sigma_0$ so as to uniquely determine smooth $\swei{\upupsilon}{0}$ (and therefore $\bsy{\widehat{\rho}}$ and $\bsy{\widehat{\sigma}}$) on the entire $\Sigma_0$.

    Once all the unknowns are determined on $\Sigma_0$, the modified Maxwell equations \eqref{D3_sigma}--\eqref{DA_rho} determine the first order normal (to $\Sigma_0$) covariant derivative of all the unknowns along $\Sigma_0$. 
    
    The fact that $\bsy{\alpha}$ and $\bsy{\alphab}$ decouple into Teukolsky equations, together with the standard well-posedness theory (in the smooth class) for the Cauchy problem for the Teukolsky equations (in light of  Lemma~\ref{lemma:isomorphism}, see e.g.\ \cite[Proposition 4.12]{SRTdC2023}), guarantee that one can then uniquely determine smooth $\bsy{\alpha}$ and $\bsy{\alphab}$ on $J^+(\Sigma_0)\cap\mathcal{M}$. Knowing $\bsy{\alpha}$ and $\bsy{\alphab}$ on $J^+(\Sigma_0)\cap\mathcal{M}$, one applies the equations to uniquely determine smooth $\widehat{\bsy{\rho}}$ and $\widehat{\bsy{\sigma}}$ on $J^+(\Sigma_0)\cap\mathcal{M}$. One can check that the first order ODE originally used to determine $\bsy{\widehat{\rho}}$ and $\bsy{\widehat{\sigma}}$ on $\Sigma_0$ is automatically satisfied on any $\Sigma_{\tau}\subset J^+(\Sigma_0)\cap\mathcal{M}$ with $\tau\geq 0$ (i.e.~propagation of constraints).
\end{proof}

\section{Energy norms} \label{sec:energy_norms_head}

In this section, we define energy norms for both extremal and middle Maxwell components. The energy norms are defined over the hyperboloidal hypersurfaces $\Sigma_{\tau}$ defined in \eqref{def_hyperboloidal_hypers} (see also Figure \eqref{fig:hyperboloidal}).

\subsection{Extremal Maxwell components}
\label{sec:Teukolsky-norms}

Let $\bsy\alpha$ and $\bsy \alphab$ be $\mathfrak D_{\mc N_{\rm as}}$ one-forms, and let $p\in[0,2)$. We define the second order energy fluxes
\begin{align*}
\mathbb{E}_p^1[r^2\bsy\alpha](\tau)
&:=
\int_{\Sigma_{\tau}}\lp(r^{p}\big|\nablasl^2(r^2\bsy\alpha)\big|^2+r^{p}\big|{\nablasl}\nablasl_4(r^2\bsy\alpha)\big|^2+r^p\big|\nablasl_4^2(r^2\bsy\alpha)\big|^2\rp)drd\Omega\\
&\qquad +
\int_{\Sigma_{\tau}}\lp(r^{p+2}\big|\nablasl_4\nablasl_3(r^2\bsy\alpha)\big|^2+\frac{\Delta}{r^2}\big|\nablasl_3^2(r^2\bsy\alpha)\big|^2+r^2\big|{\nablasl}\nablasl_3(r^2\bsy\alpha)\big|^2+\big|\nablasl_3(r^2\bsy\alpha)\big|^2\rp)drd\Omega\\
&\qquad +
\int_{\Sigma_{\tau}}\lp(r^p\big|\nablasl_4(r^2\bsy\alpha)\big|^2+\big|{\nablasl}(r^2\bsy\alpha)\big|^2+r^{-2}\big|r^2\bsy\alpha\big|^2\rp)drd\Omega\,,\\
\overline{\mathbb{E}}_p^1[r^2\bsy\alpha](\tau)
&:=
\int_{\Sigma_{\tau}}\lp(r^{p}\big|{\nablasl}^2(r^2\bsy\alpha)\big|^2+r^{p}\big|{\nablasl}\nablasl_4(r^2\bsy\alpha)\big|^2+r^p\big|\nablasl_4^2(r^2\bsy\alpha)\big|^2\rp)drd\Omega\\
&\qquad +
\int_{\Sigma_{\tau}}\lp(r^{p+2}\big|\nablasl_4\nablasl_3(r^2\bsy\alpha)\big|^2+\big|\nablasl_3^2(r^2\bsy\alpha)\big|^2+r^2\big|{\nablasl}\nablasl_3(r^2\bsy\alpha)\big|^2+\big|\nablasl_3(r^2\bsy\alpha)\big|^2\rp)drd\Omega\\
&\qquad +
\int_{\Sigma_{\tau}}\lp(r^p\big|\nablasl_4(r^2\bsy\alpha)\big|^2+\big|\nablasl(r^2\bsy\alpha)\big|^2+r^{-2}\big|r^2\bsy\alpha\big|^2\rp)drd\Omega\,,
\end{align*}
for $\bsy\alpha$ and
\begin{align*}
\mathbb{E}_p^1[r\bsy\alphab](\tau)
&:=
\int_{\Sigma_{\tau}}\frac{\Delta}{r^2}\lp(r^{p}\big|{\nablasl}^2(r\bsy\alphab)\big|^2+ r^{-2}(r-M)^2\big|\nablasl_3{\nablasl}(r\bsy\alphab)\big|^2+r^{-6}(r-M)^4\big|\nablasl_3^2(r\bsy\alphab)\big|^2\rp)drd\Omega\\
&\qquad+\int_{\Sigma_{\tau}}\lp(r^{p}\big|\nablasl_4\underline{\mathfrak T}(r\bsy\alphab)\big|^2+\frac{\Delta}{r^4}\big|\nablasl_3\underline{\mathfrak T}( r\bsy\alphab)\big|^2+\big|{\nablasl}\underline{\mathfrak T}( r\bsy\alphab)\big|^2+r^{-2}\big|\underline{\mathfrak T}( r\bsy\alphab)\big|^2\rp)drd\Omega\\
&\qquad+\int_{\Sigma_{\tau}}\lp(r^{-6}(r-M)^2\Delta\big|\nablasl_3(r\bsy\alphab)\big|^2+r^{-2}(r-M)^2\lp(\big|{\nablasl}(r\bsy\alphab)\big|^2+r^{-2}\big|r\bsy\alphab\big|^2\rp)\rp)drd\Omega\,,\\
\overline{\mathbb{E}}_p^1[r\bsy\alphab](\tau)
&:=
\int_{\Sigma_{\tau}}\lp(r^{p}\big|{\nablasl}^2(r\bsy\alphab)\big|^2+r^{-1}\big|{\nablasl}\nablasl_3(r\bsy\alphab)\big|^2+r^{-2}\big|\nablasl_3^2(r\bsy\alphab)\big|^2\rp)drd\Omega\\
&\qquad +
\int_{\Sigma_{\tau}}\lp(r^{p+4}\big|\nablasl_4^2(r\bsy\alphab)\big|^2+\big|\nablasl_3\nablasl_4( r\bsy\alphab)\big|^2+r^{4}\big|{\nablasl}\nablasl_4( r\bsy\alphab)\big|^2+r^{2}\big|\nablasl_4( r\bsy\alphab)\big|^2\rp)drd\Omega\\
&\qquad +
\int_{\Sigma_{\tau}}\lp(r^{-2}\big|\nablasl_3(r\bsy\alphab)\big|^2+\big|{\nablasl}(r\bsy\alphab)\big|^2+r^{-2}\big|r\bsy\alphab\big|^2\rp)drd\Omega
\end{align*}
for $\bsy\alphab$ (recall definitions \eqref{def_quantities_3_alpha_4_alphab}). We note that what distinguishes $\mathbb E_p$ from $\overline{\mathbb E}_p$ is the degeneracy of the former norms at the event horizon (i.e.~at $r=r_+$). We will also consider corresponding higher order energy fluxes
\begin{align*}
\mathbb{E}^{J+1}_p[\bsy \varsigma](\tau)&:= \sum_{J_1+J_2+J_3\leq J}\mathbb{E}_p^1\lp[\lp((1-Mr^{-1})\nablasl_3\rp)^{J_1}\nablasl_4^{J_2}\nablasl^{J_3}\bsy \varsigma\rp](\tau)\,, \\
\overline{\mathbb{E}}^{J+1}_p[\bsy \varsigma](\tau)&:= \sum_{J_1+J_2+J_3\leq J}\overline{\mathbb{E}}_p^1\lp[\nablasl_3^{J_1}\nablasl_4^{J_2}\nablasl^{J_3}\bsy\varsigma\rp](\tau) 
\end{align*}
for $J\geq 0$ and $\bsy\varsigma=\{r^2\bsy\alpha,r\bsy\alphab\}$. By convention, we write  ${\mathbb E}^{J+1}={\mathbb E}^{J+1}_0$ and $\overline{\mathbb E}^{J+1}=\overline{\mathbb E}^{J+1}_0$. Finally, it will be useful to introduce notation for the degenerate (first order) energy flux of $\bsy\alpha$ through $\mc H^+$,
\begin{align*}
\mathbb{E}_{\mc H^+}[r^2\bsy\alpha](\tau_1,\tau_2)=\int_{\tau_1}^{\tau_2}\int_{\mathbb S^2_{ \tilde{t}{}^*,r_+}}\lp(|r^2\bsy \alpha|^2+|{\nablasl}(r^2\bsy \alpha)|^2+|\nablasl_4(r^2\bsy \alpha)|^2\rp) d\Omega d \tilde{t}{}^*
\end{align*}
for any $\tau_1<\tau_2\leq\infty$, and for the non-degenerate (second order) energy fluxes of $\bsy\alpha$ and $\bsy\alphab$ through $\mc H^+$,
\begin{align*}
\overline{\mathbb{E}}^1_{\mc H^+}[r^2\bsy\alpha](\tau_1,\tau_2)&=\int_{\tau_1}^{\tau_2}\int_{\mathbb S^2_{ \tilde{t}{}^*,r_+}}\lp(| \nablasl{\mathfrak T}(r^2\bsy \alpha)|^2+|\nablasl_4{\mathfrak T}(r^2\bsy \alpha)|^2+|{\mathfrak T}(r^2\bsy \alpha)|^2\rp) d\Omega d \tilde{t}{}^*\\
&\qquad +\int_{\tau_1}^{\tau_2}\int_{\mathbb S^2_{ \tilde{t}{}^*,r_+}}\lp(| \nablasl(r^2\bsy \alpha)|^2+|\nablasl_4(r^2\bsy \alpha)|^2+|r^2\bsy \alpha|^2\rp) d\Omega d \tilde{t}{}^*,
\\
\overline{\mathbb{E}}^1_{\mc H^+}[r\bsy\alphab](\tau_1,\tau_2)&=\int_{\tau_1}^{\tau_2}\int_{\mathbb S^2_{ \tilde{t}{}^*,r_+}}\lp(| \nablasl\underline{\mathfrak T}(r\bsy \alphab)|^2+|\nablasl_4\underline{\mathfrak T}(r\bsy \alphab)|^2+|\underline{\mathfrak T}(r\bsy \alphab)|^2\rp) d\Omega d \tilde{t}{}^*\\
&\qquad +\int_{\tau_1}^{\tau_2}\int_{\mathbb S^2_{ \tilde{t}{}^*,r_+}}\lp(| \nablasl(r\bsy \alphab)|^2+|\nablasl_4(r\bsy \alphab)|^2+|r\bsy \alphab|^2\rp) d\Omega d \tilde{t}{}^*
\end{align*}
for any $\tau_1<\tau_2\leq\infty$.

\subsubsection{Relation with spin-weighted energy norms} \label{rmk:translation-SRTdC}

In the $|a|<M$ case, the energy norms above can be compared to those introduced in \cite{SRTdC2023} for spin-weighted quantities. 

Let $\swei{\upalpha}{\pm 1}$ be the $(\pm 1)-$spin-weighted functions obtained from $(\bsy\alpha,\bsy\alphab)$ through the isomorphism of Lemma~\ref{lemma:isomorphism}. We define
\begin{gather*}
\begin{gathered}
\swei{\upphi}{\pm 1}_0 = \frac{\Delta^{\frac{1\pm 1}{2}}}{\sqrt{r^2+a^2}}\swei{\upalpha}{\pm 1}\,, \numberthis \label{eq:def-upphi}
\end{gathered}
\end{gather*}
the rescalings
\begin{align*}
\swei{\tilde \upphi}{+ 1}_0 &:=\sqrt{r^2+a^2}\swei{\upphi}{+ 1}_0 = \Delta^{+1}\swei{\upalpha}{+ 1} = -\frac{1}{\sqrt{2}}\Sigma\bsy\alpha(e_1^{\rm as}+ie_2^{\rm as})\,, \\
 \swei{\tilde \upphi}{- 1}_0 &:=\frac{r^2+a^2}{\Delta}\swei{\upphi}{- 1}_0 = \sqrt{r^2+a^2}\Delta^{-1}\swei{\upalpha}{-1} = \frac{\kappa^2}{\sqrt{2}\Sigma}(r^2+a^2)^{1/2}\bsy\alphab(e_1^{\rm as}-ie_2^{\rm as}) 
\end{align*}
and the so-called \emph{transformed} quantities
\begin{align*}
\begin{split}
\swei{\Phi}{+ 1}&:= (r^2+a^2) e_3^{\rm as} (\swei{\upphi}{+ 1}_0)
=\frac{(r^2+a^2)^{1/2}}{\sqrt{2}}\lp[-\mathfrak{T}(\bsy\alpha)+\lp(\frac{a^2r\sin^2\theta}{r^2+a^2}-ia\cos\theta \rp)\bsy\alpha\rp](e_1^{\rm as}+ie_2^{\rm as})\,,\\
\swei{\Phi}{-1}
&:=\frac{r^2+a^2}{\Delta} \Sigma e_4^{\rm as} (\swei{\upphi}{- 1}_0)
%
%&=\frac{1}{\sqrt{2}}(r^2+a^2)\frac{\Sigma}{\Delta}\nablasl_4\lp[\frac{\Delta}{\Sigma}\frac{\kappa^2}{(r^2+a^2)^{-1/2}}\bsy\alphab\rp](e_1^{\rm as}-ie_2^{\rm as})\\
%&=\frac{1}{\sqrt{2}}(r^2+a^2)\lp[\frac{\kappa^2}{(r^2+a^2)^{1/2}}\nablasl_4\bsy\alphab+\partial_r\lp(\frac{\Delta}{\Sigma}\frac{\kappa^2}{(r^2+a^2)^{-1/2}}\rp)\bsy\alphab\rp](e_1^{\rm as}-ie_2^{\rm as})\\
%&=\frac{1}{\sqrt{2}}(r^2+a^2)^{1/2}\lp[\kappa^2\lp(\nablasl_4\bsy\alphab+\omegah\,\bsy\alphab+\frac12\trchi\,\bsy\alphab\rp)+\frac{\kappa^2\Delta}{\Sigma^2}\lp(\frac{a^2r\sin^2\theta}{(r^2+a^2)}+2ia\cos\theta \rp)\bsy\alphab\rp](e_1^{\rm as}-ie_2^{\rm as})\\
%
=\frac{\kappa^2(r^2+a^2)^{1/2}}{\sqrt{2}\Sigma}\lp[\underline{\mathfrak T}(\bsy\alphab)+\frac{\Delta}{\Sigma}\lp(\frac{a^2r\sin^2\theta}{r^2+a^2}+ia\cos\theta \rp)\bsy\alphab\rp](e_1^{\rm as}-ie_2^{\rm as})\,.
\end{split} \numberthis \label{eq:def-Phi}
\end{align*}
By applying Lemma~\ref{lemma:isomorphism}, one can obtain pointwise estimates (whose implicit constants depend only on $M$) relating the absolute value of the above complex scalars with the pointwise norms of the tensorial quantities $(\bsy\alpha,\bsy\alphab)$. For instance, from the definitions \eqref{eq:map-B-to-NP}, one can compute 
\begin{align*}
\big|\swei{\tilde\upphi}{+1}_0\big|^2 &= |\Delta\swei{\upalpha}{+1}|^2 =\frac12 \lvert\Sigma \bsy\alpha\rvert_{{\slashed g}}^2,\\
\big|\swei{\tilde\upphi}{-1}_0\big|^2 &= |(r^2+a^2)^{1/2}\Delta^{-1}\swei{\upalpha}{-1}|^2 = 
\frac12|\kappa^{-2}\Sigma|^2 \,\lvert (r^2+a^2)^{1/2}\bsy\alphab\rvert_{{\slashed g}}^2 = 
\frac12 \lvert (r^2+a^2)^{1/2}\bsy\alphab\rvert_{{\slashed g}}^2,\\
\big|\swei{\Phi}{-1}\big|^2 &\sim |\underline{\mathfrak T}(r\bsy\alphab)|_{{\slashed g}}^2+\frac{\Delta}{r^4}|r\bsy\alphab|_{{\slashed g}}^2
\end{align*}
and other similar relations. We note that the computations leading to such relations make use of the particular choice of frame vectors \eqref{explicit_e1_e2}, but eventually yield relations which are independent of any such choice.

With the above relations in mind, the degenerate energy norms introduced in \cite{SRTdC2023} can be shown to satisfy (recall definition \eqref{def_L_Lbar})
\begin{align*}
\mathbb{E}_p\big[\swei{\tilde\upphi_0}{+1}\big](\tau)&=\int_{\Sigma_{\tau}}\lp(r^p\big|L\swei{\tilde\upphi_0}{+1}\big|^2+\frac{1}{\Delta}\big|\uL\swei{\tilde\upphi_0}{+1}\big|^2+r^{-2}\big|\mathring{\slashed\nabla}\swei{\tilde\upphi_0}{+1}\big|^2+r^{-2}\big|\swei{\tilde\upphi_0}{+1}\big|^2\rp)drd\Omega\\
&\sim \int_{\Sigma_{\tau}}\lp(r^p\big|\nablasl_4(r^2\bsy\alpha)\big|^2+\frac{\Delta}{r^4}\big|\nablasl_3(r^2\bsy\alpha)\big|^2+\big|{\nablasl}(r^2\bsy\alpha)\big|^2+r^{-2}\big|r^2\bsy\alpha\big|^2\rp)drd\Omega\,,
\\
\mathbb{E}_p\big[\swei{\Phi}{+1}\big](\tau)&=\int_{\Sigma_{\tau}}\lp(r^p\big|L\swei{\Phi}{+1}\big|^2+\frac{1}{\Delta}\big|\uL\swei{\Phi}{+1}\big|^2+r^{-2}\big|\mathring{\slashed\nabla}^{[+1]}\swei{\Phi}{+1}\big|^2+r^{-2}\big|\swei{\Phi}{+1}\big|^2\rp)drd\Omega\\
&\sim \int_{\Sigma_{\tau}}\lp(r^{p+2}\big|\nablasl_4\nablasl_3(r^2\bsy\alpha)\big|^2+\frac{\Delta}{r^2}\big|\nablasl_3^2(r^2\bsy\alpha)\big|^2+r^2\big|{\nablasl}\nablasl_3(r^2\bsy\alpha)\big|^2+\big|\nablasl_3(r^2\bsy\alpha)\big|^2\rp)drd\Omega \\
&\qquad+ \mathbb{E}_p\big[\swei{\tilde\upphi_0}{+1}\big](\tau)
\end{align*}
for spin $+1$ and
\begin{align*}
\mathbb{E}_p\big[\swei{\tilde\upphi_0}{-1}\big](\tau)&:=\int_{\Sigma_{\tau}}\lp(r^p\big|L\swei{\tilde\upphi}{-1}_0\big|^2+\frac{1}{\Delta}\big|\uL\swei{\tilde\upphi_0}{-1}\big|^2+r^{-2}\big|\mathring{\slashed\nabla}^{[-1]}\swei{\tilde\upphi_0}{-1}\big|^2+r^{-2}\big|\swei{\tilde\upphi_0}{-1}\big|^2\rp)drd\Omega\\
&\sim \int_{\Sigma_{\tau}}\lp(r^{p}\big|\nablasl_4(r\bsy\alphab)\big|^2+\frac{\Delta}{r^4}\big|\nablasl_3(r\bsy\alphab)\big|^2+\big|{\nablasl}(r\bsy\alphab)\big|^2+r^{-2}\big|r\bsy\alphab\big|^2\rp)drd\Omega\,,
\\
\mathbb{E}_p\big[\swei{\Phi}{-1}\big](\tau)&:=\int_{\Sigma_{\tau}}\lp(r^p\big|L\swei{\Phi}{-1}\big|^2+\frac{1}{\Delta}\big|\uL\swei{\Phi}{-1}\big|^2+r^{-2}\big|\mathring{\slashed\nabla}^{[-1]}\swei{\Phi}{-1}\big|^2+r^{-2}\big|\swei{\Phi}{-1}\big|^2\rp)drd\Omega\\
&\sim \int_{\Sigma_{\tau}}\lp(r^{p}\big|\nablasl_4\underline{\mathfrak T}(r\bsy\alphab)\big|^2+\frac{\Delta}{r^4}\big|\nablasl_3\underline{\mathfrak T}( r\bsy\alphab)\big|^2+\big|{\nablasl}\underline{\mathfrak T}( r\bsy\alphab)\big|^2+r^{-2}\big|\underline{\mathfrak T}( r\bsy\alphab)\big|^2\rp)drd\Omega\\
&\qquad+ \mathbb{E}_p\big[\swei{\tilde\upphi_0}{-1}\big](\tau)
\end{align*}
for spin $-1$. Here, $\mathring{\slashed\nabla}^{[\pm 1]}$ is the spinorial gradient defined in e.g.~\cite[Section 3]{SRTdC2023}. By adding up, for each spin sign, the two energy norms, one obtains almost every term in the energy norms introduced in Section \ref{sec:Teukolsky-norms} for the tensorial quantities $(\bsy\alpha,\bsy\alphab)$, i.e.\ almost every term in $\mathbb E_p[r^2\bsy\alpha]$ and $\mathbb E_p[r\bsy\alphab]$. The missing terms are obtained by elliptic estimates, as explained in  \cite[Section 4]{SRTdC2023}. To obtain the non-degenerate energy norms $\overline{\mathbb E}_p[r^2\bsy\alpha]$ and $\overline{\mathbb E}_p[r\bsy\alphab]$, one follows a similar procedure: one considers the spin-weighted energy norms $\overline{\mathbb E}_p[\swei{\tilde\upphi}{\pm 1}]$ and $\overline{\mathbb E}_p[\swei{\Phi}{\pm 1}]$ in \cite{SRTdC2023}, and then revisit the elliptic estimates there accordingly.

\begin{remark}
The reader familiar with the standard definitions of energy norms for wave equations may be surprised to find zeroth order terms in the energy norms above. However, we recall from \cite[Section 3]{SRTdC2023} that, by the properties of the spinorial gradient, one has
\begin{equation*}
\int_{\Sigma_{\tau}}\lp(r^{-2}\big|\mathring{\slashed\nabla}^{[s]}\swei{\Phi}{s}\big|^2+\frac{r-M}{r^4}\big|\swei{\Phi}{s}\big|^2\rp)drd\Omega \gtrsim \int_{\Sigma_{\tau}}\lp(r^{-2}\big|\mathring{\slashed\nabla}^{[s]}\swei{\Phi}{s}\big|^2+r^{-2}\lp(s^2+\frac{r-M}{r^2}\rp)\big|\swei{\Phi}{s}\big|^2\rp)drd\Omega
\end{equation*}
for any $s$-spin-weighted function $\swei{\Phi}{s}$, with $s\in\frac12\mathbb Z$.
\end{remark}

\subsection{Middle Maxwell components}
\label{sec:middle-norms}

Let $\bsy f=\left\lbrace \widehat{\bsy\rho}, \widehat{\bsy\sigma} \right\rbrace$. We define the first order energy fluxes
\begin{align*}
\dot{\mathbb{E}}[\bsy f](\tau)&:=\int_{\Sigma_\tau}\lp[|\nablasl_4\bsy f|^2+\Delta r^{-4}|\nablasl_3\bsy f|^2+|{\nablasl}\bsy f|^2\rp] dr d\Omega\,,\\
\dot{\overline{\mathbb{E}}}[\bsy f](\tau)&:=\int_{\Sigma_\tau}\lp[|\nablasl_4\bsy f|^2+r^{-2}|\nablasl_3{\bsy f}\,|^2+|{\nablasl}\bsy f|^2\rp]dr d\Omega\,,
\end{align*}
where the overline indicates the non-degeneracy of the energy norm at the event horizon (i.e.~at $r=r_+$), and corresponding higher order energy fluxes 
\begin{align*}
\dot{\mathbb{E}}^{J}[\bsy f](\tau)&:= \sum_{J_1+J_2+J_3\leq J}\dot{\mathbb{E}}\lp[\lp((1-Mr^{-1})\nablasl_3\rp)^{J_1}\nablasl_4^{J_2}\nablasl^{J_3}\bsy f\rp](\tau)\,, \\
\dot{\overline{\mathbb{E}}}^{J}[\bsy f](\tau)&:= \sum_{J_1+J_2+J_3\leq J}\dot{\overline{\mathbb{E}}}\lp[\nablasl_3^{J_1}\nablasl_4^{J_2}\nablasl^{J_3}\bsy f\rp](\tau) 
\end{align*}
for $J\geq 0$. We note that $\dot{\mathbb{E}}^{0}=\dot{\mathbb{E}}$ and $\dot{\overline{\mathbb{E}}}^{0}= \dot{\overline{\mathbb{E}}}$. The dot indicates that the energy norm does \emph{not} contain a zeroth order term. It will also be convenient to define the energy fluxes
\begin{align*}
\mathbb{E}^J[\bsy f](\tau)&:=\dot{\mathbb{E}}^{J}[\bsy f](\tau)+\int_{ \Sigma_{\tau}}r^{-2}|\bsy f|^2 drd\Omega\,,\\
\overline{\mathbb{E}}^J[\bsy f](\tau)&:=\dot{\overline{\mathbb{E}}}^{J}[\bsy f](\tau)+\int_{ \Sigma_{\tau}}r^{-2}|\bsy f|^2drd\Omega\,,
\end{align*}
which contain a zeroth order term.

\section{Statement of the main theorems} \label{sec:statement_theorems}

In this section, we provide a precise statement of our main theorems. All results are stated for the modified Maxwell equations \eqref{D3_sigma}--\eqref{DA_rho}. The reader can easily retrieve analogous statements for the original Maxwell equations \eqref{eq:del3-sigma}--\eqref{eq:del4-ualpha}.

The first theorem proven in the present paper concerns energy boundedness and decay for the Maxwell components in the sub-extremal $|a|<M$ case. As we shall explain in Section \ref{sec:analysis-teukolsky}, the energy boundedness and decay statements for the \emph{extremal} Maxwell components stated in Theorem \ref{thm:precise-EB-ED-sub} are a direct translation of the results obtained in \cite{SRTdC2020,SRTdC2023}.

\begin{theorem}[Sub-extremal case] \label{thm:precise-EB-ED-sub} Fix $M>0$ and $a_0\in[0,M)$. Let $J\geq 1$ and $\eta\in(0,1)$. Then, there exists a uniform constant $C=C(a_0,M,J,\eta)>0$ such that, for any $|a|\leq a_0$ and any solution 
$$\bsy{\mathfrak S}=(\bsy{\alpha}[\bsy{\mathfrak S}],\bsy{\alphab}[\bsy{\mathfrak S}],\widehat{\bsy{\rho}}[\bsy{\mathfrak S}],\widehat{\bsy{\sigma}}[\bsy{\mathfrak S}])$$ 
to the modified Maxwell equations \eqref{D3_sigma}--\eqref{DA_rho} arising from seed initial data prescribed on $\Sigma_0$, the following energy estimates hold.
\begin{itemize}
\item Non-degenerate energy boundedness: if $J\geq 3$, then
\begin{align*}
\overline{\mathbb{E}}^{J}\big[r^2\bsy \alpha[\bsy{\mathfrak S}]\big](\tau)&\leq C\,\overline{\mathbb{E}}^{J}\big[r^2\bsy\alpha [\bsy{\mathfrak S}]\big](0)\,,\\
\overline{\mathbb{E}}^{J}\big[r\bsy \alphab[\bsy{\mathfrak S}]\big](\tau)&\leq C\,\overline{\mathbb{E}}^{J}\big[r\bsy\alphab[\bsy{\mathfrak S}]\big](0)\,,\\
 \overline{\mathbb{E}}^{J}\big[\widehat{\bsy\rho}[\bsy{\mathfrak S}]\big](\tau)+\overline{\mathbb{E}}^{J}\big[\widehat{\bsy\sigma}[\bsy{\mathfrak S}]\big](\tau) 
& \leq C \Big(\lVert|\widehat{\bsy \rho}[\bsy{\mathfrak S}]\rVert^2_{L^2(\mathbb{S}^2_{0,r_+})}+\lVert\widehat{\bsy \sigma}[\bsy{\mathfrak S}]\rVert^2_{L^2(\mathbb{S}^2_{0,r_+})}
\Big)\\
&\qquad+ C\lp(\overline{\mathbb{E}}_{2-\eta}^{3}\big[r^2{\bsy\alpha}[\bsy{\mathfrak S}]\big](0)+\overline{\mathbb{E}}^{J}\big[r^2{\bsy\alpha}[\bsy{\mathfrak S}]\big](0)+\overline{\mathbb{E}}^{J}\big[r{\bsy\alphab}[\bsy{\mathfrak S}]\big](0)\rp)
\end{align*}
for all $\tau\geq 0$.
\item  Non-degenerate energy decay (extremal Maxwell components): we have
\begin{align*}
    \overline{\mathbb{E}}^{J}\big[r^2\bsy \alpha [\bsy{\mathfrak S}]\big](\tau)&\leq \frac{C}{(1+\tau)^{2-\eta}}\,\overline{\mathbb{E}}_{2-\eta}^{J+2}\big[r^2\bsy \alpha [\bsy{\mathfrak S}]\big](0)\,,\\
\overline{\mathbb{E}}^J\big[r\bsy \alphab[\bsy{\mathfrak S}]\big](\tau)
&\leq \frac{C}{(1+\tau)^{2-\eta}}\,\overline{\mathbb{E}}^{J+2}_{2-\eta}\big[ r\bsy\alphab[\bsy{\mathfrak S}]\big](0)
\end{align*}
for all $\tau\geq 0$.
\item  Non-degenerate energy decay (middle Maxwell components): there exist finite $\mathfrak{q}_1,\mathfrak{q}_2\in\mathbb{R}$, which can be read off from the seed initial data, such that the normalised solution $\widetilde{\bsy{\mathfrak S}}=\bsy{\mathfrak S}-\bsy{\mathfrak S}_{\mathfrak{q}_1,\mathfrak{q}_2}$ satisfies 
\begin{align*}
\bsy \alpha [\widetilde{\bsy{\mathfrak S}}](\tau)&=\bsy \alpha [\bsy{\mathfrak S}](\tau) \, , & \bsy \alphab[\widetilde{\bsy{\mathfrak S}}](\tau)&=\bsy \alphab[\bsy{\mathfrak S}](\tau)
\end{align*}
and
\begin{equation*}
\overline{\mathbb{E}}^{J}\big[\widehat{\bsy\rho}\big[\widetilde{\bsy{\mathfrak S}}\big]\big](\tau)+\overline{\mathbb{E}}^{J}\big[\widehat{\bsy\sigma}\big[\widetilde{\bsy{\mathfrak S}}\big]\big](\tau) \leq \frac{C}{(1+\tau)^{1-\eta}}\lp(\overline{\mathbb{E}}_{2-\eta}^{J+3}\big[r^2{\bsy\alpha}[\bsy{\mathfrak S}]\big](0)+\overline{\mathbb{E}}_{2-\eta}^{J+3}\big[r{\bsy\alphab}[\bsy{\mathfrak S}]\big](0)\rp) 
\end{equation*}
for all $\tau\geq 0$.
\end{itemize}
\end{theorem}

\begin{remark}
    The energy norms appearing on the left hand side of the energy estimates in Theorem \ref{thm:precise-EB-ED-sub} are non-degenerate at the event horizon (i.e.~at $r=r_+$).
\end{remark}

\begin{remark} In Theorem~\ref{thm:precise-EB-ED-sub}, the energy boundedness statement does not lose derivatives. In that sense, it provides a true orbital stability statement, comparable to those obtained for $|a|\ll M$ in \cite{Andersson2015a,Blue2008,Pasqualotto2016}. 
\end{remark}

\begin{remark}
The energy decay statement of Theorem~\ref{thm:precise-EB-ED-sub} can be improved, provided that higher regularity and stronger $r$-decay are assumed on the initial data. Sharp decay estimates in this form are obtained in \cite{Ma2022}.
\end{remark}

\begin{remark}
We expect the loss in $r$-weights of Theorem~\ref{thm:precise-EB-ED-sub} to be suboptimal. Treating the middle Maxwell components as solutions to a wave equation (see Appendix~\ref{app:FI-equations}) would lead to improved $r$-weights.
\end{remark}

In Theorem~\ref{thm:precise-EB-ED-sub}, we have $C(a_0)\to \infty$ as $a_0\to M$, and thus the energy boundedness and decay statements break down towards extremality. As we shall explain in Section \ref{sec:analysis-teukolsky}, such a breakdown emerges from the breakdown of energy boundedness and decay statements for the \emph{extremal} Maxwell components. Nevertheless, it is expected that energy boundedness and (weaker) decay for the extremal Maxwell components still hold in the nearly extremal $|a|\approx M$ case, at least under some additional assumptions on the initial data. More precisely, one may conjecture:

\begin{conjecture}[Sub-extremal and extremal cases] \label{conj:teukolsky-EB-ILED-ext} Fix $M>0$. Let $\eta\in(0,1)$ and $m\in \mathbb{Z}$. Then, there exist $J_2,J_3\geq J_{1}\geq J_{\textup{min}}\geq 1$ and a uniform constant $C=C(m, M,J_1,J_2,J_3,\eta)>0$ such that, for any $|a|\leq M$ and any solution $$\bsy{\mathfrak S}=(\bsy{\alpha}[\bsy{\mathfrak S}],\bsy{\alphab}[\bsy{\mathfrak S}],\widehat{\bsy{\rho}}[\bsy{\mathfrak S}],\widehat{\bsy{\sigma}}[\bsy{\mathfrak S}])$$ to the modified Maxwell equations \eqref{D3_sigma}--\eqref{DA_rho} arising from seed initial data supported on a fixed azimuthal mode $m$ prescribed on $\Sigma_0$, the following energy estimates hold.
\begin{itemize}
\item Degenerate energy boundedness: we have
\begin{align*}
{\mathbb{E}}^{J_1}\big[r^2\bsy \alpha[\bsy{\mathfrak S}]\big](\tau)&\leq C\,{\mathbb{D}}_{\mathrm{ bdd}}^{J_2}\big[r^2\bsy{\alpha}[\bsy{\mathfrak S}]\big](0)\,,\\
 {\mathbb{E}}^{J_1}\big[r\bsy \alphab[\bsy{\mathfrak S}]\big](\tau)&\leq C\,{\mathbb{D}}_{\mathrm{ bdd}}^{J_2}\big[r\bsy\alphab[\bsy{\mathfrak S}]\big](0)
\end{align*}
for all $\tau\geq 0$.
\item  Degenerate energy decay: we have
\begin{align*}
{\mathbb{E}}^{J_1}\big[r^2\bsy \alpha [\bsy{\mathfrak S}]\big](\tau)&\leq \frac{C}{(1+\tau)^{1-\eta}}\,\mathbb{D}_{\mathrm{dec}}^{J_3}\big[r^2\bsy \alpha [\bsy{\mathfrak S}]\big](0)\,,\\
{\mathbb{E}}^{J_1}\big[r\bsy \alphab[\bsy{\mathfrak S}]\big](\tau)&\leq \frac{C}{(1+\tau)^{1-\eta}}\,\mathbb{D}_{\mathrm{dec}}^{J_3}\big[r\bsy\alphab[\bsy{\mathfrak S}]\big](0)
\end{align*}
for all $\tau\geq 0$.
\item  Decay at the event horizon: we have
\begin{align*}
{\mathbb{E}}_{\mc H^+}\big[r^2\bsy \alpha[\bsy{\mathfrak S}]\big](\tau,\infty)&\leq \frac{C}{(1+\tau)^{2-\eta}}\,\mathbb{D}_{\mathrm{dec}}^{J_3}\big[r^2\bsy \alpha [\bsy{\mathfrak S}]\big](0)\,,\\
\int_{\mathbb{S}^2_{\tau,r_+}}\lp(|\nablasl{\mathfrak T}(r^2\bsy\alpha)|^2+|{\mathfrak T}(r^2\bsy\alpha)|^2\rp) d\Omega &\leq \frac{C}{(1+\tau)^{1-\eta}}\,\mathbb{D}_{\mathrm{dec}}^{J_3}\big[r^2\bsy \alpha [\bsy{\mathfrak S}]\big](0)\,,\\
\int_{\mathbb{S}^2_{\tau,r_+}}\lp(|\nablasl\underline{\mathfrak T}(r\bsy\alphab)|^2+|\underline{\mathfrak T}(r\bsy\alphab)|^2\rp) d\Omega &\leq \frac{C}{(1+\tau)^{1-\eta}}\,\mathbb{D}_{\mathrm{dec}}^{J_3}\big[r\bsy \alphab [\bsy{\mathfrak S}]\big](0)\,,
\end{align*}
for all $\tau\geq 0$.
\end{itemize}
In the above, $\mathbb{D}_{\mathrm{bdd}}^{J}$ and $\mathbb{D}_{\mathrm{dec}}^{J}$ are $J$-th order energy quantities depending only on the seed initial data. We also recall definition \eqref{def_quantities_3_alpha_4_alphab}.
\end{conjecture}

\begin{remark}
    The energy norms appearing on the left hand side of the energy estimates in Conjecture \ref{conj:teukolsky-EB-ILED-ext} are degenerate at the event horizon (i.e.~at $r=r_+$).
\end{remark}

\begin{remark} \label{rmk:SRTdC20-extremal}
We note that Conjecture~\ref{conj:teukolsky-EB-ILED-ext} is a uniform-in-$a$ statement in the full $|a|\leq M$ range. Very few such results exist in the literature at the present time. To the best of our knowledge, the only uniform-in-$a$ statements in the full $|a|\leq M$ range were shown in \cite{SRTdC2020} for the radial ODE obtained from the Teukolsky equation \eqref{eq:Teukolsky-spin} by (spacetime) frequency decomposition, but only in the case where the frequency parameters are sufficiently far from the so-called superradiant threshold. A lot of work in the physics and mathematics literature from recent years has focused on the \emph{exactly} extremal $|a|=M$ case. Conjecture~\ref{conj:teukolsky-EB-ILED-ext} is formulated based on this literature, which we briefly review in Section \ref{sec:teukoslsky-extremal}.
\end{remark}

The second theorem proven in the present paper concerns energy boundedness and decay for the \emph{middle} Maxwell components in the full $|a|\leq M$ black hole parameter range and assumes Conjecture~\ref{conj:teukolsky-EB-ILED-ext}. The energy estimates for the extremal Maxwell components of Conjecture~\ref{conj:teukolsky-EB-ILED-ext} are embedded in the statement below and play a crucial role in the proof of the energy estimates for the middle Maxwell components.

\begin{theorem}[Sub-extremal and extremal cases] \label{thm:precise-EB-ED-ext} Fix $M>0$. Let $\eta\in(0,1)$ and $m\in\mathbb{Z}$. Assume that Conjecture~\ref{conj:teukolsky-EB-ILED-ext} holds. Then, there exist $J_2,J_3\geq J_{1}\geq J_{\textup{min}}\geq 1$ and a uniform constant $C=C(m, M,J_1,J_2,J_3,\eta)>0$ such that, for any $|a|\leq M$ and any solution $$\bsy{\mathfrak S}=(\bsy{\alpha}[\bsy{\mathfrak S}],\bsy{\alphab}[\bsy{\mathfrak S}],\widehat{\bsy{\rho}}[\bsy{\mathfrak S}],\widehat{\bsy{\sigma}}[\bsy{\mathfrak S}])$$ to the modified Maxwell equations \eqref{D3_sigma}--\eqref{DA_rho} arising from seed initial data supported on a fixed azimuthal mode $m$ prescribed on $\Sigma_0$, the following energy estimates hold.
\begin{itemize}
\item Degenerate energy boundedness: we have
\begin{align*}
{\mathbb{E}}^{J_1}\big[r^2\bsy \alpha[\bsy{\mathfrak S}]\big](\tau)&\leq C\,{\mathbb{D}}_{\mathrm{ bdd}}^{J_2}\big[r^2\bsy{\alpha}[\bsy{\mathfrak S}]\big](0)\,,\\
{\mathbb{E}}^{J_1}\big[r\bsy \alphab[\bsy{\mathfrak S}]\big](\tau)&\leq C\,{\mathbb{D}}_{\mathrm{ bdd}}^{J_2}\big[r\bsy\alphab[\bsy{\mathfrak S}]\big](0)\,,\\
\mathbb{E}^{J_1}\big[\widehat{\bsy\rho}[\bsy{\mathfrak S}]\big](\tau)+\mathbb{E}^{J_1}\big[\widehat{\bsy\sigma}[\bsy{\mathfrak S}]\big](\tau) 
& \leq C \Big(\lVert|\widehat{\bsy \rho}[\bsy{\mathfrak S}]\rVert^2_{L^2(\mathbb{S}^2_{0,r_+})}+\lVert\widehat{\bsy \sigma}[\bsy{\mathfrak S}]\rVert^2_{L^2(\mathbb{S}^2_{0,r_+})}
\Big)\\
& \qquad+ C\lp({\mathbb{D}}_{\mathrm{ dec}}^{J_3}\big[r^2{\bsy\alpha}[\bsy{\mathfrak S}]\big](0)+{\mathbb{D}}_{\mathrm{ bdd}}^{J_2}\big[r^2{\bsy\alpha}[\bsy{\mathfrak S}]\big](0)\rp) \\
& \qquad+ C\lp({\mathbb{D}}_{\mathrm{ dec}}^{J_3}\big[r{\bsy\alphab}[\bsy{\mathfrak S}]\big](0)+{\mathbb{D}}_{\mathrm{ bdd}}^{J_2}\big[r{\bsy\alphab}[\bsy{\mathfrak S}]\big](0)\rp) 
\end{align*}
for all $\tau\geq 0$.
\item Degenerate energy decay (extremal Maxwell components): we have
\begin{align*}
    {\mathbb{E}}^{J_1}\big[r^2\bsy \alpha[\widetilde{\bsy{\mathfrak S}}]\big](\tau)&\leq \frac{C}{(1+\tau)^{1-\eta}}\,\mathbb{D}_{\mathrm{dec}}^{J_3}\big[r^2\bsy \alpha [\bsy{\mathfrak S}]\big](0)\,,\\
{\mathbb{E}}^{J_1}\big[r\bsy \alphab[\widetilde{\bsy{\mathfrak S}}]\big](\tau)&\leq \frac{C}{(1+\tau)^{1-\eta}}\,\mathbb{D}_{\mathrm{dec}}^{J_3}\big[r\bsy\alphab[\bsy{\mathfrak S}]\big](0)
\end{align*}
for all $\tau\geq 0$.
\item Degenerate energy decay (middle Maxwell components): there exist finite $\mathfrak{q}_1,\mathfrak{q}_2\in\mathbb{R}$, which can be read off from the seed initial data, such that the normalised solution $\widetilde{\bsy{\mathfrak S}}=\bsy{\mathfrak S}-\bsy{\mathfrak S}_{\mathfrak{q}_1,\mathfrak{q}_2}$ satisfies
\begin{align*}
    \bsy \alpha [\widetilde{\bsy{\mathfrak S}}](\tau)&=\bsy \alpha [\bsy{\mathfrak S}](\tau) \, , & \bsy \alphab[\widetilde{\bsy{\mathfrak S}}](\tau)&=\bsy \alphab[\bsy{\mathfrak S}](\tau)
\end{align*}
and
\begin{align*}
\mathbb{E}^{J_1}\big[\widehat{\bsy\rho}[\widetilde{\bsy{\mathfrak S}}]\big](\tau)+\mathbb{E}^{J_1}\big[\widehat{\bsy\sigma}\big[\widetilde{\bsy{\mathfrak S}}]\big](\tau) &\leq \frac{C}{(1+\tau)^{1-\eta}} \lp({\mathbb{D}}_{\mathrm{ dec}}^{J_3}\big[r^2{\bsy\alpha}[\bsy{\mathfrak S}]\big](0)+{\mathbb{D}}_{\mathrm{ bdd}}^{J_2}\big[r^2{\bsy\alpha}[\bsy{\mathfrak S}]\big](0) \rp)\\
&\qquad +\frac{C}{(1+\tau)^{1-\eta}} \lp( {\mathbb{D}}_{\mathrm{ dec}}^{J_3}\big[r{\bsy\alphab}[\bsy{\mathfrak S}]\big](0)+{\mathbb{D}}_{\mathrm{ bdd}}^{J_2}\big[r{\bsy\alphab}[\bsy{\mathfrak S}]\big](0)\rp)
\end{align*}
for all $\tau\geq 0$.
\end{itemize}
In the above, $\mathbb{D}_{\mathrm{bdd}}^{J}$ and $\mathbb{D}_{\mathrm{dec}}^{J}$ are the $J$-th order energy quantities introduced in Conjecture \ref{conj:teukolsky-EB-ILED-ext}.
\end{theorem}

\begin{remark}
    As for Conjecture \ref{conj:teukolsky-EB-ILED-ext}, and contrary to Theorem \ref{thm:precise-EB-ED-sub}, the energy norms appearing on the left hand side of the energy estimates in Theorem \ref{thm:precise-EB-ED-ext} are degenerate at the event horizon (i.e.~at $r=r_+$). Moreover, we note that Theorem \ref{thm:precise-EB-ED-ext} is a uniform-in-$a$ statement in the full $|a|\leq M$ range.
\end{remark}

The proof of Theorems \ref{thm:precise-EB-ED-sub} and \ref{thm:precise-EB-ED-ext} is carried out in Section \ref{sec:proof_main_theorems} and relies on the separate analysis of the middle and extremal Maxwell components, which is presented in Sections \ref{sec:analysis-middle} and \ref{sec:analysis-teukolsky} respectively.

%-------------------------------------------------------------------
%
% ANALYSIS OF THE GAUGE-DEPENDENT PART
%
%-------------------------------------------------------------------
%\newpage
\section{Analysis of the middle Maxwell components}
\label{sec:analysis-middle}

In this section, we prove energy estimates for the middle Maxwell components in terms of energy norms of the extremal Maxwell components. The ones proven in this section are standalone statements. Once combined with the statements for the extremal Maxwell components of Section \ref{sec:analysis-teukolsky}, the energy estimates proven in this section establish energy boundedness and decay for the middle Maxwell components (hence the terminology ``boundedness" and ``decay" already appearing in Theorems \ref{thm:middle-EB-ED-sub} and \ref{thm:middle-EB-ED-ext} below and throughout the section). The energy estimates for the middle Maxwell components are obtained in a unified fashion for the full $|a|\leq M$ range.

We state the first theorem of the section, which concerns non-degenerate-at-$\mc H^+$ energy norms.

\begin{theorem} \label{thm:middle-EB-ED-sub} Fix $M>0$. Let $J\geq 1$ and $p_\alpha>0$. Assume that, for any $|a|\leq M$ and any solution 
$$\bsy{\mathfrak S}=(\bsy{\alpha}[\bsy{\mathfrak S}],\bsy{\alphab}[\bsy{\mathfrak S}],\widehat{\bsy{\rho}}[\bsy{\mathfrak S}],\widehat{\bsy{\sigma}}[\bsy{\mathfrak S}])$$ 
to the modified Maxwell equations \eqref{D3_sigma}--\eqref{DA_rho} arising from seed initial data prescribed on $\Sigma_0$, the qualitative statement
\begin{align*}
\lim_{\tau\to \infty}\int_{\mathbb{S}^2_{\tau,r_+}}|\underline{\mathfrak{T}}(\bsy\alphab[\bsy{\mathfrak S}])|^2 d\Omega = 0
\end{align*}
and the quantitative estimate
\begin{align*}
\mathbb E_{\mc H^+}[r^2\bsy\alpha[\bsy{\mathfrak S}]](\tau,\infty)\leq \frac{\mathbb D[r^2\bsy\alpha[\bsy{\mathfrak S}]](0)}{(1+\tau)^{1+p_{\alpha}}}
\end{align*}
hold, where $\mathbb D$ is an energy quantity depending only on the seed initial data. Then, there exists a uniform constant $C=C(M,J,p_{\alpha})>0$ such that the following energy estimates hold.
\begin{itemize}
\item Non-degenerate energy boundedness: we have 
\begin{align*}
\overline{\mathbb{E}}^J[\widehat{\bsy \rho}[\bsy{\mathfrak S}]](\tau ) +\overline{\mathbb{E}}^J[\widehat{\bsy \sigma}[\bsy{\mathfrak S}]](\tau ) &\leq C\lp(\lVert \widehat{\bsy\rho}[\bsy{\mathfrak S}]\rVert^2_{L^2({\mathbb{S}^2_{0,r_+}})}+\lVert \widehat{\bsy\sigma}[\bsy{\mathfrak S}]\rVert^2_{L^2({\mathbb{S}^2_{0,r_+}})}\rp)\\
&\qquad +C\lp(  \mathbb{D}[r^2\bsy\alpha[\bsy{\mathfrak S}]](0)  + \overline{\mathbb{E}}^{\max\{J,1\}}[r^2\bsy\alpha[\bsy{\mathfrak S}]](\tau)+\overline{\mathbb{E}}^{\max\{J,1\}}[r\bsy\alphab[\bsy{\mathfrak S}]](\tau)\rp)
\end{align*}
for all $\tau\geq 0$.
\item Non-degenerate energy decay: there exist finite $\mathfrak q_1,\mathfrak q_2\in\mathbb{R}$, which can be read off from the seed initial data, such that the normalised solution $\widetilde{\bsy{\mathfrak S}}=\bsy{\mathfrak S}-\bsy{\mathfrak S}_{\mathfrak q_1,\mathfrak q_2}$ satisfies
\begin{align*}
    \bsy \alpha [\widetilde{\bsy{\mathfrak S}}](\tau)&=\bsy \alpha [\bsy{\mathfrak S}](\tau) \, , & \bsy \alphab[\widetilde{\bsy{\mathfrak S}}](\tau)&=\bsy \alphab[\bsy{\mathfrak S}](\tau)
\end{align*}
and
\begin{align*}
\overline{\mathbb{E}}^J[\widehat{\bsy \rho}[\widetilde{\bsy{\mathfrak S}}]](\tau )+\overline{\mathbb{E}}^J[\widehat{\bsy \sigma}[\widetilde{\bsy{\mathfrak S}}]](\tau )  &\leq  \frac{C\,\mathbb{D}[r^2\bsy\alpha[\bsy{\mathfrak S}]](0)}{(1+\tau)^{p_\alpha}}  + C\lp(\overline{\mathbb{E}}^{\max\{J,1\}}[r^2\bsy\alpha[\bsy{\mathfrak S}]](\tau)+\overline{\mathbb{E}}^{\max\{J,1\}}[r\bsy\alphab[\bsy{\mathfrak S}]](\tau)\rp)
\end{align*}
for all $\tau\geq 0$.
\end{itemize}
\end{theorem}

We state the second theorem of the section, which concerns degenerate-at-$\mc H^+$ energy norms.

\begin{theorem} \label{thm:middle-EB-ED-ext} Fix $M>0$. Let $J\geq 1$ and $p_\alpha>0$. Assume that, for any $|a|\leq M$ and any solution 
$$\bsy{\mathfrak S}=(\bsy{\alpha}[\bsy{\mathfrak S}],\bsy{\alphab}[\bsy{\mathfrak S}],\widehat{\bsy{\rho}}[\bsy{\mathfrak S}],\widehat{\bsy{\sigma}}[\bsy{\mathfrak S}])$$ 
to the modified Maxwell equations \eqref{D3_sigma}--\eqref{DA_rho} arising from seed initial data prescribed on $\Sigma_0$, the qualitative statement
\begin{align*}
\lim_{\tau\to \infty}\int_{\mathbb{S}^2_{\tau,r_+}}|\underline{\mathfrak{T}}(\bsy\alphab[\bsy{\mathfrak S}])|^2 d\Omega = 0
\end{align*}
and the quantitative estimate
\begin{align*}
\mathbb E_{\mc H^+}[r^2\bsy\alpha[\bsy{\mathfrak S}]](\tau,\infty)\leq \frac{\mathbb D[r^2\bsy\alpha[\bsy{\mathfrak S}]](0)}{(1+\tau)^{1+p_{\alpha}}}
\end{align*}
hold, where $\mathbb D$ is an energy quantity depending only on the seed initial data. Then, there exists a uniform constant $C=C(M,J,p_{\alpha})>0$ such that the following energy estimates hold.
\begin{itemize}
\item Degenerate energy boundedness: we have 
\begin{align*}
\mathbb{E}^J[\widehat{\bsy \rho}[\bsy{\mathfrak S}]](\tau ) +\mathbb{E}^J[\widehat{\bsy \sigma}[\bsy{\mathfrak S}]](\tau ) &\leq 
C\lp(\lVert \widehat{\bsy\rho}[\bsy{\mathfrak S}]\rVert^2_{L^2({\mathbb{S}^2_{0,r_+}})}+\lVert \widehat{\bsy\sigma}[\bsy{\mathfrak S}]\rVert^2_{L^2({\mathbb{S}^2_{0,r_+}})}\rp)\\
&\qquad +C\int_{\mathbb S^2_{\tau,r_+}}\lp(|\nablasl\underline{\mathfrak{T}}(\bsy\alphab[\bsy{\mathfrak S}])|^2+|\underline{\mathfrak{T}}(\bsy\alphab[\bsy{\mathfrak S}])|^2\rp)d\Omega+C\,\mathbb{E}^{\max\{J,2\}}[r\bsy\alphab[\bsy{\mathfrak S}]](\tau)\\
&\qquad + C\,\mathbb{D}[r^2\bsy\alpha[\bsy{\mathfrak S}]](0) + C\,\mathbb{E}^{\max\{J,1\}}[r^2\bsy\alpha[\bsy{\mathfrak S}]](\tau)
\end{align*}
for all $\tau\geq 0$.
\item Degenerate energy decay: there exist finite $\mathfrak q_1,\mathfrak q_2\in\mathbb{R}$, which can be read off from the seed initial data, such that the normalised solution $\widetilde{\bsy{\mathfrak S}}=\bsy{\mathfrak S}-\bsy{\mathfrak S}_{\mathfrak q_1,\mathfrak q_2}$ satisfies
\begin{align*}
    \bsy \alpha [\widetilde{\bsy{\mathfrak S}}](\tau)&=\bsy \alpha [\bsy{\mathfrak S}](\tau) \, , & \bsy \alphab[\widetilde{\bsy{\mathfrak S}}](\tau)&=\bsy \alphab[\bsy{\mathfrak S}](\tau)
\end{align*}
and
\begin{align*}
\mathbb{E}^J[\widehat{\bsy \rho}[\widetilde{\bsy{\mathfrak S}}]](\tau ) +\mathbb{E}^J[\widehat{\bsy \sigma}[\widetilde{\bsy{\mathfrak S}}]](\tau ) &\leq C\int_{\mathbb S^2_{\tau,r_+}}\lp(|\nablasl\underline{\mathfrak{T}}(\bsy\alphab[\bsy{\mathfrak S}])|^2+|\underline{\mathfrak{T}}(\bsy\alphab[\bsy{\mathfrak S}])|^2\rp)d\Omega+C\,\mathbb{E}^{\max\{J,2\}}[r\bsy\alphab[\bsy{\mathfrak S}]](\tau)\\
&\qquad + \frac{C\,\mathbb{D}[r^2\bsy\alpha[\bsy{\mathfrak S}]](0)}{(1+\tau)^{p_\alpha}} + C\,\mathbb{E}^{\max\{J,1\}}[r^2\bsy\alpha[\bsy{\mathfrak S}]](\tau)
\end{align*}
for all $\tau\geq 0$.
\end{itemize}
\end{theorem}

Although both theorems are stated for the full $|a|\leq M$ range, we note that, once combined with the statements for the extremal Maxwell components of Section \ref{sec:analysis-teukolsky}, Theorem~\ref{thm:middle-EB-ED-sub} will lead to (non-degenerate) energy boundedness and decay for all Maxwell components in the sub-extremal $|a|<M$ case (i.e.~Theorem \ref{thm:precise-EB-ED-sub}), whereas Theorem~\ref{thm:middle-EB-ED-ext} will lead to (degenerate) energy boundedness and decay for all Maxwell components in the full $|a|\leq M$ black hole parameter range (i.e.~Theorem \ref{thm:precise-EB-ED-ext}).

Sections \ref{sec:middle-zeroth-order}--\ref{sec:conclusion-middle-EB-ED-ext} below contain the proof of Theorems~\ref{thm:middle-EB-ED-sub} and \ref{thm:middle-EB-ED-ext}, which is divided as follows:
\begin{itemize}
\item \textbf{Section~\ref{sec:middle-zeroth-order}}:~We prove estimates for the zeroth order energy flux (i.e.~$L^2(\Sigma_{\tau})$-norm) of the middle Maxwell components. These estimates are achieved in Proposition \ref{prop:L2-norm}. To this end, one needs control over the zero mode of the middle Maxwell components (see Lemmas \ref{lemma:zero-mode-horizon-oldsystem} and \ref{lemma:zero-mode-horizon}), which is obtained by applying the $\mathbb{S}^2$-projection procedure of Section \ref{sec:projection_formulae} and exploiting the enhanced structure of the algebraically special frame at the event horizon.
\item \textbf{Section~\ref{sec:middle-higher-order}}:~We establish estimates for first (and higher) order energy fluxes of the middle Maxwell components. These estimates are immediately read off from the Maxwell equations \eqref{D3_sigma}--\eqref{DA_rho}.
\item \textbf{Section~\ref{sec:conclusion-middle-EB-ED-sub}}:~We combine the estimates of Sections \ref{sec:middle-zeroth-order}--\ref{sec:middle-higher-order} to prove Theorem \ref{thm:middle-EB-ED-sub}. The reader interested only in the sub-extremal $|a|<M$ case need not read beyond this section.
\item \textbf{Section~\ref{sec:conclusion-middle-EB-ED-ext}}:~We prove Theorem \ref{thm:middle-EB-ED-ext} by combining the estimates of Sections \ref{sec:middle-zeroth-order}--\ref{sec:middle-higher-order} with an additional pointwise identity (in the spin-weighted formalism) for transversal derivatives of the middle Maxwell components. As a warm-up, the additional identity is first derived in the Schwarzschild (and Reissner--Nordstr\"om) case (see Section \ref{sec:warm-up-spherically-symmetric}). The derivation of the identity in the Kerr case is carried out in a sequence of three lemmas in Section \ref{sec:three_key_lemmas} (see the final Lemma \ref{lemma:Maxwell-n-identity} for a statement of the identity). In Section \ref{sec:combining_extremal_lemmas}, we employ the identity to prove Theorem \ref{thm:middle-EB-ED-ext}.
\end{itemize}

Additionally, we have:
\begin{itemize}
\item \textbf{Section~\ref{rmk:extremal-axisym}}:~Besides yielding a proof of Theorem~\ref{thm:middle-EB-ED-ext}, the additional identity derived in Section \ref{sec:conclusion-middle-EB-ED-ext} allows one to establish novel conservation laws along the event horizon for the middle Maxwell components under axisymmetry in the extremal $|a|=M$ case.
\end{itemize}

\subsection{Control over the zeroth order flux}
\label{sec:middle-zeroth-order}

By exploiting the Maxwell equations \eqref{eq:del3-sigma'}--\eqref{eq:del4-ualpha'}, we can derive conservation laws which control the zero mode of the middle Maxwell components $\Sigma{\bsy\rho}$ and $\Sigma{\bsy\sigma}$ along the event horizon, thereby obtaining the following lemma.

\begin{lemma} \label{lemma:zero-mode-horizon-oldsystem}
Fix $M>0$. Then, there exists a uniform constant $C=C(M)>0$ such that, for any $|a|\leq M$ and any solution 
$$\bsy{\mathfrak S}_\circ=(\bsy{\alpha}[\bsy{\mathfrak S}_\circ],\bsy{\alphab}[\bsy{\mathfrak S}_\circ],\bsy{\rho}[\bsy{\mathfrak S}_\circ],\bsy{\sigma}[\bsy{\mathfrak S}_\circ])$$ 
to the Maxwell equations \eqref{eq:del3-sigma'}--\eqref{eq:del4-ualpha'} arising from seed initial data prescribed on $\Sigma_0$, the following estimates hold.
\begin{itemize}
\item Horizon boundedness of zero modes: we have
\begin{align*}
&\Bigg|\int_{\mathbb S^2_{\tau,r_+}}\Sigma{\bsy\rho}[\bsy{\mathfrak S}_\circ]\,d\Omega\Bigg|^2+\Bigg|\int_{\mathbb S^2_{\tau,r_+}}\Sigma {\bsy\sigma}[\bsy{\mathfrak S}_\circ]\, d\Omega\Bigg|^2\\
&\quad \leq C\left( \lVert \Sigma{\bsy\rho}[\bsy{\mathfrak S}_\circ]\rVert^2_{L^2({\mathbb{S}^2_{0,r_+}})}+\lVert \Sigma{\bsy\sigma}[\bsy{\mathfrak S}_\circ]\rVert^2_{L^2({\mathbb{S}^2_{0,r_+}})}+\lVert r^2\bsy\alpha[\bsy{\mathfrak S}_\circ]\rVert^2_{L^2({\mathbb{S}^2_{0,r_+}})}+\lVert r^2\bsy\alpha[\bsy{\mathfrak S}_\circ]\rVert^2_{L^2({\mathbb{S}^2_{\tau,r_+}})}\right)
\end{align*}
for all $\tau\geq 0$.
\item Horizon decay of zero modes: there exist finite $\mathfrak q_1, \mathfrak q_2\in \mathbb R$, which can be read off from the seed initial data, such that the normalised solution $\widetilde{\bsy{\mathfrak S}}_{\circ}=\bsy{\mathfrak S}_{\circ}-\bsy{\mathfrak S}_{\circ}{}_{\mathfrak q_1,\mathfrak q_2}$ satisfies 
\begin{align*}
\lp|\int_{\mathbb S^2_{\tau,r_+}}\Sigma {\bsy\rho}[\widetilde{\bsy{\mathfrak S}}_\circ]\,d\Omega\rp|^2+\lp|\int_{\mathbb S^2_{\tau,r_+}}\Sigma {\bsy\sigma}[\widetilde{\bsy{\mathfrak S}}_\circ]\,d\Omega\rp|^2\leq C \lVert r^2\bsy\alpha[\bsy{\mathfrak S}_\circ]\rVert^2_{L^2({\mathbb{S}^2_{\tau,r_+}})}
\end{align*}
for all $\tau\geq 0$.
\end{itemize}
\end{lemma}

\begin{proof} We start from the Maxwell equations \eqref{eq:del4-sigma'}--\eqref{eq:del4-rho'}. Restricted to $\mc H^+$, these take the form
\begin{align}
\nablasl_4(\Sigma\bsy{\rho}) &= {\slashed{\textup{div}}}\, (\Sigma {\bsy{\alpha}}) \, , & \nablasl_4(\Sigma\bsy{\sigma})&= -{\slashed{\textup{curl}}}\, (\Sigma {\bsy{\alpha}})  \, . \label{eq:del4-on-H}
\end{align}
We recall that, for $|a|\neq 0$, the differential operators on the right hand side of equations \eqref{eq:del4-on-H} are defined relative to the (non-integrable) distribution $\mathfrak D_{\mc N_{\rm as}}$. By applying the $\mathbb{S}^2$-projection formulae of Proposition~\ref{prop:proj-formulas}, we project equations \eqref{eq:del4-on-H} over the horizon foliation spheres $\mathbb{S}^2_{\tau,r_+}$ and obtain
\begin{align*}
\check{\nablasl}_4\left(\Sigma\bsy{\rho}-\frac{1}{2}\, \left(\mathfrak k, \Sigma\widetilde{\bsy{\alpha}}\right)_{\check{\slashed{g}}}\right)=\check{\nablasl}_4(\Sigma\bsy{\rho})-\frac{1}{2}\, \left(\mathfrak k, \check{\nablasl}_4\,(\Sigma\widetilde{\bsy{\alpha}})\right)_{\check{\slashed{g}}}&= \check{\slashed{\textup{div}}}\, (\Sigma\widetilde{\bsy{\alpha}})  \, ,\\
\check{\nablasl}_4\left(\Sigma\bsy{\sigma}+\frac{1}{2}\, \mathfrak k\wedge_{\check{\slashed{g}}} \Sigma\widetilde{\bsy{\alpha}}\right)=\check{\nablasl}_4(\Sigma\bsy{\sigma})+\frac{1}{2}\, \mathfrak k\wedge_{\check{\slashed{g}}} \check{\nablasl}_4\,(\Sigma\widetilde{\bsy{\alpha}})&= -\check{\slashed{\textup{curl}}}\, (\Sigma\widetilde{\bsy{\alpha}})  \, , 
\end{align*}
where we used the identity $\check{\nablasl}_4\mathfrak k|_{\mathcal{H}^+}=0$ (see Proposition~\ref{prop:proj-formulas}). We note that the differential operators $\check{\slashed{\textup{div}}}$ and $ \check{\slashed{\textup{curl}}}$ are now angular operators on $\mathbb{S}^2_{\tau,r_+}$-spheres acting on $\mathbb{S}^2_{\tau,r_+}$-tensors. One can therefore integrate over the $\mathbb{S}^2_{\tau,r_+}$-spheres to obtain the conservation laws
\begin{align}
\check{\nablasl}_4\left[\int_{\mathbb{S}^2_{\tau,r_+}}\left(\Sigma\bsy{\rho}-\frac{1}{2}\, \left(\mathfrak k, \Sigma\widetilde{\bsy{\alpha}}\right)_{\check{\slashed{g}}}\right)\slashed{\varepsilon}_{\check{\slashed{g}}}\right]&= 0  \, , &
\check{\nablasl}_4\left[\int_{\mathbb{S}^2_{\tau,r_+}}\left(\Sigma\bsy{\sigma}+\frac{1}{2}\, \mathfrak k\wedge_{\check{\slashed{g}}} \Sigma\widetilde{\bsy{\alpha}}\right)\slashed{\varepsilon}_{\check{\slashed{g}}}\right]&=0  \label{eq:zero-mode-horizon-conservation}
\end{align}
or, equivalently, the identities
\begin{align}\label{eq:zero-mode-horizon-oldsystem-int}
\begin{split}
\int_{\mathbb{S}^2_{\tau,r_+}}\Sigma\bsy{\rho}\,\slashed{\varepsilon}_{\check{\slashed{g}}}&=\int_{\mathbb{S}^2_{0,r_+}}\left(\Sigma\bsy{\rho}-\frac{1}{2}\, \left(\mathfrak k, \Sigma\widetilde{\bsy{\alpha}}\right)_{\check{\slashed{g}}}\right)\slashed{\varepsilon}_{\check{\slashed{g}}} + \frac{1}{2}\int_{\mathbb{S}^2_{\tau,r_+}} \left(\mathfrak k, \Sigma\widetilde{\bsy{\alpha}}\right)_{\check{\slashed{g}}} \slashed{\varepsilon}_{\check{\slashed{g}}} \,,\\
\int_{\mathbb{S}^2_{\tau,r_+}}\Sigma\bsy{\sigma}\,\slashed{\varepsilon}_{\check{\slashed{g}}} &=\int_{\mathbb{S}^2_{0,r_+}}\left(\Sigma\bsy{\sigma}+\frac{1}{2}\, \mathfrak k\wedge_{\check{\slashed{g}}} \Sigma\widetilde{\bsy{\alpha}}\right)\slashed{\varepsilon}_{\check{\slashed{g}}} +\frac{1}{2}\int_{\mathbb{S}^2_{\tau,r_+}} \left(\mathfrak k\wedge_{\check{\slashed{g}}} \Sigma\widetilde{\bsy{\alpha}}\right)\slashed{\varepsilon}_{\check{\slashed{g}}} \,.
\end{split}
\end{align}
By using the identity \eqref{identity_event_horizon_volume_forms_bis}, we can replace $\slashed{\varepsilon}_{\check{\slashed{g}}}$ by $d\Omega_{\mathbb{S}^2_{\tau,r_+}}$ in identities \eqref{eq:zero-mode-horizon-conservation}--\eqref{eq:zero-mode-horizon-oldsystem-int}. We can then estimate the last term of each of the identities \eqref{eq:zero-mode-horizon-oldsystem-int} as follows
\begin{align*}
\Bigg|\int_{\mathbb{S}^2_{\tau,r_+}} \left(\mathfrak k, \Sigma\widetilde{\bsy{\alpha}}\right)_{\check{\slashed{g}}}d\Omega\Bigg|^2+\Bigg|\int_{\mathbb{S}^2_{\tau,r_+}} \left(\mathfrak k\wedge_{\check{\slashed{g}}} \Sigma\widetilde{\bsy{\alpha}}\right)d\Omega\Bigg|^2\lesssim  \int_{\mathbb{S}^2_{\tau,r_+}} |\widetilde{\bsy\alpha}|_{\check{\slashed{g}}}^2\, d\Omega = \int_{\mathbb{S}^2_{\tau,r_+}} |{\bsy\alpha}|^2_{\slashed{g}}\, d\Omega \, ,
\end{align*}
where we estimated $\sup_{\mc H^+}|\mathfrak{k}|_{\check{\slashed{g}}}\lesssim 1$ depending only on the spacetime parameters and used the property \eqref{equality_proj_norms} in the last equality. On the other hand, the first term on the right hand side of identities \eqref{eq:zero-mode-horizon-oldsystem-int} is determined by initial data. In particular, by setting 
\begin{gather*}
\begin{gathered}
\mathfrak q_1 :=\frac{1}{I_{a,M}}\int_{\mathbb{S}^2_{0,r_+}}\left(\Sigma\bsy{\rho}-\frac{1}{2}\, \left(\mathfrak k, \Sigma\widetilde{\bsy{\alpha}}\right)_{\check{\slashed{g}}}\right)d\Omega\,, \qquad \mathfrak q_2 :=\frac{1}{I_{a,M}}\int_{\mathbb{S}^2_{0,r_+}}\left(\Sigma\bsy{\sigma}+\frac{1}{2}\, \mathfrak k\wedge_{\check{\slashed{g}}} \Sigma\widetilde{\bsy{\alpha}}\right)d\Omega\,,
\end{gathered} \numberthis\label{eq:choice-charges}\\
I_{a,M}:=\int_{\mathbb{S}^2_{0,r_+}}\frac{1-(a/r_+)^2\cos^2\theta}{1+(a/r_+)^2\cos^2\theta}\, d\Omega = 4\pi\lp(\frac{2\arctan(a/r_+)}{(a/r_+)}-1\rp), 
\end{gather*}
those terms vanish for the renormalised solution $\widetilde{\bsy{\mathfrak S}}_{\circ}=\bsy{\mathfrak S}_{\circ}-\bsy{\mathfrak S}_{\circ}{}_{\mathfrak q_1,\mathfrak q_2}$.
\end{proof}

\begin{remark}
    By expressing the Maxwell tensor $\bsy{F}$ in terms of $(\bsy{\alpha},\bsy{\alphab},\bsy{\rho},\bsy{\sigma})$, one can check that the conserved (along $\mathcal{H}^+$) quantities in \eqref{eq:zero-mode-horizon-conservation} correspond to the (everywhere) conserved topological charges
    \begin{align}
        &\int_{\mathbb{S}^2_{\tau,r}}\bsy{F} \, , & &\int_{\mathbb{S}^2_{\tau,r}}\bsy{\star}\bsy{F}
    \end{align}
    when the latter integrals are computed at $\mathcal{H}^+$.
\end{remark}

\begin{remark} \label{rmk:zero-mode-horizon-oldsystem} For $|a|=0$, it follows from equations \eqref{eq:del4-sigma'}--\eqref{eq:del4-rho'} that one has the conservation laws
\begin{align*}
{\nablasl}_4\left[\int_{\mathbb{S}^2_{\tau,r}}\Sigma\bsy{\rho}\,d\Omega\right]&= 0  \, , &
{\nablasl}_4\left[\int_{\mathbb{S}^2_{\tau,r}}\Sigma\bsy{\sigma}\,d\Omega\right]&=0 
\end{align*}
for \textit{any} $r\in [2M,\infty)$. Indeed, the differential operators on the right hand side of equations \eqref{eq:del4-sigma'}--\eqref{eq:del4-rho'} are now angular operators on the foliation spheres, thus giving rise to vanishing integrals over the spheres. 
\end{remark}

The Maxwell equations \eqref{eq:del3-sigma'}--\eqref{eq:del4-ualpha'} are \emph{coupled} equations for derivatives of $\Sigma\bsy{\rho}$ and $\Sigma\bsy{\sigma}$. For the remainder of the analysis, it will be more convenient to work with the modified Maxwell equations \eqref{D3_sigma}--\eqref{DA_rho}, which are \emph{decoupled} equations for derivatives of $\widehat{\bsy{\rho}}$ and $\widehat{\bsy{\sigma}}$. To this end, we next show that the previous Lemma~\ref{lemma:zero-mode-horizon-oldsystem} yields control over the zero mode of the modified middle Maxwell components $\widehat{\bsy{\rho}}$ and $\widehat{\bsy{\sigma}}$ along the event horizon.

\begin{lemma} \label{lemma:zero-mode-horizon}
Fix $M>0$ and let $p_\alpha>0$. Assume that, for any $|a|\leq M$ and any solution 
$$\bsy{\mathfrak S}=(\bsy{\alpha}[\bsy{\mathfrak S}],\bsy{\alphab}[\bsy{\mathfrak S}],\widehat{\bsy{\rho}}[\bsy{\mathfrak S}],\widehat{\bsy{\sigma}}[\bsy{\mathfrak S}])$$ 
to the modified Maxwell equations \eqref{D3_sigma}--\eqref{DA_rho} arising from seed initial data prescribed on $\Sigma_0$, the qualitative statement
\begin{align*}
\lim_{\tau\to \infty}\int_{\mathbb{S}^2_{\tau,r_+}}\lp(|\mathfrak{T}(\bsy\alpha[\bsy{\mathfrak S}])|^2+|\underline{\mathfrak{T}}(\bsy\alphab[\bsy{\mathfrak S}])|^2\rp)d\Omega = 0
\end{align*}
and the quantitative estimate
\begin{align*}
\mathbb E_{\mc H^+}[r^2\bsy\alpha[\bsy{\mathfrak S}]](\tau,\infty)\leq \frac{\mathbb D[r^2\bsy\alpha[\bsy{\mathfrak S}]](0)}{(1+\tau)^{1+p_{\alpha}}}
\end{align*}
hold, where $\mathbb D$ is an energy quantity depending only on the seed initial data. Then, there exists a uniform constant $C=C(M,p_{\alpha})>0$ such that the following estimates hold.
\begin{itemize}
\item Horizon boundedness of zero modes: we have
\begin{align*}
\Bigg|\int_{\mathbb S^2_{\tau,r_+}}\widehat{\bsy\rho}[\bsy{\mathfrak S}]\,d\Omega\Bigg|^2+\Bigg|\int_{\mathbb S^2_{\tau,r_+}}\widehat{\bsy\sigma}[\bsy{\mathfrak S}]\, d\Omega\Bigg|^2\leq C\left( \lVert \widehat{\bsy\rho}[\bsy{\mathfrak S}]\rVert^2_{L^2({\mathbb{S}^2_{0,r_+}})}+\lVert \widehat{\bsy\sigma}[\bsy{\mathfrak S}]\rVert^2_{L^2({\mathbb{S}^2_{0,r_+}})}+\mathbb D[r^2\bsy\alpha[\bsy{\mathfrak S}]](0)\right)
\end{align*}
for all $\tau\geq 0$.
\item Horizon decay of zero modes: there exist finite $\mathfrak q_1,\mathfrak q_2\in\mathbb{R}$, which can be read off from the seed initial data, such that the normalised solution $\widetilde{\bsy{\mathfrak S}}=\bsy{\mathfrak S}-\bsy{\mathfrak S}_{\mathfrak q_1,\mathfrak q_2}$ satisfies 
\begin{align*}
\lp|\int_{\mathbb S^2_{\tau,r_+}}\widehat{\bsy\rho}[\widetilde{\bsy{\mathfrak S}}]\,d\Omega\rp|^2+\lp|\int_{\mathbb S^2_{\tau,r_+}}\widehat{\bsy\sigma}[\widetilde{\bsy{\mathfrak S}}]\,d\Omega\rp|^2\leq C \frac{\mathbb D[r^2\bsy\alpha[\bsy{\mathfrak S}]](0)}{(1+\tau)^{p_{\alpha}}}
\end{align*}
for all $\tau\geq 0$.
\end{itemize}
\end{lemma}

\begin{proof} Let $\gamma:= \frac{a}{r_+}$. We note
\begin{align*}
I_\gamma &:= \int_{\mathbb{S}^2_{\tau,r_+}}\frac{1-\gamma^2\cos^2\theta}{1+\gamma^2\cos^2\theta}\, d\Omega = 4\pi\lp(\frac{2\arctan\gamma}{\gamma}-1\rp)\neq 0 \,, &
\int_{\mathbb{S}^2_{\tau,r_+}}\frac{\cos\theta \, d\Omega}{1+\gamma^2\cos^2\theta} &=0 \, .
\end{align*}
Let $f_\rho,f_\sigma\colon [0,\infty)\to \mathbb{R}$ such that
\begin{align*}
f_\rho(\tau) &:= \int_{\mathbb{S}^2_{\tau,r_+}} \widehat{\bsy\rho}\, d\Omega \, , & f_\sigma(\tau):= \int_{\mathbb{S}^2_{\tau,r_+}} \widehat{\bsy\sigma}\, d\Omega \, .
\end{align*}
By the fundamental theorem of calculus and equation \eqref{D4_rho}, we have
\begin{align}
|f_\rho (\tau)-f_\rho(0)| &\lesssim \lp|\int_0^{\tau}\int_{\mathbb{S}^2_{\tau',r_+}} \nablasl_4\widehat{\bsy\rho}\, d\Omega d\tau'\rp| \lesssim \int_0^{\tau}\int_{\mathbb{S}^2_{\tau',r_+}} (|\nablasl \bsy\alpha|+|\bsy\alpha|)\, d\Omega d\tau' \nonumber\\
&\lesssim \sqrt{\mathbb E_{\mc H^+}[\bsy\alpha](0,1)} + \int_1^{\tau}\int_{\mathbb{S}^2_{\tau',r_+}} (|\nablasl \bsy\alpha|+|\bsy\alpha|)\, d\Omega d\tau'\,. \label{fin_term_f_rho}
\end{align}
To control the second term in \eqref{fin_term_f_rho}, we consider the dyadic sequence $\{\tau_i=2^i\}_{i=0}^\infty$ and obtain
\begin{align*}
\int_{1}^{\infty}\int_{\mathbb{S}^2_{\tau,r_+}}\lp(| {\bsy \alpha}|+|\nablasl {\bsy \alpha}|\rp)d\Omega d\tau
&= \sum_{i=0}^\infty\int_{\tau_i}^{\tau_{i+1}}\int_{\mathbb{S}^2_{\tau,r_+}}\lp(| {\bsy \alpha}|+|\nablasl {\bsy \alpha}|\rp)d\Omega d\tau\\
&\leq \sum_{i=0}^\infty\sqrt{\tau_{i+1}-\tau_i}\Big(\int_{\tau_i}^{\tau_{i+1}}\int_{\mathbb{S}^2_{\tau,r_+}}\lp(| {\bsy \alpha}|^2+|\nablasl {\bsy \alpha}|^2\rp)d\Omega d\tau\Big)^{1/2}\\
&\lesssim \sum_{i=0}^\infty\sqrt{\tau_i}\lp(\mathbb{E}_{\mc H^+}[{\bsy \alpha}](\tau_i,\infty)\rp)^{1/2} \lesssim \sum_{i=0}^\infty\sqrt{\tau_i}\sqrt{\frac{\mathbb{D}[r^2{\bsy \alpha}](0)}{\tau_i^{1+p_{\alpha}}}}\lesssim \sqrt{\mathbb{D}[r^2{\bsy \alpha}](0)} \,. \numberthis\label{eq:rho-zero-mode-bddness}
\end{align*}
Combining \eqref{eq:rho-zero-mode-bddness} with \eqref{fin_term_f_rho}, we deduce
\begin{align*}
|f_\rho(\tau)|&\lesssim \lVert \widehat{\bsy \rho} \rVert_{L^2(\mathbb{S}^2_{0,r_+})}+\sqrt{\mathbb{D}[r^2{\bsy \alpha}](0)} \, , & |f_\sigma(\tau)|&\lesssim \lVert \widehat{\bsy \sigma} \rVert_{L^2(\mathbb{S}^2_{0,r_+})}+\sqrt{\mathbb{D}[r^2{\bsy \alpha}](0)} \, ,
\end{align*}
where the latter estimate is derived by an analogous argument. This concludes the proof of the first bullet point in the lemma. For the second bullet point in the lemma, we rely on the previous Lemma~\ref{lemma:zero-mode-horizon-oldsystem}. In view of \eqref{eq:def-rho-hat}--\eqref{eq:def-sigma-hat}, we have
\begin{align*}
\int_{\mathbb{S}^2_{\tau,r_+}}\widehat{\bsy \rho} \, d\Omega&=\frac{1}{I_\gamma}\int_{\mathbb{S}^2_{\tau,r_+}}\Sigma \bsy \rho \, d\Omega \\
&\qquad -\frac{1}{I_\gamma}\int_{\mathbb{S}^2_{\tau,r_+}}\frac{1-\gamma^2\cos^2\theta}{1+\gamma^2\cos^2\theta} (\widehat{\bsy\rho}-f_\rho) \, d\Omega +\frac{2\gamma}{I_\gamma} \int_{\mathbb{S}^2_{\tau,r_+}}\frac{\cos\theta}{1+\gamma^2\cos^2\theta}  (\widehat{\bsy\sigma}-f_\sigma) \, d\Omega \,, \\
\int_{\mathbb{S}^2_{\tau,r_+}}\widehat{\bsy \sigma} \, d\Omega&= \frac{1}{I_\gamma}\int_{\mathbb{S}^2_{\tau,r_+}}\Sigma \bsy \sigma \, d\Omega\\
&\qquad -\frac{1}{I_\gamma}\int_{\mathbb{S}^2_{\tau,r_+}}\frac{1-\gamma^2\cos^2\theta}{1+\gamma^2\cos^2\theta} (\widehat{\bsy\sigma}-f_\sigma) \, d\Omega -\frac{2\gamma}{I_\gamma} \int_{\mathbb{S}^2_{\tau,r_+}}\frac{\cos\theta}{1+\gamma^2\cos^2\theta}  (\widehat{\bsy\rho}-f_\rho) \, d\Omega \,.
\end{align*}
By using equations \eqref{DA_rho}--\eqref{DA_sigma}, we can estimate the last two terms on the right hand side by
\begin{align*}
&\Big|\int_{\mathbb{S}^2_{\tau,r_+}}\frac{1-\gamma^2\cos^2\theta}{1+\cos^2\theta} (\widehat{\bsy\rho}-f_\rho) \, d\Omega\Big|^2+\Big|\int_{\mathbb{S}^2_{\tau,r_+}}\frac{\cos\theta}{1+\gamma^2\cos^2\theta}  (\widehat{\bsy\sigma}-f_\sigma) \, d\Omega\Big|^2\\
&\qquad +\Big|\int_{\mathbb{S}^2_{\tau,r_+}}\frac{1-\gamma^2\cos^2\theta}{1+\cos^2\theta} (\widehat{\bsy\sigma}-f_\sigma) \, d\Omega\Big|^2+\Big|\int_{\mathbb{S}^2_{\tau,r_+}}\frac{\cos\theta}{1+\gamma^2\cos^2\theta}  (\widehat{\bsy\rho}-f_\rho) \, d\Omega\Big|^2\\
&\quad \lesssim \int_{\mathbb{S}^2_{\tau,r_+}}\lp[(\widehat{\bsy\rho}-f_\rho)^2+(\widehat{\bsy\sigma}-f_\sigma)^2\rp] \, d\Omega \lesssim  \int_{\mathbb{S}^2_{\tau,r_+}}\lp(|{\nablasl}\widehat{\bsy\rho}|^2+|{\nablasl}\widehat{\bsy\sigma}|^2\rp)d\Omega\lesssim\int_{\mathbb{S}^2_{\tau,r_+}}\lp(|\mathfrak T(\bsy\alpha)|^2+|\underline{\mathfrak T}(\bsy\alphab)|^2\rp)d\Omega\,.
\end{align*}
By the assumptions on the dynamics of $(\bsy\alpha,\bsy\alphab)$ in the statement of the lemma, we deduce that 
\begin{align*}
\lim_{\tau\to \infty}\int_{\mathbb{S}^2_{\tau,r_+}}\widehat{\bsy \rho} \, d\Omega=\frac{1}{I_\gamma}\lim_{\tau\to \infty}\int_{\mathbb{S}^2_{\tau,r_+}}\Sigma \bsy \rho \, d\Omega\,, \qquad 
\lim_{\tau\to \infty}\int_{\mathbb{S}^2_{\tau,r_+}}\widehat{\bsy \sigma} \, d\Omega= \frac{1}{I_\gamma}\lim_{\tau\to \infty}\int_{\mathbb{S}^2_{\tau,r_+}}\Sigma \bsy \sigma \, d\Omega\,. \numberthis \label{eq:zero-mode-horizon-newold}
\end{align*}
By Lemma~\ref{lemma:zero-mode-horizon-oldsystem}, we deduce $\tau$-decay (without a rate) of $f_\rho$ and $f_\sigma$ for the normalised solution $\tilde{\bsy{\mathfrak S}}$, i.e.~$\lim_{\tau\to \infty}f_\rho(\tau)=\lim_{\tau\to \infty}f_\sigma(\tau)=0$. By the fundamental theorem of calculus, the normalised solution $\tilde{\bsy{\mathfrak S}}$ satisfies
\begin{align*}
|f_\rho(\tau)|\lesssim \Bigg|\int_{\tau}^\infty \int_{\mathbb{S}^2_{\tau',r_+}}\nablasl_4\widehat{\bsy\rho}\, d\Omega d\tau'\Bigg|\,, \qquad |f_\sigma(\tau)|\lesssim \Bigg|\int_{\tau}^\infty \int_{\mathbb{S}^2_{\tau',r_+}}\nablasl_4\widehat{\bsy\sigma}\, d\Omega d\tau'\Bigg|\,.
\end{align*}
By applying the same strategy as in the boundedness argument, one considers the sequence $\left\lbrace \tau_i=2^i \tau_0\right\rbrace_{i=0}^\infty$ with $\tau_0=1+\tau$ and uses the assumptions on the dynamics of $(\bsy\alpha,\bsy\alphab)$ in the statement of the lemma to obtain
\begin{align*}
|f_\rho(\tau)|^2+|f_\sigma(\tau)|^2&\lesssim \mathbb{E}_{\mc H^+}[{\bsy \alpha}](\tau,\tau_0)+\lp(\sum_{i=0}^\infty\sqrt{\tau_i}\sqrt{\frac{\mathbb{D}[r^2{\bsy \alpha}](0)}{\tau_i^{1+p_{\alpha}}}}\rp)^2\lesssim \frac{\mathbb{D}[r^2{\bsy \alpha}](0)}{(1+\tau)^{p_{\alpha}}}\,, 
\end{align*}
which concludes the proof of the lemma.
\end{proof}

\begin{remark}  
In the proof of Lemma~\ref{lemma:zero-mode-horizon}, we obtain the boundedness part of the statement by relying solely on the modified Maxwell equations. In contrast, the decay argument relies on the prior choice of charges and decay estimates of Lemma~\ref{lemma:zero-mode-horizon-oldsystem} for the original middle Maxwell components $\Sigma\bsy{\rho}$ and $\Sigma\bsy{\sigma}$. However, we point out that, once the boundedness part of Lemma~\ref{lemma:zero-mode-horizon} is established, one may choose finite charges on the \emph{final} horizon sphere such that
\begin{align*}
{\mathfrak q}_1 &:=\lim_{\tau\to \infty} \int_{\mathbb S^2_{\tau,r_+}}\widehat{\bsy\rho}\, d\Omega\,, & {\mathfrak q}_2 &:=\lim_{\tau\to \infty} \int_{\mathbb S^2_{\tau,r_+}}\widehat{\bsy\sigma}\, d\Omega  \numberthis\label{eq:choice-charges-2}
\end{align*}
and directly prove the decay part of Lemma~\ref{lemma:zero-mode-horizon} for the modified middle Maxwell components $\widehat{\bsy\rho}$ and $\widehat{\bsy\sigma}$ by subtracting the modified stationary solution $\bsy{\mathfrak S}_{\mathfrak{q}_1,\mathfrak{q}_2}$. This alternative decay argument would only employ the modified Maxwell equations \eqref{D4_sigma}--\eqref{D4_rho}, thus producing a full proof of Lemma~\ref{lemma:zero-mode-horizon} which never invokes Lemma~\ref{lemma:zero-mode-horizon-oldsystem} nor the original system of Maxwell equations. The teleological choice of charges \eqref{eq:choice-charges-2} and alternative decay argument would moreover allow one to dispense with any projection procedure over spheres. Note, however, that, in view of \eqref{eq:zero-mode-horizon-newold}, \eqref{eq:zero-mode-horizon-conservation} and \eqref{eq:choice-charges}, invoking the projection procedure over spheres \textit{a posteriori} would allow us to deduce that the teleological charges can be read off from the initial data: in other words, the choices \eqref{eq:choice-charges} and \eqref{eq:choice-charges-2} turn out to be the same. 
\end{remark}

\begin{remark} \label{rmk_a_priori_horizon_flux}
Once combined with the modified Maxwell equations for derivatives of $\widehat{\bsy\rho}$ and $\widehat{\bsy\sigma}$, Lemma~\ref{lemma:zero-mode-horizon} immediately provides a-priori control over the $L^2(\mathbb{S}^2_{\tau,r})$-norm of $\widehat{\bsy\rho}$ and $\widehat{\bsy\sigma}$ along the event horizon.
\end{remark}

Lemma \ref{lemma:zero-mode-horizon} can be applied to establish control over the $L^2(\Sigma_{\tau})$-norm of the modified middle Maxwell components $\widehat{\bsy\rho}$ and $\widehat{\bsy\sigma}$, as stated in the following proposition. We remark that, in the estimates stated in the proposition, all future energy norms of the extremal Maxwell components appearing on the right hand side are degenerate at $\mathcal{H}^+$. We also note that the estimates possess near-horizon integrals of transversal derivatives of $\widehat{\bsy\rho}$ and $\widehat{\bsy\sigma}$ on the right hand side. These terms will be treated in the sequel (see Sections \ref{sec:conclusion-middle-EB-ED-sub}--\ref{sec:conclusion-middle-EB-ED-ext}).

\begin{proposition} \label{prop:L2-norm} Let $\epsilon>0$.
Under the assumptions of Lemma~\ref{lemma:zero-mode-horizon}, there exists a uniform constant $C=C(M,p_{\alpha},\epsilon)>0$ such that the following estimates hold.
\begin{itemize}
\item $L^2(\Sigma_{\tau})$-boundedness: we have
\begin{align*}
\int_{\Sigma_\tau}r^{-2}(|\widehat{\bsy\rho}[\bsy{\mathfrak S}]|^2+|\widehat{\bsy\sigma}[\bsy{\mathfrak S}]|^2)d\Omega dr &\leq C \Bigg|\int_{r_+}^{r_++\epsilon}\int_{\mathbb S^2_{\tau,r_+}}\nablasl_3\widehat{\bsy \rho}[\bsy{\mathfrak S}]\,d\Omega dr\Bigg|^2+ C\Bigg|\int_{r_+}^{r_++\epsilon}\int_{\mathbb S^2_{\tau,r_+}} \nablasl_3\widehat{\bsy \sigma}[\bsy{\mathfrak S}]\,d\Omega dr\Bigg|^2\\
&\qquad + C\left(\mathbb{D}[r^2\bsy\alpha[\bsy{\mathfrak S}]](0)  + \mathbb{E}^1[r^2\bsy\alpha[\bsy{\mathfrak S}]](\tau)+\mathbb{E}^1[r\bsy\alphab[\bsy{\mathfrak S}]](\tau) \right)
\end{align*}
for all $\tau\geq 0$.
\item $L^2(\Sigma_{\tau})$-decay: there exist finite $\mathfrak q_1,\mathfrak q_2\in\mathbb{R}$, which can be read off from the seed initial data, such that the normalised solution $\widetilde{\bsy{\mathfrak S}}=\bsy{\mathfrak S}-\bsy{\mathfrak S}_{\mathfrak q_1,\mathfrak q_2}$ satisfies 
\begin{align*}
\int_{\Sigma_\tau}r^{-2}(|\widehat{\bsy\rho}[\widetilde{\bsy{\mathfrak S}}]|^2+|\widehat{\bsy\sigma}[\widetilde{\bsy{\mathfrak S}}]|^2)d\Omega dr &\leq C \Bigg|\int_{r_+}^{r_++\epsilon}\int_{\mathbb S^2_{\tau,r_+}} \nablasl_3\widehat{\bsy \rho}[\bsy{\mathfrak S}]\,d\Omega dr\Bigg|^2+ C\Bigg|\int_{r_+}^{r_++\epsilon}\int_{\mathbb S^2_{\tau,r_+}} \nablasl_3\widehat{\bsy \sigma}[\bsy{\mathfrak S}]\,d\Omega dr\Bigg|^2\\
&\qquad +C\left(  \frac{\mathbb{D}[r^2\bsy\alpha[\bsy{\mathfrak S}]](0)}{\tau^{p_{\alpha}}}  + \mathbb{E}^1[r^2\bsy\alpha[\bsy{\mathfrak S}]](\tau)+\mathbb{E}^1[r\bsy\alphab[\bsy{\mathfrak S}]](\tau) \right)
\end{align*}
for all $\tau\geq 0$.
\end{itemize}
\end{proposition}

\begin{proof} The proof can be divided into two steps: in the first step we control the relevant fluxes for large $r\geq R\gg r_+$, whereas in the second step we control the relevant fluxes in the bounded-$r$ region up to $\mc H^+$. We present the proof of the estimates for $\widehat{\bsy\rho}$. The case of $\widehat{\bsy \sigma}$ follows by the same arguments.

\medskip
\noindent\textit{Step 1: the large-$r$ region.} We integrate the identity
\begin{align*}
\nablasl_4(r^{-1}|\widehat{\bsy\rho}|^2)+\frac{\Delta}{r^2\Sigma}|\widehat{\bsy\rho}|^2&= \frac{2}{r}\widehat{\bsy\rho}\, \nablasl_4\widehat{\bsy\rho}
\end{align*}
over the hypersurface $\Sigma_\tau$ and obtain
\begin{align*}
\int_{\Sigma_\tau\cap \{ r\geq R\}} \frac{\Delta}{r^2\Sigma}|\widehat{\bsy\rho}|^2 d\Omega dr 
&= -\int_{\Sigma_\tau\cap \{ r\geq R\}}\nablasl_4\lp(\frac{1}{r}|\widehat{\bsy\rho}|^2\rp)d\Omega dr +\int_{\Sigma_\tau\cap \{ r\geq R\}}\frac{2}{r}\widehat{\bsy\rho}\, \nablasl_4\widehat{\bsy\rho} \, d\Omega dr\\
%EXTRA COMPUTATION; DO NOT ERASE
%&= \int_{\Sigma_\tau\cap \{ r\geq R\}}\lp(\tilde X-\nablasl_4\rp)(r^{-1}|\widehat{\bsy\rho}|^2)d\Omega dr +\int_{\Sigma_{\tau}\cap\{r\geq R\}}\frac{2}{r}\widehat{\bsy\rho}\, \nablasl_4\widehat{\bsy\rho}\,  d\Omega dr
%- \int_{R}^\infty\int_{\mathbb S^2_{\tau,r}}\tilde X\lp(\frac{1}{r}|\widehat{\bsy\rho}|^2\rp) d\Omega dr \\
&\leq  \frac{1}{R}\int_{\mathbb S^2_{\tau,R}}|\widehat{\bsy\rho}|^2 d\Omega+\int_{\Sigma_\tau\cap \{ r\geq R\}}\lp(\tilde X-\nablasl_4\rp)\lp(\frac{1}{r}|\widehat{\bsy\rho}|^2\rp)d\Omega dr \\
&\qquad +\int_{\Sigma_{\tau}\cap\{r\geq R\}}\frac{2}{r}\widehat{\bsy\rho}\, \nablasl_4\widehat{\bsy\rho}\,  d\Omega dr\,, 
\end{align*}
where we introduced the vector field $\tilde X$ defined in \eqref{eq:tilde-X} and dropped $\lim_{r\rightarrow\infty}\int_{\mathbb S^2_{\tau,r}}-\frac{1}{r}|\widehat{\bsy\rho}|^2 d\Omega\leq 0$. By a Cauchy--Schwarz inequality applied to the last two terms, we get
\begin{align*}
\int_{\Sigma_\tau\cap \{ r\geq R\}} r^{-2}|\widehat{\bsy\rho}|^2 d\Omega dr &\lesssim \int_{\mathbb S^2_{\tau,R}}|\widehat{\bsy\rho}|^2 d\Omega+\int_{\Sigma_\tau\cap \{ r\geq R\}}\Big|\lp(\tilde X-\nablasl_4\rp)(r^{-1})\Big||\widehat{\bsy\rho}|^2 d\Omega dr  \numberthis \label{eq:L2-norm-int} \\
&\qquad +\int_{\Sigma_{\tau}\cap\{r\geq R\}}\lp[|\nablasl_4\widehat{\bsy\rho}|^2+|\lp(\tilde X-\nablasl_4\rp)\widehat{\bsy\rho}|^2\rp]  d\Omega dr.
\end{align*}
By direct inspection of definition \eqref{eq:tilde-X}, we have 
\begin{align*}
|\lp(\tilde X-\nablasl_4\rp)(r^{-1})| \lesssim r^{-3}\,,
\end{align*}
and therefore the second term on the right hand side of \eqref{eq:L2-norm-int} can be absorbed into the left hand side for $R$ sufficiently large, which we fix accordingly. Then, using the modified Maxwell equations, we observe
\begin{align*}
\int_{\Sigma_{\tau}\cap\{r\geq R\}}\lp[|\nablasl_4\widehat{\bsy\rho}|^2+|\lp(\tilde X-\nablasl_4\rp)\widehat{\bsy\rho}|^2\rp]d\Omega dr \lesssim \mathbb{E}^1[r^2\bsy\alpha](\tau)+ \mathbb{E}^1[r\bsy\alphab](\tau).
\end{align*}
Putting everything together, we conclude that 
\begin{align*}
\int_{\Sigma_\tau\cap \{ r\geq R\}} r^{-2}|\widehat{\bsy\rho}|^2 d\Omega dr &\lesssim \int_{\mathbb S^2_{\tau,R}}|\widehat{\bsy\rho}|^2 d\Omega+\mathbb{E}^1[r^2\bsy\alpha](\tau)+ \mathbb{E}^1[r\bsy\alphab](\tau)\,.
\end{align*}

\medskip
\noindent\textit{Step 2: the bounded-$r$ region.} 
Let 
\begin{align*}
f(\tau,r):= \int_{\mathbb{S}^2_{\tau,r}}\widehat{\bsy\rho}\,d\Omega
\end{align*}
and note the Poincar\'{e} inequality
\begin{align*}
\int_{\mathbb{S}^2_{\tau,r}}|\widehat{\bsy\rho}|^2d\Omega \lesssim \lp|f(\tau,r)\rp|^2 +\int_{\mathbb{S}^2_{\tau,r}}r^2|{\nablasl}_{\mathbb S^2_{\tau,r}}\widehat{\bsy\rho}|^2d\Omega\,.
\end{align*}
We can integrate the above inequality over the hypersurface $\Sigma_\tau$ and obtain (note that, since we are integrating in a bounded-$r$ region, the $r$-weights in the following estimates are in fact irrelevant)
\begin{align}
\int_{\Sigma_\tau\cap\{r\leq R\}}r^{-2}|\widehat{\bsy\rho}|^2d\Omega dr &\lesssim \int_{r_+}^R r^{-2}\lp|f(\tau,r)\rp|^2 dr +\int_{\Sigma_{\tau}\cap\{r\leq R\}}|{\nablasl}_{\mathbb S^2_{\tau,r}}\widehat{\bsy\rho}|^2d\Omega dr \nonumber\\
&\lesssim \int_{r_+}^R r^{-2}\lp|f(\tau,r)\rp|^2 dr + \mathbb{E}^1[r^2\bsy\alpha](\tau)+ \mathbb{E}^1[r\bsy\alphab](\tau)\,, \label{aux_ineq_deg_energies_alphab}
\end{align}
where we have used the modified Maxwell equations in the second inequality. We remark that the inequality
\begin{equation*}
    \int_{\Sigma_{\tau}\cap\{r\leq R\}}|{\nablasl}_{\mathbb S^2_{\tau,r}}\widehat{\bsy\rho}|^2d\Omega dr\lesssim \int_{ \Sigma_{\tau}\cap \{r\leq R\}} \lp(\lp|\nablasl_4\widehat{\bsy \rho}\rp|^2+\lp|\Delta\nablasl_3\widehat{\bsy \rho}\rp|^2+\lp|\nablasl\widehat{\bsy \rho}\rp|^2\rp)d\Omega dr
\end{equation*}
and the structure of the equations are, in particular, used to conclude that inequality \eqref{aux_ineq_deg_energies_alphab} can in fact be written in terms of degenerate-at-$\mc H^+$ energy norms of $\bsy\alpha$ and $\bsy\alphab$. For the first term in \eqref{aux_ineq_deg_energies_alphab}, we observe
\begin{align*}
\int_{r_+}^R r^{-2}\lp|f(\tau,r)\rp|^2 dr\leq \sup_{r\in[r_+R]} \lp|f(\tau,r)\rp|^2 \int_{r_+}^R r^{-2} dr \leq \frac{1}{r_+}\sup_{r\in[r_+R]} \lp|f(\tau,r)\rp|^2.
\end{align*}
We then apply the fundamental theorem of calculus and obtain
\begin{align*}
\sup_{r\in[r_+,R]} \lp|f(\tau,r)\rp|^2&\leq   |f(\tau,r_+)|^2 + \Bigg|\int_{r_+}^{R} \p_r  f(\tau,r) dr\Bigg|^2\\
&\lesssim_\epsilon |f(\tau,r_+)|^2+ \Bigg|\int_{ \Sigma_{\tau}\cap \{r\leq r_++\epsilon\}} \nablasl_3\widehat{\bsy \rho}\, d\Omega dr\Bigg|^2\\
&\qquad+ \int_{ \Sigma_{\tau}\cap \{r\leq R\}} \lp(\lp|\nablasl_4\widehat{\bsy \rho}\rp|^2+\lp|\frac{\Delta}{\Sigma}\nablasl_3\widehat{\bsy \rho}\rp|^2+\lp|\nablasl\widehat{\bsy \rho}\rp|^2\rp)d\Omega dr\\
&\lesssim_\epsilon \lp| f(\tau,r_+)\rp|^2 + \Bigg|\int_{ \Sigma_{\tau}\cap \{r\leq r_++\epsilon\}} \nablasl_3\widehat{\bsy \rho}\,d\Omega dr\Bigg|^2 + \mathbb{E}^1[r^2\bsy\alpha](\tau)+\mathbb{E}^1[r\bsy\alphab](\tau)
\end{align*}
for any $\epsilon>0$, where we used the modified Maxwell equations in the last inequality. We remark that the equations do not allow to estimate the second term in the last inequality in terms of the degenerate-at-$\mc H^+$ energy norm $\mathbb{E}^1[r\bsy\alphab](\tau)$. 

Putting everything together, we obtain 
\begin{align*}
\int_{\Sigma_\tau\cap\{r\leq R\}}r^{-2}|\widehat{\bsy\rho}|^2d\Omega dr &\lesssim_\epsilon \Bigg| \int_{\mathbb{S}^2_{\tau,r_+}}\widehat{\bsy\rho}\,d\Omega\Bigg|^2  + \Bigg|\int_{ \Sigma_{\tau}\cap \{r\leq r_++\epsilon\}} \nablasl_3\widehat{\bsy \rho}\,d\Omega dr\Bigg|^2 + \mathbb{E}^1[r^2\bsy\alpha](\tau)+\mathbb{E}^1[r\bsy\alphab](\tau) \, .
\end{align*}
By combining with the previous step and with Lemma~\ref{lemma:zero-mode-horizon}, we conclude the proof of the proposition. 
\end{proof}

\subsection{Control over higher order fluxes}
\label{sec:middle-higher-order}

Control over first (and higher) order energy fluxes of the modified middle Maxwell components $\widehat{\bsy{\rho}}$ and $\widehat{\bsy{\sigma}}$ can be immediately read off from the modified Maxwell equations \eqref{D3_sigma}--\eqref{DA_rho}.

\begin{proposition} \label{prop:higher-derivatives} 
Fix $M>0$ and let $J\geq 0$. Then, there exists a uniform constant $C=C(M,J)>0$ such that, for any $|a|\leq M$ and any solution 
$$\bsy{\mathfrak S}=(\bsy{\alpha}[\bsy{\mathfrak S}],\bsy{\alphab}[\bsy{\mathfrak S}],\widehat{\bsy{\rho}}[\bsy{\mathfrak S}],\widehat{\bsy{\sigma}}[\bsy{\mathfrak S}])$$ 
to the modified Maxwell equations \eqref{D3_sigma}--\eqref{DA_rho}, the energy estimates
\begin{align*}
\dot{\mathbb{E}}^{J}\big[\widehat{\bsy\rho}[\bsy{\mathfrak S}]\big](\tau)+\dot{\mathbb{E}}^{J}\big[\widehat{\bsy\sigma}[\bsy{\mathfrak S}]\big](\tau) &\leq C\left( {\mathbb{E}}^{\max\{J,1\}}\big[r^2{\bsy\alpha}[\bsy{\mathfrak S}]\big](\tau)+{\mathbb{E}}^{\max\{J,1\}}\big[r{\bsy\alphab}[\bsy{\mathfrak S}]\big](\tau) \right) \,,\\
\dot{\overline{\mathbb{E}}}^{J}\big[\widehat{\bsy\rho}[\bsy{\mathfrak S}]\big](\tau)+\dot{\overline{\mathbb{E}}}^{J}\big[\widehat{\bsy\sigma}[\bsy{\mathfrak S}]\big](\tau) &\leq C \left( \overline{\mathbb{E}}^{\max\{J,1\}}\big[r^2{\bsy\alpha}[\bsy{\mathfrak S}]\big](\tau)+{\overline{\mathbb{E}}}^{\max\{J,1\}}\big[r{\bsy\alphab}[\bsy{\mathfrak S}]\big](\tau) \right)
\end{align*}
hold for all $\tau\in\mathbb{R}$.
\end{proposition}

\subsection{Proof of Theorem \ref{thm:middle-EB-ED-sub}}
\label{sec:conclusion-middle-EB-ED-sub}

By combining the results of Sections \ref{sec:middle-zeroth-order}--\ref{sec:middle-higher-order}, one obtains:

\begin{proof}[Proof of Theorem~\ref{thm:middle-EB-ED-sub}]
By the modified Maxwell equations~\eqref{D3_sigma}--\eqref{D3_rho}, we have 
\begin{align*}
\Bigg|\int_{r_+}^{r_++\epsilon}\int_{\mathbb S^2_{\tau,r}}\nablasl_3\widehat{\bsy \rho}\,d\Omega dr\Bigg|^2+ \Bigg|\int_{r_+}^{r_++\epsilon}\int_{\mathbb S^2_{\tau,r}} \nablasl_3\widehat{\bsy \sigma}\,d\Omega dr\Bigg|^2 \lesssim \int_{r_+}^{r_++\epsilon}\int_{\mathbb S^2_{\tau,r}}\lp(|\nablasl\bsy\alphab|^2+|\bsy\alphab|^2\rp)d\Omega dr \lesssim  \overline{\mathbb{E}}^1[r\bsy\alphab](\tau)\,,
\end{align*}
where the implicit constants are independent of $\epsilon$. By combining the above inequality with Propositions~\ref{prop:L2-norm} and \ref{prop:higher-derivatives}, we conclude the proof of the theorem.
\end{proof}

We recall that the reader  interested only in the analysis for the full sub-extremal $|a|<M$ case can now skip ahead to Section~\ref{sec:analysis-teukolsky}. The reader  interested also in the nearly extremal ($|a|\approx M$) and exactly extremal ($|a|=M$) cases should continue to the next sections.

\subsection{Proof of Theorem \ref{thm:middle-EB-ED-ext}}
\label{sec:conclusion-middle-EB-ED-ext}

The aim of this section is to prove Theorem \ref{thm:middle-EB-ED-ext} by combining Propositions~\ref{prop:L2-norm} and \ref{prop:higher-derivatives}. This requires one to control the first two terms on the right hand side of the estimates in Proposition~\ref{prop:L2-norm} in terms of \emph{degenerate}-at-$\mc H^+$ energy norms of the extremal Maxwell components $\bsy \alpha$ and $\bsy\alphab$. 

To this end, we recall the notation introduced in Sections~\ref{sec:Maxwell-Kerr} and \ref{sec:Teukolsky-norms} and the quantities
\begin{equation*}
\swei{\upupsilon}{0}=-\widehat{\bsy\rho }+i\widehat{\bsy\sigma}
\end{equation*}
and
\begin{align*}
\swei{\tilde \upphi}{+ 1}_0 &=\Delta\swei{\upalpha}{+1}=-\frac{1}{\sqrt{2}}\Sigma\bsy\alpha(e_1^{\rm as}+ie_2^{\rm as})\,, \\
 \swei{\tilde \upphi}{- 1}_0 &= \sqrt{r^2+a^2}\Delta^{-1}\swei{\upalpha}{-1}=\frac{\kappa^2}{\sqrt{2}\Sigma}(r^2+a^2)^{1/2}\bsy\alphab(e_1^{\rm as}-ie_2^{\rm as})\,, \\
\swei{\Phi}{-1}&=\frac{r^2+a^2}{\Delta} \Sigma e_4^{\rm as} (\swei{\upphi}{- 1}_0)\,,\\
\swei{\Phi}{+ 1}&= (r^2+a^2) e_3^{\rm as} (\swei{\upphi}{+ 1}_0) \, .
\end{align*}
We will also employ the scalar function
\begin{align*}
    w(r):=\frac{\Delta}{(r^2+a^2)^2} \,.
\end{align*}
It will be useful to keep in mind Lemma~\ref{lemma:isomorphism} and Remark~\ref{rmk:spin-weighted} (see also Appendix~\ref{sec:appendix-MaxwellNP}) throughout this section.

The core of the section is to derive a suitable estimate for the quantity 
\begin{align*}
\Bigg|\int_{r_+}^{r_++\epsilon}\int_{\mathbb S^2_{\tau,r}}\nablasl_3\widehat{\bsy \rho}\,d\Omega dr\Bigg|^2+ \Bigg|\int_{r_+}^{r_++\epsilon}\int_{\mathbb S^2_{\tau,r}} \nablasl_3\widehat{\bsy \sigma}\,d\Omega dr\Bigg|^2 = \Bigg|\int_{r_+}^{r_++\epsilon}\int_{\mathbb S^2_{\tau,r}} e_3^{\rm as}\big(\swei{\upupsilon}{0}\big)d\Omega dr\Bigg|^2 \numberthis \label{eq:to-estimate-extremal}
\end{align*}
with $\epsilon>0$ to be fixed.

\subsubsection{Warm-up: from Schwarzschild to extremal Reissner--Nordstr\"om}
\label{sec:warm-up-spherically-symmetric}

To illustrate our approach, we first consider the Schwarzschild $|a|=0$ case. The following proposition and its proof will serve as a guide for the more intricate calculations in the $|a|\neq 0$ case.

\begin{proposition}[An identity on Schwarzschild] \label{prop:upupsilon-Schwarzschild} Fix $M>0$ and let $|a|=0$. For $L\geq |s|$, let
\begin{align*}
S_{L}^{[s]}(\theta,\phi):=\sum_{\max\{|m|,|s|\}\leq L}S_{m,L}^{[s]}(\theta)e^{im\phi}
\end{align*}
be the $s$ spin-weighted spherical harmonics, which form a complete orthonormal basis of $L^2(\sin\theta d\theta)$, see for instance \cite[Section 3.1.2]{SRTdC2023} for a rigorous definition. Then, by writing
\begin{align*}
\swei{\upupsilon}{0}_L:=\int_{\mathbb{S}^2}\swei{\upupsilon}{0}S_{L}^{[0]} d\Omega\,, \quad L\geq 0\,,\qquad\qquad 
\swei{\Phi}{-1}_L:=\int_{\mathbb{S}^2}\swei{\Phi}{-1} S_{L}^{[-1]} d\Omega\,, \quad L\geq 1\,,\\ 
\end{align*}
we have the pointwise identity
\begin{align*}
e_3^{\rm as}\big(\swei{\upupsilon}{0}_L\big)=\frac{1}{\sqrt{2r^2}}\begin{dcases} 0\,, & L=0\\ \frac{e_3^{\rm as}\big(\swei{\Phi}{-1}_{L}\big)}{\sqrt{L(L+1)}}\,, &L\geq 1\end{dcases}\,. \numberthis\label{eq:sphsym-identity}
\end{align*}
\end{proposition}

\begin{proof} We divide the proof into three steps.

\medskip
\noindent \textit{Step 1: an identity from the Teukolsky equation}. From the spin $-1$ Teukolsky equation \eqref{eq:Teukolsky-spin}, using  \eqref{eq:def-upphi}, we deduce the constraint identity 
\begin{align*}
e_3^{\rm as}\big( \swei{\Phi}{-1}\big) = -(\swei{\mathring{\slashed\triangle}}{-1}+1) \swei{\tilde \upphi}{-1}_0 \,,
\end{align*}
where $\swei{\mathring{\slashed\triangle}}{-1}$ is the $(-1)$-spin-weighted Laplacian. See already Lemma~\ref{lemma:Teukolsky-invertible-identity} below for an explicit formula and more detail regarding the derivation of this identity.

\medskip
\noindent \textit{Step 2: inverting an angular operator}.  The spin-weighted spherical harmonics are eigenfunctions for the angular operator $\swei{\mathring{\slashed\triangle}}{-1}+1$, and
\begin{align*}
\lp(\swei{\mathring{\slashed\triangle}}{-1}+1\rp) S_{L}^{[-1]}(\theta,\phi)=L(L+1)S_{L}^{[-1]}(\theta,\phi)\,.
\end{align*}
Since $L\geq 1$, we have $L(L+1)\geq 2$, and hence we can define the inverse operator $\lp(\swei{\mathring{\slashed\triangle}}{-1}+1\rp)^{-1}$ on $L^2(\sin\theta d\theta)$. Indeed, we can explicitly write
\begin{align*}
\lp(\swei{\mathring{\slashed\triangle}}{-1}+1\rp)^{-1}S_{L}^{[-1]}(\theta,\phi) = \frac{1}{L(L+1)}S_{L}^{[-1]}(\theta,\phi)\,.
\end{align*}

\medskip
\noindent \textit{Step 3: an identity from the Maxwell equations}. From the Maxwell equation \eqref{eq:Maxwell-n}, we have
\begin{align*}
e_3^{\rm as}\,\swei{\upupsilon}{0} &= \frac{1}{\sqrt{2r^2}}\lp(\p_\theta+\frac{i}{\sin\theta}\p_\phi +\cot\theta\rp)\swei{\tilde\upphi}{-1}_0\,.
\end{align*}

\medskip
\noindent \textit{Conclusion}. By combining the previous steps, we have
\begin{align*}
e_3^{\rm as}\,\swei{\upupsilon}{0}= -\frac{1}{\sqrt{2r^2}}\lp(\p_\theta+\frac{i}{\sin\theta}\p_\phi +\cot\theta\rp)(\swei{\mathring{\slashed\triangle}}{-1}+1)^{-1}e_3^{\rm as}\big( \swei{\Phi}{-1}\big)\,.
\end{align*}
The conclusion follows by a computation, and using the orthogonality of the spin-weighted spherical harmonics.  It will be useful to use the identity
\begin{align*}
-\lp(\p_\theta+\frac{i}{\sin\theta}\p_\phi +\cot\theta\rp)S_{L}^{[-1]}(\theta,\phi) = \sqrt{L(L+1)}S_{L}^{[0]}(\theta,\phi), \quad L\geq 1,
\end{align*}
as the differential operator on the left hand side is the spin-raising operator.
\end{proof}

Proposition~\ref{prop:upupsilon-Schwarzschild} does \emph{not} exploit the sub-extremality, or positive surface gravity, of the Schwarzschild solution. In fact, the pointwise identity \eqref{eq:sphsym-identity} continues to hold if one replaces the Schwarzschild solution $(\mc M,g_{0,M})$ with any Reissner--Nordstr\"om solution $(\mc M,g_{e,M})$ with $|e|\leq M$. Indeed, the identities
\begin{align*}
e_3^{\rm as}\big(\swei{\upupsilon}{0}\big) &= \frac{1}{\sqrt{2r^2}}\lp(\p_\theta+\frac{i}{\sin\theta}\p_\phi +\cot\theta\rp)\swei{\tilde\upphi}{-1}_0\,, \\
\lp(\p_\theta-\frac{i}{\sin\theta}\p_\phi\rp)\swei{\upupsilon}{0}&=\frac{1}{\sqrt{2r^2}}\swei{\Phi}{-1}
\end{align*}
still hold for the Maxwell equations on Reissner--Nordstr\"{o}m spacetimes, where one can check that equations \eqref{eq:Maxwell-m}--\eqref{eq:Maxwell-n} hold on Reissner--Nordstr\"{o}m after replacing $\Delta$ with $r^2-2M r+e^2$ and then setting $|a|=0$ (see \cite[Chapter 44, eq.\ (200)--(203)]{Chandrasekhar}). For $|e|=M$, We remark that the relation $e_3^{\rm as}\big(\swei{\upupsilon}{0}_L\big)\sim_{r,L} e_3^{\rm as}\big(\swei{\Phi}{-1}_L\big)$ for $L\geq 1$ implies that $\swei{\upupsilon}{0}$ inherits the dynamics of the transversal derivatives of $\swei{\Phi}{-1}$ along $\mc H^+$, and in particular the conservation laws for transversal derivatives of $\swei{\Phi}{-1}$ which can be derived for example by adapting \cite[Section 6]{Apetroaie2022}. This would provide information on the nature of the Maxwell horizon hair formed in the evolution of the Maxwell equations on extremal Reissner--Nordstr\"{o}m spacetimes\footnote{We again emphasize that the evolution problem for the Maxwell equations on extremal Reissner--Nordstr\"om black holes is distinct from the evolution problem for the (linearised) Einstein--Maxwell system around extremal Reissner--Nordstr\"om. It is the latter (not the former) which is the object of study in \cite{Apetroaie2022}.}.

\subsubsection{Three key lemmas} \label{sec:three_key_lemmas}

We now consider the $|a|\neq 0$ case. The aim of the section is to generalise the three steps in the proof of Proposition~\ref{prop:upupsilon-Schwarzschild} to the full $|a|\leq M$ range. First, we use the spin $-1$ Teukolsky equation~\eqref{eq:Teukolsky-spin} to deduce a useful identity for $\swei{\Phi}{-1}$ and $\swei{\tilde\upphi}{-1}_0$. In the remainder of the section, we adopt the notation $f^\prime=\frac{\Delta}{r^2+a^2}\frac{d}{dr}f$ for any scalar function $f$.

\begin{lemma} \label{lemma:Teukolsky-invertible-identity} Fix $M>0$ and let $|a|\leq M$. For $\swei{\Phi}{-1}$ and $\swei{\tilde\upphi}{-1}_0$ as above, we have the pointwise identity
\begin{align*}
\mathring{\slashed\triangle}^{[-1]}_{\mc H}\swei{\tilde \upphi}{-1}_0&=-e_3^{\rm as}\swei{\Phi}{-1}
\\
&\qquad +\frac{1}{r^2+a^2}\lp[a^2\sin^2\theta\lp(\frac{\Sigma e_4^{\rm as}}{4(r^2+a^2)}-\frac{a}{r^2+a^2} Z+\frac{\Delta e_3^{\rm as}}{2(r^2+a^2)}-\frac{w'}{4w}\rp)+a(Z+i\cos\theta)\rp]\swei{\Phi}{-1} \\
&\qquad +\lp(\frac14a^2\sin^2\theta e_3^{\rm as}\lp(\frac{\Delta}{r^2 +a^2}e_3^{\rm as}\rp)+\frac{a\Sigma}{r^2+a^2}Z e_3^{\rm as}  +ia\cos\theta e_3^{\rm as}\rp)\lp(\frac{\Delta}{r^2 +a^2}\swei{\tilde \upphi}{-1}_0\rp) \\
&\qquad +aw\lp[a-\frac{a^2r\sin^2\theta }{r^2+a^2}Z+\frac{(r^2+a^2)}{M}Z\rp]\swei{\tilde \upphi}{-1}_0\\
&\qquad +\frac{a^2}{Mr_+}\frac{r^2-r_+^2}{r^2+a^2}\lp[Z +i\cos\theta -\frac{a^2\sin^2\theta}{2Mr_+}\lp(1 -\frac{r^2-r_+^2}{2(r^2+a^2)}Z\rp)\rp]Z\swei{\tilde \upphi}{-1}_0 \,, \numberthis\label{eq:Teukolsky-invertible-identity}
\end{align*}
where, letting $\swei{\mathring{\slashed\triangle}}{-1}$ denote the $- 1$ spin-weighted spherical Laplacian
\begin{align} \label{eq:spin-weighted-laplacian}
\mathring{\slashed\triangle}^{[- 1]}&:=  -\frac{1}{\sin \theta} \frac{\partial}{\partial \theta} \left(\sin \theta \frac{\partial}{\partial \theta}\right) - \frac{1}{\sin^2 \theta} Z^2 - 2 i(- 1)\frac{ \cos \theta}{\sin^2 \theta} Z + (- 1)^2\cot^2 \theta \,,
\end{align}
we write $\mathring{\slashed\triangle}^{[-1]}_{\mc H}$ for the related operator
\begin{align} \label{eq:spin-weighted-laplacian-H}
\mathring{\slashed\triangle}^{[-1]}_{\mc H}:= \swei{\mathring{\slashed\triangle}}{-1}+1+\frac{a^2}{Mr_+}\lp(1-\frac{a^2}{4Mr_+}\sin^2\theta\rp)Z^2+\frac{a}{Mr_+}\lp(r_++ia\cos\theta\rp)Z\,.
\end{align}
\end{lemma}

\begin{remark} We note that, schematically, identity \eqref{eq:Teukolsky-invertible-identity} has the form
\begin{align*}
\mathring{\slashed\triangle}^{[-1]}_{\mc H}\swei{\tilde \upphi}{-1}_0= -e_3^{\rm as}\swei{\Phi}{-1}+ \dots\,,
\end{align*}
where the ellipses represent either terms in which $\swei{\tilde \upphi}{-1}_0$ and its $(e_3^{\rm as}, e_2^{\rm as}, e_1^{\rm as})$-derivatives are multiplied by functions of $(r,\theta)$ which, if $|a|=M$, vanish at least linearly as $r\to r_+$, or terms proportional to $\swei{\Phi}{-1}$ and its $((r^2+a^2)^{-1}\Delta e_3^{\rm as}, e_2^{\rm as}, e_1^{\rm as})$-derivatives.
\end{remark}

\begin{proof}[Proof of Lemma~\ref{lemma:Teukolsky-invertible-identity}]
As a preliminary step, we rewrite $T$-derivatives of $\swei{\upphi}{-1}_0$ in terms of the vector fields $e_3^{\rm as}$ and $e_4^{\rm as}$, or, to make contact with \cite{SRTdC2023}, in terms of $\uL$ and $L$ (recall definition \eqref{def_L_Lbar}). Making use of \eqref{eq:def-upphi}, we have
\begin{align*}
T\swei{\upphi}{-1}_0 &= \lp(\frac{\uL+L}{2}-\frac{a}{r^2+a^2}Z\rp)\swei{\upphi}{-1}_0 =\lp(\frac12 \uL-\frac{a}{r^2+a^2}Z\rp)\swei{\upphi}{-1}_0 +\frac12 w\swei{\Phi}{-1}\,.\numberthis\label{eq:Tupphi0}
\end{align*}
We can repeat this calculation to obtain a formula for $T^2$-derivatives as follows
\begin{align*}
T^2\swei{\upphi}{-1}_0 &=\frac12 \uL T\swei{\upphi}{-1}_0-\frac{a}{r^2+a^2}Z T\swei{\upphi}{-1}_0 +\frac12 w T \swei{\Phi}{-1}\\
&=\lp(\frac12 \uL -\frac{a}{r^2+a^2}Z \rp)\lp(\frac12 \uL\swei{\upphi}{-1}_0 +\frac12 w\swei{\Phi}{-1}-\frac{a}{r^2+a^2}Z\swei{\upphi}{-1}_0\rp)+\frac12w \lp(\frac{L+\uL}{2}-\frac{a}{r^2+a^2}Z\rp)\swei{\Phi}{-1}\\
&=\frac14 \uL\uL\swei{\upphi}{-1}_0+\frac14 \lp(w\uL\swei{\Phi}{-1}-w'\swei{\Phi}{-1}\rp)-\frac{a}{r^2+a^2}Z\lp(\uL\swei{\upphi}{-1}_0+rw\swei{\upphi}{-1}_0\rp) \\
&\qquad-\frac{aw}{2(r^2+a^2)} Z\swei{\Phi}{-1}+\frac{a^2}{(r^2+a^2)^2}Z^2\swei{\upphi}{-1}_0+\frac14w (L+\uL)\swei{\Phi}{-1}-\frac{aw}{2(r^2+a^2)}Z\swei{\Phi}{-1}\\
&=\frac14 w L\swei{\Phi}{-1}-\frac{aw}{r^2+a^2} Z\swei{\Phi}{-1}+\frac{a^2}{(r^2+a^2)^2}Z^2\swei{\upphi}{-1}_0 \\
&\qquad + \frac12w \uL\swei{\Phi}{-1}-\frac14 w'\swei{\Phi}{-1}+ \frac14 \uL\uL\swei{\upphi}{-1}_0-\frac{a}{r^2+a^2}Z\uL\swei{\upphi}{-1}_0 -\frac{awr}{r^2+a^2}Z\swei{\upphi}{-1}_0 . \numberthis\label{eq:T^2upphi0}
\end{align*}
We now consider the spin $-1$ Teukolsky equation \eqref{eq:Teukolsky-spin}. From \eqref{eq:def-upphi}, we deduce that it can be written as the constraint equation, cf.\ \cite[Equation (3.19)]{SRTdC2023},
\begin{align}
\uL\swei{\Phi}{-1}= -(\swei{\mathring{\slashed\triangle}}{-1}_T+1)\swei{\upphi}{-1}_0-\frac{2ar}{r^2+a^2}Z\swei{\upphi}{-1}_0+a^2w\swei{\upphi}{-1}_0, \label{eq:constraint}
\end{align}
where, following \cite[Equation (3.13)]{SRTdC2023}, we have defined
\begin{align}
\swei{\mathring{\slashed\triangle}}{-1}_T=\swei{\mathring{\slashed\triangle}}{-1}-a^2\sin^2\theta TT -2aTZ-2ia\cos\theta T.\label{eq:angular-op}
\end{align}
Using the computations in \eqref{eq:Tupphi0}--\eqref{eq:T^2upphi0}, we obtain that the angular operator \eqref{eq:angular-op}  acts on $\swei{\upphi}{-1}_0$ via
\begin{align*}
&\swei{\mathring{\slashed\triangle}}{-1}\swei{\upphi}{-1}_0-\swei{\mathring{\slashed\triangle}}{-1}_T\swei{\upphi}{-1}_0\\
&\quad=\lp(a^2\sin^2\theta TT +2aTZ+2ia\cos\theta T\rp)\swei{\upphi}{-1}_0\\
%%%%%%%%%%%%%%%%%%%%%%%%%%%%%%%
%% EXTRA DETAIL; DO NOT ERASE
%%%%%%%%%%%%%%%%%%%%%%%%%%%%%%%
%&\quad= a^2\sin^2\theta \lp(\frac14 w L\swei{\Phi}{-1}-\frac{aw}{r^2+a^2} Z\swei{\Phi}{-1}+\frac{a^2}{(r^2+a^2)^2}Z^2\swei{\upphi}{-1}_0 \rp)\\
%&\quad\qquad+ 2a(Z+i\cos\theta)\lp(\frac12 w\swei{\Phi}{-1}-\frac{a}{r^2+a^2}Z\swei{\upphi}{-1}_0 \rp)\\
%&\quad\qquad + a^2\sin^2\theta \lp(\frac12w \uL\swei{\Phi}{-1}-\frac14 w'\swei{\Phi}{-1}+ \frac14 \uL\uL\swei{\upphi}{-1}_0-\frac{a}{r^2+a^2}Z\uL\swei{\upphi}{-1}_0 -\frac{arw}{r^2+a^2}Z\swei{\upphi}{-1}_0 \rp)\\
%&\quad\qquad+ a(Z+i\cos\theta) \uL\swei{\upphi}{-1}_0\\ 
%%%%%%%%%%%%%%%%%%%%%%%%%%%%%%%%5
%
&\quad = w\lp[a^2\sin^2\theta\lp(\frac14 L-\frac{a}{r^2+a^2} Z+\frac12 \uL-\frac{w'}{4w}\rp)+a(Z+i\cos\theta)\rp]\swei{\Phi}{-1}\\
&\quad\qquad -\frac{2a^2}{r^2+a^2}\lp[\lp(1-\frac{a^2}{2(r^2+a^2)}\sin^2\theta\rp)Z+i\cos\theta\rp]Z\swei{\upphi}{-1}_0\\
&\quad\qquad +  \lp(\frac14a^2\sin^2\theta \uL+\frac{a\Sigma}{r^2+a^2}Z  +ia\cos\theta\rp)\uL\swei{\upphi}{-1}_0 -\frac{a^3rw\sin^2\theta}{r^2+a^2}Z\swei{\upphi}{-1}_0. 
\end{align*}
We can plug the above into \eqref{eq:constraint} to obtain
\begin{align*}
e_3^{\rm as}\swei{\Phi}{-1}&=\frac{1}{r^2+a^2}\lp[a^2\sin^2\theta\lp(\frac14 L-\frac{a}{r^2+a^2} Z+\frac12 \uL-\frac{w'}{4w}\rp)+a(Z+i\cos\theta)\rp]\swei{\Phi}{-1} \\
&\qquad +\lp(\frac14a^2\sin^2\theta e_3^{\rm as}\uL+\frac{a\Sigma}{r^2+a^2}Z e_3^{\rm as}  +ia\cos\theta e_3^{\rm as}\rp)\lp(\frac{\Delta}{r^2 +a^2}\swei{\tilde \upphi}{-1}_0\rp) +a^2w\lp[1-\frac{ar\sin^2\theta }{r^2+a^2}Z\rp]\swei{\tilde \upphi}{-1}_0\\
&\qquad -\lp[\swei{\mathring{\slashed\triangle}}{-1}+\frac{2a^2}{r^2+a^2}\lp(1-\frac{a^2}{2(r^2+a^2)}\sin^2\theta\rp)Z^2+\frac{2a(r+ia\cos\theta)}{r^2+a^2}Z\rp]\swei{\tilde \upphi}{-1}_0.
\end{align*}
Finally, we add and subtract the last line with the differential operator evaluated at $r=r_+$.
%%%%%%%%%%%%%%%%
%% EXTRA DETAILS related to final sentence
%%%%%%%%%%%%%%%%
%\begin{align*}
%\frac{2a^2}{r^2+a^2}\lp(1-\frac{a^2}{2(r^2+a^2)}\sin^2\theta\rp)
%&=\frac{a^2}{M r_+} \lp(1-\frac{a^2\sin^2\theta}{4Mr_+}\rp)-\frac{a^2}{Mr_+}\lp(1-\frac{a^2\sin^2\theta}{2Mr_+}\rp)\frac{r^2-r_+^2}{r^2+a^2}-\lp(\frac{a^2\sin\theta}{2Mr_+}\rp)^2\lp(\frac{r^2-r_+^2}{r^2+a^2}\rp)^2,\\
%%
%\frac{2a(r+ia\cos\theta)}{r^2+a^2}&=\frac{a}{M}+i\frac{a^2}{Mr_+}\cos\theta-\frac{ia^2\cos\theta}{Mr_+}\frac{(r^2-r_+^2)}{(r^2+a^2)}-\frac{a}{M}\frac{\Delta}{r^2+a^2}
%\end{align*}
\end{proof}

Second, we show that the operator on the left hand side of \eqref{eq:Teukolsky-invertible-identity} is invertible for the full $|a|\leq M$ range.

\begin{lemma}\label{lemma:invertibility} Fix $M>0$ and let $|a|\leq M$. Then, the operator $\mathring{\slashed\triangle}^{[-1]}_{\mc H}\colon L^2(d\Omega)\to L^2( d\Omega)$ is invertible and its norm is bounded in terms of an explicitly computable constant that depends only on $M$.  
\end{lemma}

\begin{proof} The proof is similar to the proof of the analogous statement for the spin-weighted Laplacian $\mathring{\slashed\triangle}^{[s]}$, see  e.g.\ \cite[Section 3.1.2]{SRTdC2023}. We apply Sturm--Liouville theory to deduce that the operator $\mathring{\slashed\triangle}^{[-1]}_{\mc H}$ admits a countable family of eigenfunctions of the form $e^{im\tilde{\phi}{}^*}\tilde S_{m\ell}^{[-1]}(\theta)$, normalized to have unit $L^2(d\Omega)$-norm,  with corresponding (complex) eigenvalues $\tilde \lambda_{m\ell}^{[-1]}$, with $\ell\in\mathbb Z_{\geq \max\{|m|,1\}}$ being our choice of parameter to index the family.  To show the desired invertibility, it is enough to show that $|\tilde \lambda_{m\ell}^{[-1]}|$ is bounded below uniformly in $(m,\ell)$ and with respect to the black hole parameters. 

For $m=0$ or $|a|=0$, we have $\tilde \lambda_{0\ell}^{[-1]}=\ell(\ell+1)\geq 2>0$ and $\tilde S_{0\ell}^{[-1]}(\theta)=S_{0\ell}^{[-1]}(\theta)$ are the usual spin-weighted spherical harmonics. For $|a|\neq 0$ and $m\neq 0$, we have
\begin{align*}
\Im\tilde \lambda_{m\ell}^{[-1]}&=\frac{a}{M}m,\\
\Re \tilde \lambda_{m\ell}^{[-1]}&= \int_0^{\pi}\lp(|\p_\theta\tilde \Xi|^2+\tilde V^{\rm ang}_m|\tilde \Xi|^2\rp)\sin\theta d\theta,
\end{align*}
where $\tilde\Xi= c\tilde S_{m\ell}^{[-1]}(\theta)$ for some constant $c$ chosen so that the normalization condition $\int_0^{\pi}|\tilde \Xi|^2\sin\theta d\theta=1$ holds, and  where
\begin{align*}
\tilde V^{\rm ang}_m = \frac{(m\cos\theta-1)^2}{\sin^2\theta}+\frac{1}{4}\sin^2\theta m^2-\frac{a^2}{Mr_+} m\cos\theta+\lp(1-\frac{a^2}{Mr_+}\rp)\lp(1-\frac14 \sin^2\theta\rp) m^2.
\end{align*}
We divide the analysis into three cases:
\begin{itemize}
\item Case $|m|\geq 3$: We have
\begin{align*}
\tilde V^{\rm ang}_m (\theta) &\geq \inf_{\theta\in (0,\pi)}\lp(\frac{(m\cos\theta-1)^2}{\sin^2\theta}+\frac{1}{4}\sin^2\theta m^2-\frac{a^2}{Mr_+} m\cos\theta\rp)\\
&= \inf_{\theta\in (0,\pi/2)}\lp(\frac{(m\cos\theta-1)^2}{\sin^2\theta}+\frac{1}{4}\sin^2\theta m^2-\frac{a^2}{Mr_+}|m|\cos\theta\rp) \geq \frac12
\end{align*}
and hence $\Re \tilde \lambda_{m\ell}^{[-1]}\geq \frac12 > 0$.
\item Case $|m|=1,2$ and $|a|\geq \frac12 M$: We have $|\Im \tilde \lambda_{m\ell}^{[-1]}|\geq \frac12$. 
\item Case $|m|=1,2$ and $|a|\leq \frac12 M$: We obtain
\begin{align*}
\tilde V^{\rm ang}_{\pm 1} (\theta)&\geq \inf_{\theta\in (0,\pi), |m|=1}\tilde V^{\rm ang}_m(\theta)=\tilde V^{\rm ang}_{1}(0)= 1-\frac{2a^2}{Mr_+}\geq \frac12 ,\\
\tilde V^{\rm ang}_{\pm 2} (\theta)&\geq \inf_{\theta\in (0,\pi), |m|=2, |a|\leq \frac12 M}\tilde V^{\rm ang}_m(\theta)=\inf_{\theta\in (0,\pi/2), |a|=\frac12 M}\tilde V^{\rm ang}_2(\theta)\geq 2
\end{align*}
by an explicit computation, and thus $\Re \tilde \lambda_{m\ell}^{[-1]}\geq \frac12 > 0$. 
\end{itemize}
We deduce that, for all $m\in\mathbb{Z}$, $\ell\in\mathbb{Z}_{\geq \max\{|m|,1\}}$ and $a\in[0,M]$, we have the uniform bound $|\tilde \lambda_{m\ell}^{[-1]}|\geq \frac12 >0$. This concludes the proof.
%Actually the result follows from using the real part alone if we compute the eigenvalues numerically.
%MATHEMATICAL CODE 
%\[ScriptCapitalL] = -Cot[\[Theta]]*D[f[\[Theta]], \[Theta]] - 
%    D[f[\[Theta]], {\[Theta], 
%      2}] + ((m Cos[\[Theta]] - 1)^2/Sin[\[Theta]]^2 - 
%       m Cos[\[Theta]] + m^2 (1/4 Sin[\[Theta]]^2)) f[\[Theta]] /. 
%   m -> 1;
%{vals, funs} = 
%  NDEigensystem[\[ScriptCapitalL], f[\[Theta]], {\[Theta], 0, \[Pi]}, 
%   10];
%vals // Chop
\end{proof}

As the third, and final step, we use the Maxwell equation \eqref{eq:Maxwell-n} to derive a useful identity relating $\swei{\upupsilon}{0}$ with $\swei{\tilde \upphi}{-1}_0$ and $\swei{\Phi}{-1}$.

\begin{lemma} \label{lemma:Maxwell-n-identity} Fix $M>0$ and let $|a|\leq M$. For $\swei{\upupsilon}{0}$, $\swei{\Phi}{-1}$ and $\swei{\tilde\upphi}{-1}_0$ as above, we have the pointwise identity 
\begin{align*}
e_3^{\rm as} \swei{\upupsilon}{0}&= \frac{\kappa} {\sqrt{2(r^2+a^2)}}\mathring{\nablasl}^{[-1]}_{\mc H} \swei{\tilde \upphi}{-1}_0 +\frac{ia\kappa \sin\theta}{2\sqrt 2 (r^2+a^2)^{3/2}} \swei{\Phi}{-1} \\
&\qquad+\frac{ia\sin\theta} {\sqrt{2(r^2+a^2)}} \lp[\frac{\kappa}{2}e_3^{\rm as}\lp(\frac{\Delta}{r^2+a^2} \swei{\tilde \upphi}{-1}_0\rp) + \frac{a \kappa}{2Mr_+}\frac{r^2-r_+^2}{r^2+a^2}Z\swei{\tilde \upphi}{-1}_0  +\frac{(r-r_+)}{\kappa (r_+)}\swei{\tilde \upphi}{-1}_0\rp]\,,\numberthis\label{eq:Maxwell-n-identity}
\end{align*}
where we have introduced the notation
\begin{align*}
\mathring{\nablasl}_{\mc H}^{[-1]}:=\lp(\p_\theta+\frac{i\Sigma(r_+)}{2Mr_+\sin\theta}Z-\frac{ia\sin\theta}{\kappa(r_+)}+\cot\theta\rp)\,.
\end{align*}
\end{lemma}

\begin{remark} We note that, schematically, identity \eqref{eq:Maxwell-n-identity} has the form
$$e_3^{\rm as}\swei{\upupsilon}{0}= \frac{\kappa}{\sqrt{2(r^2+a^2)}}\mathring{\nablasl}_{\mc H}^{[-1]} \swei{\tilde \upphi}{-1}_0 +\dots$$
where the ellipses represent either terms in which $\swei{\tilde \upphi}{-1}_0$ and its $(e_3^{\rm as}, e_2^{\rm as}, e_1^{\rm as})$-derivatives are multiplied by functions of $(r,\theta)$ which, if $|a|=M$, vanish at least linearly as $r\to r_+$, or terms proportional to $\swei{\Phi}{-1}$.
\end{remark}

\begin{proof}[Proof of Lemma~\ref{lemma:Maxwell-n-identity}]
We recall identity \eqref{eq:Tupphi0} and the notation from the proof of Lemma~\ref{lemma:Teukolsky-invertible-identity}. By using the spin-weighted Maxwell equation~\eqref{eq:Maxwell-n}, we obtain
\begin{align*}
\uL\swei{\upupsilon}{0} &= \frac{\kappa^2}{\sqrt{2}(r^2+a^2)}\lp(\p_\theta+\frac{i}{\sin\theta}Z+ia\sin\theta T+\cot\theta\rp)\lp(\frac{\sqrt{r^2+a^2}}{\kappa}\swei{\upphi}{-1}_0\rp)\\
%%%%%%%%%%%%%%%%%%%
%% EXTRA DETAILS - DO NOT ERASE
%%%%%%%%%%%%%%%%%%%
%&=\frac{\kappa^2}{\sqrt{2(r^2+a^2)}}\lp(\p_\theta+\frac{i}{\sin\theta}Z+\cot\theta\rp)\lp(\frac{1}{\kappa}\swei{\upphi}{-1}_0\rp)\\
%&\qquad +\frac{ia\sin\theta \kappa}{\sqrt{2(r^2+a^2)}}\lp(\frac12 \uL \swei{\upphi}{-1}_0+\frac12w\swei{\Phi}{-1}-\frac{a}{r^2+a^2}Z\swei{\upphi}{-1}_0\rp)\\
%&= \frac{\kappa^2}{\sqrt{2(r^2+a^2)}}\lp(\p_\theta+\frac{i\Sigma}{(r^2+a^2)\sin\theta}Z+\cot\theta\rp)\lp(\frac{1}{\kappa}\swei{\upphi}{-1}_0\rp) +\frac{ia\sin\theta \kappa}{2\sqrt{2(r^2+a^2)}}\lp(\uL \swei{\upphi}{-1}_0+w\swei{\Phi}{-1}\rp)\\
%%%%%%%%%%%%%%%%%%%
&=\frac{\kappa}{\sqrt{2(r^2+a^2)}}\lp(\p_\theta+\frac{i\Sigma}{(r^2+a^2)\sin\theta}Z-\frac{ia\sin\theta}{\kappa}+\cot\theta\rp)\swei{\upphi}{-1}_0 +\frac{ia\sin\theta \kappa}{2\sqrt{2(r^2+a^2)}}\lp(\uL \swei{\upphi}{-1}_0+w\swei{\Phi}{-1}\rp).
\end{align*}
To conclude,  we simply compare the operator acting on the first term on the right hand side with $\mathring{\nablasl}^{[-1]}_{\mc H}$. 
%%%%%%%%%%%%%%%%%%%
%% EXTRA DETAILS - DO NOT ERASE
%%%%%%%%%%%%%%%%%%%
%\begin{align*}
%\frac{i\Sigma}{(r^2+a^2)\sin\theta} &= \frac{i\Sigma}{2Mr_+\sin\theta}+\frac{i}{\sin\theta}\frac{a^2\sin^2\theta}{2Mr_+}\frac{r^2-r_+^2}{r^2+a^2}= \frac{i\Sigma}{2Mr_+\sin\theta}+\frac{ia^2\sin\theta}{2Mr_+}\frac{r^2-r_+^2}{r^2+a^2}\,.\\
%\frac{ia\sin\theta}{\kappa} &= \frac{ia\sin\theta}{\kappa(r_+)} -\frac{ia\sin\theta(r-r_+)}{\kappa \kappa (r_+)}
%\end{align*}
%%%%%%%%%%%%%%%%%%%
\end{proof}

\subsubsection{Combining the key lemmas} \label{sec:combining_extremal_lemmas}

We now combine the three key lemmas of Section \ref{sec:three_key_lemmas} to estimate the quantity \eqref{eq:to-estimate-extremal}, thus extending Proposition~\ref{prop:upupsilon-Schwarzschild} to the full $|a|\leq M$ range.

\begin{proposition} \label{prop:control-near-hor-spin} Fix $M>0$ and let $0<\epsilon\leq \frac14 M$. Then, there exists a uniform constant $C=C(M,\epsilon)>0$ such that, for any $|a|\leq M$, the estimate
\begin{align*}
&\lp|\int_{r_+}^{r_++\epsilon}\int_{\mathbb{S}^2_{\tau,r}}e_3^{\rm as}(\swei{\upupsilon}{0}) \, dr d\Omega -{\int_{\mathbb{S}^2_{\tau,r_+}}\frac{\kappa(r_+) }{\sqrt{4Mr_+}}\mathring{\nablasl}_{\mc H}^{[-1]}\big(\mathring{\slashed{\triangle}}^{[-1]}_{\mc H}\big)^{-1}\swei{\Phi}{-1}}\,d\Omega\rp|^2 \\
&\quad \leq C\left( \sum_{\mathfrak X\in\{\mathbb X_1,\mathbb X_1\mathring{\nablasl}_{\mc H}^{[-1]}\}} \int_{ \Sigma_{\tau}\cap\{r\leq r_++\epsilon\}}|\mathfrak X\swei{\Phi}{-1}|^2 dr d\Omega+\sum_{\mathfrak X\in\{\mathbb X_0,\mathbb X_0\mathring{\nablasl}_{\mc H}^{[-1]}\}} \int_{ \Sigma_{\tau}\cap\{r\leq r_++\epsilon\}}|\mathfrak X\swei{\tilde\upphi}{-1}_0|^2 dr d\Omega \right) \numberthis \label{eq:control-near-hor-spin}
\end{align*}
holds, where $\big(\mathring{\slashed{\triangle}}^{[-1]}_{\mc H}\big)^{-1}$ and $\mathring{\nablasl}_{\mc H}^{[-1]}$ are defined above in Lemmas~\ref{lemma:invertibility} and \ref{lemma:Maxwell-n-identity} respectively and 
$\mathbb{X}_0$, $\mathbb{X}_1$ denote the collections of vector fields
\begin{align*}
\mathbb{X}_1&=\{\Delta e_3^{\rm as}, e_4^{\rm as}, Z,1\}\,,\\
\mathbb{X}_0&=\{\Delta^2 \lp(e_3^{\rm as}\rp)^2,\Delta e_3^{\rm as}, (r-r_+)^2 Z^2,(r-r_+)Z,(r-M)\} \, .
\end{align*} 
\end{proposition}

\begin{proof} Throughout the proof, we assume $r\leq 9M/4$. In view of Lemma~\ref{lemma:invertibility}, we let $\big(\mathring{\slashed\triangle}^{[-1]}_{\mc H}\big)^{-1}\colon L^2(d\Omega)\to L^2(d\Omega)$ be the inverse of the differential operator $\mathring{\slashed\triangle}^{[-1]}_{\mc H}$. From Lemma~\ref{lemma:Teukolsky-invertible-identity}, we readily derive the identity
\begin{align*}
\swei{\tilde \upphi}{-1}_0&=-e_3^{\rm as}\big(\mathring{\slashed\triangle}^{[-1]}_{\mc H}\big)^{-1}\swei{\Phi}{-1}\\
&\qquad +\frac{1}{r^2+a^2}\big(\mathring{\slashed\triangle}^{[-1]}_{\mc H}\big)^{-1}\lp[a^2\sin^2\theta\lp(\frac14 L-\frac{a}{r^2+a^2} Z+\frac12 \uL-\frac{w'}{4w}\rp)+a(Z+i\cos\theta)\rp]\swei{\Phi}{-1} \\
&\qquad +\big(\mathring{\slashed\triangle}^{[-1]}_{\mc H}\big)^{-1}\lp(\frac14a^2\sin^2\theta e_3^{\rm as}\uL+\frac{a\Sigma}{r^2+a^2}Z e_3^{\rm as}  +ia\cos\theta e_3^{\rm as}\rp)\lp(\frac{\Delta}{r^2 +a^2}\swei{\tilde \upphi}{-1}_0\rp) \\
&\qquad +aw\big(\mathring{\slashed\triangle}^{[-1]}_{\mc H}\big)^{-1}\lp[a-\frac{a^2r\sin^2\theta }{r^2+a^2}Z+\frac{(r^2+a^2)}{M}Z\rp]\swei{\tilde \upphi}{-1}_0\\
&\qquad +\frac{a^2}{Mr_+}\frac{r^2-r_+^2}{r^2+a^2}\big(\mathring{\slashed\triangle}^{[-1]}_{\mc H}\big)^{-1}\lp[\lp(Z +i\cos\theta\rp) -\frac{a^2\sin^2\theta}{2Mr_+}\lp(1 -\frac{r^2-r_+^2}{2(r^2+a^2)}Z\rp)\rp]Z\swei{\tilde \upphi}{-1}_0\,.\numberthis\label{eq:Teukolsky-invertible-identity-again}
\end{align*}
By combining the identities \eqref{eq:Maxwell-n-identity} and \eqref{eq:Teukolsky-invertible-identity-again}, we obtain  
\begin{align*}
e_3^{\rm as} \swei{\upupsilon}{0}
&=-e_3^{\rm as}\lp(\frac{\kappa} {\sqrt{2(r^2+a^2)}}\mathring{\nablasl}^{[-1]}_{\mc H} \big(\mathring{\slashed\triangle}^{[-1]}_{\mc H}\big)^{-1}\swei{\Phi}{-1}\rp)\\
&\qquad +\frac{a\kappa} {\sqrt{2}(r^2+a^2)^{3/2}}\lp\{\frac{i\sin\theta}{2} \swei{\Phi}{-1}+\mathring{\nablasl}^{[-1]}_{\mc H}\big(\mathring{\slashed\triangle}^{[-1]}_{\mc H}\big)^{-1}\lp[a\sin^2\theta\lp(\frac14 L-\frac{a}{r^2+a^2} Z+\frac12 \uL-\frac{w'}{4w}\rp)\rp.\rp.\\
&\qquad\qquad\qquad\qquad\qquad\qquad\qquad\qquad\qquad\qquad\qquad\qquad\quad \lp.\lp.+(Z+i\cos\theta)-\frac{(a+ir\cos\theta)}{\kappa}\rp]\rp\}\swei{\Phi}{-1}\\
&\qquad +\frac{a\kappa} {\sqrt{2(r^2+a^2)}}\lp\{\frac{i\sin\theta}{2} e_3^{\rm as}+\mathring{\nablasl}^{[-1]}_{\mc H}\big(\mathring{\slashed\triangle}^{[-1]}_{\mc H}\big)^{-1}\lp(\frac14a\sin^2\theta e_3^{\rm as}\uL+\frac{\Sigma}{r^2+a^2}Z e_3^{\rm as}  \rp.\rp.\\
&\qquad\qquad\qquad\qquad\qquad\qquad\qquad\qquad\qquad\qquad\qquad\qquad\qquad\quad \lp.\lp.\vphantom{\frac12}+i\cos\theta e_3^{\rm as}\rp)\rp\}\lp(\frac{\Delta}{r^2+a^2} \swei{\tilde \upphi}{-1}_0\rp)\\
&\qquad +\frac{a\kappa (r-r_+)} {Mr_+\sqrt{2(r^2+a^2)}} \lp\{i\sin\theta\lp[ \frac{a(r+r_+)}{2(r^2+a^2)}Z  +\frac{Mr_+}{\kappa (r_+)\kappa}\rp]\rp.\\
&\qquad\qquad\qquad \lp.+\frac{a(r+r_+)}{r^2+a^2}\mathring{\nablasl}^{[-1]}_{\mc H} \big(\mathring{\slashed\triangle}^{[-1]}_{\mc H}\big)^{-1}\lp[\lp(Z +i\cos\theta\rp) -\frac{a^2\sin^2\theta}{2Mr_+}\lp(1 -\frac{r^2-r_+^2}{2(r^2+a^2)}Z\rp)\rp]Z\rp\}\swei{\tilde \upphi}{-1}_0\\
&\qquad  +\frac{a\kappa w} {\sqrt{2(r^2+a^2)}}\mathring{\nablasl}^{[-1]}_{\mc H} \big(\mathring{\slashed\triangle}^{[-1]}_{\mc H}\big)^{-1}\lp[a-\frac{a^2r\sin^2\theta }{r^2+a^2}Z+\frac{(r^2+a^2)}{M}Z\rp]\swei{\tilde \upphi}{-1}_0\,, \numberthis\label{eq:e3upsilon-identity-final}
\end{align*}
where we used the commutation relations $[e_3^{\rm as}, \big(\mathring{\slashed{\triangle}}^{[-1]}_{\mc H}\big)^{-1}]=0=[e_3^{\rm as}, \mathring{\slashed{\nabla}}^{[-1]}_{\mc H}]$ holding for $r\leq 9M/4$ (see Section~\ref{sec_Kerr_star_coords}). By integrating identity \eqref{eq:e3upsilon-identity-final} over $\mathbb{S}^2_{\tau,r}$-spheres and over the interval $r\in[r_+,r_++\epsilon]$, we deduce that, for $\epsilon\leq M/4$, we have
\begin{align*}
&\lp|\int_{r_+}^{r_++\epsilon}\int_{\mathbb{S}^2_{\tau,r}}e_3^{\rm as}(\swei{\upupsilon}{0}) dr d\Omega-\int_{\mathbb{S}^2_{\tau,r_+}}\frac{\kappa(r_+) }{\sqrt{4Mr_+}}\mathring{\nablasl}_{\mc H}^{[-1]}\big(\mathring{\slashed{\triangle}}^{[-1]}_{\mc H}\big)^{-1}\swei{\Phi}{-1}d\Omega \rp|^2\\
&\quad\lesssim_\epsilon \frac{|a|}{M}\sum_{\mathfrak X\in\{\mathring{\nablasl},\mathrm{id}\}} \int_{\Sigma_{\tau}\cap\{r\leq r_++\epsilon\}}\lp[\lp( \Delta^2| e_3^{\rm as}\mathfrak X\swei{\Phi}{-1}|^2+|e_4^{\rm as}\mathfrak X\swei{\Phi}{-1}|^2+|Z\mathfrak X\swei{\Phi}{-1}|^2\rp)\rp] dr d\Omega\\
&\quad \qquad +\frac{|a|}{M}\sum_{\mathfrak X\in\{\mathring{\nablasl},\mathrm{id}\}}\int_{\Sigma_{\tau}\cap\{r\leq r_++\epsilon\}}\lp[\Delta^4| (e_3^{\rm as})^2\mathfrak X\swei{\tilde\upphi}{-1}_0|^2+\Delta^2| e_3^{\rm as}\mathfrak X\swei{\tilde\upphi}{-1}_0|^2\rp]drd\Omega \\
&\quad \qquad +\frac{|a|}{M}\sum_{\mathfrak X\in\{\mathring{\nablasl},\mathrm{id}\}}\int_{\Sigma_{\tau}\cap\{r\leq r_++\epsilon\}}\lp[(r-r_+)^4|Z^2\mathfrak X\swei{\tilde\upphi}{-1}_0|^2 +(r-r_+)^2|Z\mathfrak X\swei{\tilde\upphi}{-1}_0|^2\rp]dr d\Omega\\
&\quad \qquad +\sum_{\mathfrak X\in\{\mathring{\nablasl},\mathrm{id}\}}\int_{\Sigma_{\tau}\cap\{r\leq r_++\epsilon\}}\lp[\lp(\frac{|a|}{M}+(r-M)^2\rp)|\mathfrak X\swei{\Phi}{-1}|^2+\frac{|a|}{M}(r-M)^2|\swei{\mathfrak X\tilde\upphi}{-1}_0|^2\rp]dr d\Omega,
\end{align*}
as claimed. In the estimates, the implicit constants depend on the norm of the inverse operator $\big(\mathring{\slashed\triangle}^{[-1]}_{\mc H}\big)^{-1}$ obtained in Lemma~\ref{lemma:invertibility}.
\end{proof}

We now return to our original tensorial notation to control the first two terms on the right hand side of the estimates in Proposition~\ref{prop:L2-norm} in terms of degenerate-at-$\mc H^+$ energy norms of the extremal Maxwell components $\bsy \alpha$ and $\bsy\alphab$ and conclude the proof of Theorem~\ref{thm:middle-EB-ED-ext}.

\begin{proof}[Proof of Theorem~\ref{thm:middle-EB-ED-ext}]
We combine Propositions~\ref{prop:L2-norm} and \ref{prop:higher-derivatives} with Proposition~\ref{prop:control-near-hor-spin}. By appealing to the isomorphism of Lemma~\ref{lemma:isomorphism}, one can express Proposition~\ref{prop:control-near-hor-spin} in our original tensorial notation.
\end{proof}

\subsection{Horizon conservation laws in the \texorpdfstring{$|a|=M$}{extremal Kerr} case} \label{rmk:extremal-axisym}

In the extremal $|a|=M$ case, Lemmas~\ref{lemma:Teukolsky-invertible-identity} and \ref{lemma:Maxwell-n-identity} allow one to derive conservation laws for the middle Maxwell components along the event horizon (i.e.~at $r=r_+=|a|=M$).

We consider an axisymmetric Teukolsky field $\swei{\upalpha}{-1}$, i.e.~such that $Z\swei{\upalpha}{-1}=0$, and introduce the notation
\begin{gather*}
\swei{\upupsilon}{0}_0 :=\frac{1}{\sqrt{4\pi}}\int_{\mathbb S^2_{\tau,r_+}}\swei{\upupsilon}{0}d\Omega\,,\\
\swei{\tilde\upphi}{-1}_{0,L} :=\int_{\mathbb S^2_{\tau,r_+}}\swei{\tilde\upphi}{-1}_0 S_{0,L}^{[-1]}d\Omega\,,\qquad
\swei{\Phi}{-1}_{L} :=\int_{\mathbb S^2_{\tau,r_+}}\swei{\Phi}{-1} S_{0,L}^{[-1]}d\Omega\,,\qquad L\geq 1\,,
\end{gather*}
where $S_{m,L}^{[s]}(\theta,\tilde{\phi}{}^*)$ denote the $s$-spin-weighted spherical harmonics already introduced in Proposition~\ref{prop:upupsilon-Schwarzschild}. From identity \eqref{eq:Maxwell-n-identity}, we have
\begin{align*}
&e_3^{\rm as} \big(\swei{\upupsilon}{0}\big) = \frac{1}{2M}\lp(\p_\theta-\frac{2ia\sin\theta}{M-ia\cos\theta}+\cot\theta\rp)\lp((M-ia\cos\theta)\swei{\tilde \upphi}{-1}_0\rp)+\frac{(M\cos\theta-ia) \sin\theta}{8M^2}\swei{\Phi}{-1}
\end{align*}
along $\mc H^+$, which yields
\begin{align*}
e_3^{\rm as} \lp(\swei{\upupsilon}{0}_0\rp)
&= \frac{1}{4\sqrt{30}M}\lp(\frac{ia}{M}\sqrt{5}\swei{\Phi}{-1}_{1}- \swei{\Phi}{-1}_2\rp)+ \frac{ia}{M}\frac{(\sqrt{16}+1)}{2\sqrt{6}}\swei{\tilde \upphi}{-1}_{0,1}\\
&\qquad -\sum_{L\in 2\mathbb{\mathbb{N}}}\sqrt{\frac{(2L+1)}{4L(L+1)}} \swei{\tilde \upphi}{-1}_{0,L}+\frac{ia}{M}\sum_{L\in 2\mathbb{\mathbb{N}}+1}\sqrt{\frac{(2L+1)}{4L(L+1)}}\swei{\tilde \upphi}{-1}_{0,L}\,.\numberthis\label{eq:Maxwell-n-identity-axisym}
\end{align*}
On the other hand, from identity \eqref{eq:Teukolsky-invertible-identity}, we have
\begin{align*}
\lp(\mathring{\slashed\triangle}^{[-1]}+1\rp)\swei{\tilde \upphi}{-1}_0=
-e_3^{\rm as}\lp(\swei{\Phi}{-1}\rp)+\frac{1}{16}\sin^2\theta(1+\cos^2\theta)e_4^{\rm as}\lp(\swei{\Phi}{-1}\rp)+\frac{i a\cos\theta}{2M^2}\swei{\Phi}{-1}
\end{align*}
along $\mc H^+$, from which we obtain
\begin{align*}
\swei{\tilde \upphi}{-1}_{0,L}
&=-e_3^{\rm as}\lp(\frac{1}{L(L+1)}\swei{\Phi}{-1}_L\rp)+\frac{1}{40L(L+1)}e_4^{\rm as}\lp(\swei{\Phi}{-1}_{L}\rp) \numberthis\label{eq:Teukolsky-invertible-identity-axisym}\\
&\qquad +\frac{1}{2}\sum_{L'\geq 1}\frac{\sqrt{\left(2L+1\right)\left(2L'+1\right)}}{L(L+1)} \lp\{\frac{\sqrt{9}}{105}
\begin{pmatrix}
  L & L' & 4\\
  0 & 0 & 0
\end{pmatrix}
\begin{pmatrix}
  L & L' & 4\\
  -1 & +1 & 0
\end{pmatrix}e_4^{\rm as}\lp(\swei{\Phi}{-1}_{L'}\rp)\rp.\\
&\qquad\qquad\qquad\qquad\qquad\qquad\qquad\qquad \lp. 
+\frac{1}{14}
\begin{pmatrix}
  L & L' & 2\\
  0 & 0 & 0
\end{pmatrix}
\begin{pmatrix}
  L & L' & 2\\
  -1 & +1 & 0
\end{pmatrix} e_4^{\rm as}\lp(\swei{\Phi}{-1}_{L'}\rp)\rp.\\
&\qquad\qquad\qquad\qquad\qquad\qquad\qquad\qquad \lp. -\frac{i a}{M^2}\begin{pmatrix}
  L & L' & 1\\
  0 & 0 & 0
\end{pmatrix}
\begin{pmatrix}
  L & L' & 1\\
  -1 & +1 & 0
\end{pmatrix}\swei{\Phi}{-1}_{L'}\rp\}
\end{align*}
for any $L\in\mathbb{N}$, where we used the standard notation
$$
\begin{pmatrix}
  j_1 & j_2 & j_3\\
  -s_1 & -s_2 & -s_3
\end{pmatrix}
$$
to denote the so-called $3j$-symbol. Inserting \eqref{eq:Teukolsky-invertible-identity-axisym} into \eqref{eq:Maxwell-n-identity-axisym}, we deduce that the evolution of the spherical mean of $\swei{\upupsilon}{0}$ along the event horizon is completely determined by $\swei{\Phi}{-1}$ in an explicitly computable fashion. 

These formulas allow one to prove conservation laws (and thus, non-decay) for transversal derivatives of $\swei{\upupsilon}{0}_0$ along $\mc H^+$ provided  transversal derivatives of $\swei{\tilde\upphi}{-1}_{0,L}$ and $\swei{\Phi}{-1}_L$ themselves satisfy suitable conservation laws. Indeed, work \cite{Lucietti2012} obtains conservation laws for  $\swei{\tilde\upphi}{-1}_{0,1}$ and $\swei{\Phi}{-1}_1$; it is expected, see e.g.~\cite{Angelopoulos2018a}, that conservation laws for $\swei{\tilde\upphi}{-1}_{0,1}$ and $\swei{\Phi}{-1}_1$ with $L\geq 2$ can also be derived.

%%%%%%%%%%%%%%%%%%%%%
%% MORE DETAIL, DO NOT ERASE
%%%%%%%%%%%%%%%%%%%%%
%In the above, we have used:
%\begin{align*}
%\cos\theta &=\sqrt{\frac{4\pi}{3}}S_{10}^{[0]}(\theta,\phi)\\
%\sin\theta &=-\sqrt{\frac{8\pi}{3}}S_{10}^{[-1]}(\theta,\phi)\\
%\sin\theta\cos\theta &=-\sqrt{\frac{8\pi}{15}}S_{20}^{[-1]}(\theta,\phi)\\
%\int_0^\pi \cot\theta S_{L0}^{[-1]}(\theta,\phi)\sin\theta d\theta &= - \int_0^\pi \cot\theta S_{L0}^{[+1]}(\theta,\phi)\sin\theta d\theta = -\sqrt{\frac{(2L+1)}{L(L+1)\pi}} \text{~~for $L$ even}\\
%\int_0^\pi \cos\theta\cot\theta S_{L0}^{[-1]}(\theta,\phi)\sin\theta d\theta &= - \int_0^\pi \cos\theta \cot\theta S_{L0}^{[+1]}(\theta,\phi)\sin\theta d\theta = -\sqrt{\frac{(2L+1)}{L(L+1)\pi}} \text{~~for $L$ odd, except for $L=1$ where $=-\frac{1}{\sqrt{6\pi}}$}\\
%\sin^2\theta(1+\cos^2\theta)=-\frac{16\sqrt{\pi}}{105}S_{0,4}^{[0]}-\frac{8\sqrt{\pi}}{7\sqrt{5}}S_{0,2}^{[0]}+\frac{8\sqrt{\pi}}{5}S_{0,0}^{[0]}
%\end{align*}
%See also the website \url{https://en.wikipedia.org/wiki/Spin-weighted_spherical_harmonics#Triple_integral}.

%-------------------------------------------------------------------
%
% ANALYSIS OF THE EXTREMAL COMPONENTS
%
%-------------------------------------------------------------------

\section{Analysis of the extremal Maxwell components}
\label{sec:analysis-teukolsky}

In this section, we state energy boundedness and decay for the extremal Maxwell components.

\subsection{The \texorpdfstring{$|a|<M$}{subextremal Kerr} case after \texorpdfstring{\cite{SRTdC2023}}{[SRTdC23]}}
\label{sec:teukolsky-subextremal}

Works \cite{SRTdC2020,SRTdC2023} obtain energy boundedness and decay for general solutions to the Teukolsky equations \eqref{eq:Teukolsky-spin}. Using Lemma~\ref{lemma:isomorphism}, one can translate those statements into energy boundedness and decay for the extremal Maxwell components. 

\begin{theorem}[\cite{SRTdC2020,SRTdC2023}] \label{thm:teukolsky-EB-ILED} Fix $M>0$ and $a_0\in[0,M)$. Let $J\geq 1$, $p\in[0,2)$ and $\eta\in(0,1)$. Then, there exists a uniform constant $C=C(a_0,M,J,p,\eta)>0$ such that, for any $|a|\leq a_0$ and any solution 
$$\bsy{\mathfrak S}=(\bsy{\alpha}[\bsy{\mathfrak S}],\bsy{\alphab}[\bsy{\mathfrak S}],\widehat{\bsy{\rho}}[\bsy{\mathfrak S}],\widehat{\bsy{\sigma}}[\bsy{\mathfrak S}])$$ 
to the modified Maxwell equations \eqref{D3_sigma}--\eqref{DA_rho} arising from seed initial data prescribed on $\Sigma_0$, the following energy estimates hold.
\begin{itemize}
\item Non-degenerate energy boundedness: we have
\begin{align*}
\overline{\mathbb{E}}_p^{J}\big[r^2\bsy \alpha[\bsy{\mathfrak S}]\big](\tau)&\leq C\, \overline{\mathbb{E}}_p^{J}\big[r^2\bsy\alpha [\bsy{\mathfrak S}]\big](0)\,,\\
\overline{\mathbb{E}}_p^{J}\big[r\bsy \alphab[\bsy{\mathfrak S}]\big](\tau)&\leq C\, \overline{\mathbb{E}}^{J}_p\big[r\bsy\alphab[\bsy{\mathfrak S}]\big](0)
\end{align*}
for all $\tau\geq 0$.
\item Non-degenerate energy decay: we have
\begin{align*}
\overline{\mathbb{E}}^{J}\big[r^2\bsy \alpha [\bsy{\mathfrak S}]\big](\tau)&\leq \frac{C}{(1+\tau)^{2-\eta}}\,\overline{\mathbb{E}}_{2-\eta}^{J+2}\big[r^2\bsy \alpha [\bsy{\mathfrak S}]\big](0)\,,\\
\overline{\mathbb{E}}^J\big[r\bsy \alphab[\bsy{\mathfrak S}]\big](\tau)
&\leq \frac{C}{(1+\tau)^{2-\eta}}\,\overline{\mathbb{E}}^{J+2}_{2-\eta}\big[ r\bsy\alphab[\bsy{\mathfrak S}]\big](0)
\end{align*}
for all $\tau\geq 0$.
\item Energy decay at the event horizon: we have
\begin{align*}
    \overline{\mathbb{E}}_{\mc H^+}^1\big[r^2\bsy \alpha[\bsy{\mathfrak S}]\big](\tau,\infty)&\leq \frac{C}{(1+\tau)^{2-\eta}}\,\overline{\mathbb{E}}_{2-\eta}^{3}\big[r^2\bsy \alpha [\bsy{\mathfrak S}]\big](0)\,,\\
\overline{\mathbb{E}}_{\mc H^+}^1\big[r\bsy \alphab[\bsy{\mathfrak S}]\big](\tau,\infty)&\leq \frac{C}{(1+\tau)^{2-\eta}}\,\overline{\mathbb{E}}_{2-\eta}^{3}\big[r\bsy \alphab [\bsy{\mathfrak S}]\big](0)
\end{align*}
for all $\tau\geq 0$.
\end{itemize}
\end{theorem}

\begin{remark} The proof of Theorem~\ref{thm:teukolsky-EB-ILED} in \cite{SRTdC2020,SRTdC2023} makes use of the spin-weighted version of the Teukolsky equation \eqref{eq:Teukolsky-spin}, rather than the tensorial version of Proposition~\ref{prop:Teukolsky-tensor}, in the analysis. We direct the reader to \cite[Theorem 5.1]{SRTdC2023} for (energy) boundedness estimates and   \cite[Corollary 5.2 and Theorem 9.3]{SRTdC2023} for (energy) decay estimates in the spin-weighted notation, and to Section~\ref{rmk:translation-SRTdC} for a quick summary of how to translate between the tensorial and spin-weighted notations.
\end{remark}

\subsection{The \texorpdfstring{$|a|\approx M$}{nearly extremal Kerr} case: a conjectural statement for bounded azimuthal number}
\label{sec:teukoslsky-extremal}

In Theorem~\ref{thm:teukolsky-EB-ILED}, we have $C(a_0)\to \infty$ as $a_0\to M$, and thus the energy boundedness and decay statements for the extremal Maxwell components break down towards extremality. As already noted in Section \ref{sec:statement_theorems}, it is nonetheless expected that boundedness and (weaker) decay for the extremal Maxwell components still hold in the nearly extremal $|a|\approx M$ case. The formulation of this expectation is the content of Conjecture \ref{conj:teukolsky-EB-ILED-ext}. In this section, we shall review recent literature concerning scalar fields on Kerr spacetimes in the (nearly) extremal case and comment on the results motivating the formulation of Conjecture \ref{conj:teukolsky-EB-ILED-ext}.

A simpler toy problem for understanding the Teukolsky equation \eqref{eq:Teukolsky-spin} is the scalar wave equation. The first results on the behavior of scalar waves on extremal black holes are due to Aretakis. They showed that there are conservation laws for certain transverse derivatives of axisymmetric scalar field along $\mc H^+$ \cite{Aretakis2012} but that, though these conservation laws may lead to a slower decay rate for the axisymmetric scalar field itself, it does decay \cite{Aretakis2012a}. Indeed, for  initial data compactly supported away from $r=r_+$, axisymmetric scalar waves are expected to decay at a rate $(1+\tau)^{-2}$ along $\mc H^+$ and faster (on constant $r$ hypersurfaces) away from it \cite{Angelopoulos2018a,Casals2016}.
%Section 3.2.3 of Angelopoulos2018a 

Outside of axisymmetry, Casals, Gralla and Zimmerman \cite{Casals2016} used a heuristic argument (based on the assumption that the worse behavior is related to the so-called superradiant threshold, cf.\ Remark~\ref{rmk:SRTdC20-extremal} above) to compute the expected sharp decay rates for scalar fields supported on a fixed, but not necessarily zero, azimuthal mode $m$. They found that, even with compactly initial data supported away from $r=r_+$, the leading order trace of the scalar field on $\mc H^+$ at late-times had a significantly slower decay rate, $(1+\tau)^{-1/2}$, if $m\neq 0$. Very recently, these heuristics were made rigorous in a remarkable work of Gajic \cite{Gajic2023}, which not only confirms that $(1+\tau)^{-1/2}$ is the fastest decay one can hope for along $\mc H^+$ for general scalar fields on extremal Kerr but moreover obtains\footnote{Strictly speaking, this is conditional on global, qualitative integrability assumptions for the scalar field, see \cite[Remark 1.1]{Gajic2023}.} a precise, explicit, pointwise form for the late-time leading order contribution to the scalar field along $\mc H^+$. 

In two follow-up works \cite{Gralla2018,Casals2019a}, Casals, Gralla and Zimmerman extended their analysis to the Teukolsky equation~\eqref{eq:Teukolsky-spin}. They found that, along $\mc H^+$, one has $\bsy\alpha \sim (1+\tau)^{-3/2}$ but $\bsy\alphab \sim (1+\tau)^{1/2}$, i.e.\ growth for $\bsy\alphab$! Their work again suggests an explicit form for the late-time, near-horizon, $\upalpha^{\pm 1}$: for instance, using Lemma~\ref{lemma:isomorphism} to convert to tensorial notation, they propose
\begin{align*}
(r-M)\bsy\alphab \sim (1+\tau)^{-\frac12}\,, \quad (r-M)\nablasl\bsy\alphab \sim (1+\tau)^{-\frac12}\,, \quad (r-M)^{3/2}\nablasl_3\bsy\alphab \sim 1\,, \quad \nablasl_4  \bsy\alphab+\omegah\, \bsy\alphab \sim (1+\tau)^{-\frac12}\,,
\end{align*}
for initial data supported on a fixed azimuthal mode and compactly supported away from horizon. Note that these rates are consistent with the boundedness and decay estimates in Conjecture~\ref{conj:teukolsky-EB-ILED-ext}. In view of the previous paragraph, it seems likely that these claims, and thus Conjecture~\ref{conj:teukolsky-EB-ILED-ext}, can be made rigorous by following the techniques introduced by \cite{Gajic2023}.

\begin{remark} \label{rmk:ext-sum-m} As discussed above, most of the literature on extremal Kerr, $|a|=M$, or nearly extremal Kerr $|a|\approx M$, deals with the case of solutions to the fixed azimuthal mode, $m$, solutions to the Teukolsky and scalar wave equations. Understanding how to deal with the limit $m\to \infty$, and thus with general initial data, is a major open problem already for scalar wave equations. Therefore, we purposely exclude this regime from Conjecture~\ref{conj:teukolsky-EB-ILED-ext}. 
\end{remark}

%-------------------------------------------------------------------
%
% PUTTING EVERYTHING TOGETHER
%
%-------------------------------------------------------------------
%\newpage
\section{Proof of the main theorems} \label{sec:proof_main_theorems}

In this section, we combine the analysis of the middle Maxwell components of Section~\ref{sec:analysis-middle} with that of the extremal Maxwell components of Section~\ref{sec:analysis-teukolsky} to prove the two main theorems of the paper.

\begin{proof}[Proof of Theorem~\ref{thm:precise-EB-ED-sub}] In view of Theorem~\ref{thm:teukolsky-EB-ILED}, both the qualitative assumption for $\bsy\alphab$ and the quantitative assumption for $\bsy\alpha$ in Theorem~\ref{thm:middle-EB-ED-sub} hold true, where one identifies $\mathbb{D}[r^2\bsy\alpha](0)=\overline{\mathbb{E}}_{2-\eta}^3[r^2\bsy\alpha](0)$ and $p_\alpha=1-\eta$. Combining Theorem~\ref{thm:teukolsky-EB-ILED} with Theorem~\ref{thm:middle-EB-ED-sub} yields the conclusion.
\end{proof}

\begin{proof}[Proof of Theorem~\ref{thm:precise-EB-ED-ext}]
In view of Conjecture~\ref{conj:teukolsky-EB-ILED-ext}, both the qualitative assumption for $\bsy\alphab$ and the quantitative assumption for $\bsy\alpha$ in Theorem~\ref{thm:middle-EB-ED-ext} hold true. Combining Conjecture~\ref{conj:teukolsky-EB-ILED-ext} with Theorem~\ref{thm:middle-EB-ED-ext} yields the conclusion.
\end{proof}

%\newpage
\appendix 

\section{Further equations and identities from the Maxwell system} \label{sec:appendix_head}

In this appendix, we derive several identities which are used in the paper. The identities are derived starting from the Maxwell equations.

For later convenience, we recall the commutation identities 
\begin{alignat}{3}
[e_3^{\rm as},e_A^{\rm as}]&=\slashed{\underline{M}}{}_A^Be_B^{\rm as} -{\chib^{\sharp_2}}{}^B_Ae_B^{\rm as}&&=\slashed{\underline{M}}{}_A^Be_B^{\rm as} -\frac12\tr\chib e_A^{\rm as}-\frac12 \epsuchi {\slashed{\varepsilon}_{A}}^{B}e_B^{\rm as}\,,\label{eq:comm-e3eA}\\
[e_4^{\rm as},e_A^{\rm as}]&=\slashed{M}{}_A^Be_B^{\rm as} -{\chi^{\sharp_2}}{}^B_Ae_B^{\rm as}+(\eta+\etab)_A e_4^{\rm as}&&=\slashed{M}{}_A^Be_B^{\rm as} -\frac12\tr\chi e_A^{\rm as}-\frac12 \epschi {\slashed{\varepsilon}_{A}}^{B}e_B^{\rm as}+(\eta+\etab)_A e_4^{\rm as} \,,\label{eq:comm-e4eA}
\end{alignat}
which are obtained by specialising the general commutation formulae of \cite[Section 4.6]{Benomio2022} under the identification of $(\bsy{\mathcal{M}},\bsy{g})$ with the Kerr exterior manifold $(\mathcal{M},g)$ and $\bsy{\mathfrak{D}}_{\bsy{\mathcal{N}}}$ with the algebraically special horizontal distribution $ \mathfrak{D}_{\mathcal{N}_{\text{as}}}$ (see, in particular, \cite[Equations (74)--(75)]{Benomio2022}). We also recall the commutation identities
\begin{align*}
[\nablasl_3,\divsl]\varsigma &=-\frac12\tr\chib\divsl\varsigma+\frac12\epsuchi\curlsl\varsigma-\epsuchi\eta \wedge\varsigma \, ,\\
[\nablasl_3,\curlsl]\varsigma &= -\frac12\tr\chib\curlsl\varsigma-\frac12\epsuchi\divsl\varsigma
\end{align*}
for any $ \mathfrak{D}_{\mathcal{N}_{\text{as}}}$ one-form $\varsigma$ (see \cite[Equation (142)]{Benomio2022}).

%-------------------------------------------------------------------
% Teukolsky--Starobinsky identities
%-------------------------------------------------------------------
\subsection{The tensorial Teukolsky--Starobinsky identities}
\label{app:TS-identities}

In the following proposition, we show that the extremal Maxwell components $\bsy\alpha$ and $\bsy\alphab$ are related by differential identities akin to the Weyl identities for linearized gravity, see e.g.\ \cite{Silva-Ortigoza1997}, or the Teukolsky--Starobinsky identities \cite{Starobinsky1974, Teukolsky1974} in the spin-weighted  formalism.

\begin{proposition}[Teukolsky--Starobinsky] \label{prop:TS-identities} For any solution to the Maxwell equations \eqref{eq:del3-sigma}--\eqref{eq:del4-ualpha}, the identities
\begin{align}
\begin{split}&\frac{1}{\Sigma}\lp(\nablasl_3+\frac12\tr\chib+\frac32 \epsuchi{}\,\,^\star\rp)\lp[\Sigma \lp(\nablasl_3+\frac12\tr\chib+\frac12 \epsuchi\,\,{}^\star\rp)\,\bsy{\alpha}\rp]  \\
&\quad\qquad = -(\nablasl +2\eta) \divsl\bsy{\alphab}-{}^{\star}(\nablasl +2\eta)\curlsl\bsy{\alphab} \, ,
\end{split}\label{eq:TSplus}\\
\begin{split}
&\lp(\nablasl_4+\frac12 \tr \chi +\frac32 \epschi \,\,{}^\star\rp)\lp[ \Sigma \lp(\nablasl_4+\,\hat\omega +\frac12 \tr \chi +\frac12 \epschi \,\,{}^\star\rp)\bsy{\alphab}\rp]\\
&\quad\qquad  =-\lp(\nablasl+2\etab\rp)\divsl(\Sigma\bsy{\alpha})-{}^\star\lp(\nablasl+2\etab\rp)\curlsl(\Sigma\bsy{\alpha})
\end{split}\label{eq:TSminus}
\end{align}
hold.
\end{proposition}

\begin{proof}
We start by taking a $\nablasl_3$-derivative of equation \eqref{eq:del3-alpha'}. By using the commutation identity \eqref{eq:comm-e3eA}, we compute
\begin{align*}
&\nablasl_3\lp[\Sigma(\nablasl_3\bsy{\alpha}+\chib^{\sharp_2}\cdot\,\bsy{\alpha})\rp] \\
&\quad = (\nablasl+\,\eta-\etab)\nablasl_3(\Sigma \bsy{\rho})+{}^\star(\nablasl+\,\eta-\etab)\nablasl_3(\Sigma \bsy{\sigma})+[\nablasl_3,\nablasl](\Sigma\bsy{\rho})+[\nablasl_3,{}^\star\nablasl](\Sigma\bsy{\sigma})\\
&\quad\qquad +\nablasl_3(\eta-\etab)(\Sigma \bsy{\rho})+\nablasl_3({}^\star\eta-{}^\star\etab)(\Sigma \bsy{\sigma})\\
&\quad =(\nablasl +\,\eta-\etab)\lp[-\divsl(\Sigma\bsy{\alphab})+(\eta+\etab,\Sigma\bsy{\alphab})\rp]+{}^{\star}(\nablasl +\,\eta-\etab)\lp[-\curlsl(\Sigma\bsy{\alphab})+(\eta+\etab)\wedge(\Sigma\bsy{\alphab})\rp] \\
&\quad \qquad -\lp[{}^\star\nablasl\epsuchi -\nablasl_3(\eta-\etab)+{}^\star (\eta-\etab)\epsuchi \rp](\Sigma \bsy{\rho})+\lp[\nablasl\epsuchi +\nablasl_3({}^\star\eta-{}^\star\etab)+(\eta-\etab)\epsuchi \rp](\Sigma \bsy{\sigma})\\
&\quad\qquad +[\nablasl_3,\nablasl](\Sigma\bsy{\rho})+[\nablasl_3,{}^\star\nablasl](\Sigma\bsy{\sigma})-\epsuchi {}^\star\lp[\nablasl(\Sigma\bsy{\rho})+{}^\star\nablasl(\Sigma\bsy{\sigma})\rp]\\
&\quad=(\nablasl +\,\eta-\etab)\lp[-\divsl(\Sigma\bsy{\alphab})+(\eta+\etab,\Sigma\bsy{\alphab})\rp]+{}^{\star}(\nablasl +\,\eta-\etab)\lp[-\curlsl(\Sigma\bsy{\alphab})+(\eta+\etab)\wedge(\Sigma\bsy{\alphab})\rp]\\
&\qquad\quad -\frac32 \epsuchi\Sigma{}^\star (\nablasl_3\bsy{\alpha}+\chib^{\sharp_2}\cdot\,\bsy{\alpha})
-\frac12\tr\chib \Sigma (\nablasl_3\bsy{\alpha}+\chib^{\sharp_2}\cdot\,\bsy{\alpha}) \, ,
\end{align*}
where we also applied the geometric identities of Section~\ref{sec:algebraically-special-props}. We thus obtain
\begin{align*}
&\lp(\nablasl_3+\frac12\tr\chib+\frac32 \epsuchi{}\,\,^\star\rp)\lp[\Sigma \lp(\nablasl_3+\frac12\tr\chib+\frac12 \epsuchi\,\,{}^\star\rp)\,\bsy{\alpha}\rp]  \\
&\quad\qquad = (\nablasl +\,\eta-\etab)\lp[-\divsl(\Sigma\bsy{\alphab})+(\eta+\etab,\Sigma\bsy{\alphab})\rp]+{}^{\star}(\nablasl +\,\eta-\etab)\lp[-\curlsl(\Sigma\bsy{\alphab})+(\eta+\etab)\wedge(\Sigma\bsy{\alphab})\rp].
\end{align*}
To conclude the proof of \eqref{eq:TSplus}, we note
\begin{gather*}
-\divsl(\Sigma\bsy{\alphab})+(\eta+\etab,\Sigma\bsy{\alphab}) = -\Sigma \divsl\bsy{\alphab} \, , \quad\quad -\curlsl(\Sigma\bsy{\alphab})+(\eta+\etab)\wedge(\Sigma\bsy{\alphab})=-\Sigma\curlsl\bsy{\alphab} \, , \\
(\nablasl+\eta-\etab) (\Sigma \bsy{\alphab})= \Sigma (\nablasl+\eta-\etab)\bsy{\alphab}+\nablasl (\Sigma\bsy{\alphab}) = \Sigma (\nablasl+2\eta)\bsy{\alphab} \, .
\end{gather*}
To derive the identity \eqref{eq:TSminus}, we take a $\nablasl_4$-derivative of equation \eqref{eq:del4-ualpha'}. By using the commutation identity \eqref{eq:comm-e4eA} and the geometric identities of Section~\ref{sec:algebraically-special-props}, we obtain
\begin{align*}
&\nablasl_4\lp[\Sigma(\nablasl_4\bsy{\alphab}+\chi^{\sharp_2}\cdot\,\bsy{\alphab}+\hat\omega\bsy{\alphab})\rp] \\
%
%&\quad = -(\nablasl-\,\eta+\etab)\nablasl_4(\Sigma \bsy{\rho})+{}^\star(\nablasl-\,\eta+\etab)\nablasl_4(\Sigma \bsy{\sigma})-[\nablasl_4,\nablasl](\Sigma\bsy{\rho})+[\nablasl_4,{}^\star\nablasl](\Sigma\bsy{\sigma})\\
%&\quad\qquad +\nablasl_4(\eta-\etab)(\Sigma \bsy{\rho})-\nablasl_4({}^\star\eta-{}^\star\etab)(\Sigma \bsy{\sigma})\\
%
&\quad = -(\nablasl+2\etab)\nablasl_4(\Sigma \bsy{\rho})+{}^\star(\nablasl+2\etab)\nablasl_4(\Sigma \bsy{\sigma})-\lp([\nablasl_4,\nablasl]-(\eta+\etab)\nablasl_4\rp)(\Sigma\bsy{\rho})\\
&\quad\qquad +\lp([\nablasl_4,{}^\star\nablasl]-{}^\star(\eta+\etab)\nablasl_4\rp)(\Sigma\bsy{\sigma})+\nablasl_4(\eta-\etab)(\Sigma \bsy{\rho})-\nablasl_4({}^\star\eta-{}^\star\etab)(\Sigma \bsy{\sigma})\\
&\quad= -(\nablasl+2\etab)\nablasl_4\divsl(\Sigma\bsy\alpha)-{}^\star(\nablasl+2\etab)\curlsl(\Sigma\bsy\alpha)\\
&\quad \qquad +\lp[\nablasl\epschi -\nablasl_4({}^\star\eta-{}^\star\etab)+ 2\epschi{}^\star\etab \rp](\Sigma \bsy{\rho})+\lp[{}^\star\nablasl\epschi +\nablasl_4(\eta-\etab)+ 2\epschi\etab \rp](\Sigma \bsy{\sigma})\\
&\quad\qquad -\lp([\nablasl_4,\nablasl]-(\eta+\etab)\nablasl_4\rp)(\Sigma\bsy{\rho})+\lp([\nablasl_4,{}^\star\nablasl]-{}^\star(\eta+\etab)\nablasl_4\rp)(\Sigma\bsy{\sigma})-\epschi {}^\star\lp[-\nablasl(\Sigma\bsy{\rho})+{}^\star\nablasl(\Sigma\bsy{\sigma})\rp]\\
&\quad=-(\nablasl+2\etab)\nablasl_4\divsl(\Sigma\bsy\alpha)-{}^\star(\nablasl+2\etab)\curlsl(\Sigma\bsy\alpha)\\
&\qquad\quad -\frac32 \epschi\Sigma{}^\star (\nablasl_4\bsy{\alphab}+\chi^{\sharp_2}\cdot\,\bsy{\alphab}+\hat\omega\bsy{\alphab})
-\frac12\tr\chi \Sigma (\nablasl_4\bsy{\alphab}+\chi^{\sharp_2}\cdot\,\bsy{\alphab}+\hat\omega\bsy{\alphab}) \, ,
\end{align*}
from which the identity \eqref{eq:TSminus} follows.
\end{proof}

%-------------------------------------------------------------------
% Tensorial Teukolsky equations 
%-------------------------------------------------------------------
\subsection{The tensorial Teukolsky equations}
\label{app:Teukolsky-equations}

In Proposition~\ref{prop:Teukolsky-tensor}, we stated that, for any solution to the Maxwell equations, the extremal Maxwell components $\bsy\alpha$ and $\bsy\alphab$ each satisfy a decoupled wave equation akin to the spin-weighted Teukolsky equations \eqref{eq:Teukolsky-spin} \cite{Teukolsky1973}. In this section, we prove Proposition~\ref{prop:Teukolsky-tensor}, i.e.~we derive the tensorial Teukolsky equations \eqref{eq:Teukolsky-plus}--\eqref{eq:Teukolsky-minus}.

\begin{proof}[Proof of Proposition~\ref{prop:Teukolsky-tensor}] By using the Maxwell equations \eqref{eq:del3-sigma}--\eqref{eq:del4-ualpha} and the geometric identities of Section~\ref{sec:algebraically-special-props}, we compute
\begin{align*}
&\nablasl_4(\nablasl_3 \bsy \alpha+\chib^{\sharp_2}\cdot \bsy \alpha)\\
&\quad = [\nablasl_4,\nablasl]\bsy\rho +{}^\star[\nablasl_4,\nablasl]\bsy\sigma + (\nablasl+2\eta)\nablasl_4\bsy\rho +{}^\star(\nablasl+2\eta)\nablasl_4\bsy\sigma +2\nablasl_4\eta\bsy\rho +2\nablasl_4({}^\star\eta)\bsy\sigma\\
&\quad = (\nablasl+3\eta+\etab)\nablasl_4\bsy\rho +{}^\star(\nablasl+3\eta+\etab)\nablasl_4\bsy\sigma +2\nablasl_4\eta\bsy\rho +2\nablasl_4({}^\star\eta)\bsy\sigma \\
&\quad\qquad +\Big\{[\nablasl_4,\nablasl]-(\eta+\etab)\nablasl_4\Big\}\bsy\rho+{}^\star\Big\{[\nablasl_4,\nablasl]-(\eta+\etab)\nablasl_4\Big\}\bsy\sigma\\
&\quad  =(\nablasl+3\eta+\etab)(\divsl\bsy\alpha+(\eta+\etab)\cdot \bsy\alpha) -{}^\star (\nablasl+3\eta+\etab)(\curlsl \bsy\alphab+(\eta+\etab)\wedge \bsy\alpha)\\
&\quad \qquad +\lp[2\nablasl_4\eta-(\nablasl+3\eta+\etab)\tr\chi +{}^\star(\nablasl +3\eta+\etab)\epschi\rp]\bsy\rho\\
&\qquad \quad +\lp[2\nablasl_4({}^\star\eta)-{}^\star(\nablasl+3\eta+\etab)\tr\chi -(\nablasl +3\eta+\etab)\epschi\rp]\bsy\sigma-\frac{1}{2}\lp[3\tr \chi -\epschi{}^\star\rp]\lp(\nablasl\bsy\rho+{}^\star\nablasl\bsy\sigma\rp)\\
&\quad  =(\nablasl+3\eta+\etab)(\divsl\bsy\alpha+(\eta+\etab)\cdot \bsy\alpha) -{}^\star (\nablasl+3\eta+\etab)(\curlsl \bsy\alphab+(\eta+\etab)\wedge \bsy\alpha)\\
&\quad\qquad -\lp[\frac{3}{2}\tr \chi -\frac{1}{2}\epschi{}^\star\rp]\lp(\nablasl\bsy\rho+{}^\star\nablasl\bsy\sigma+2\eta \bsy\rho +2{}^\star\eta \bsy\sigma\rp)
\end{align*}
and
\begin{align*}
&\nablasl_3(\nablasl_4\bsy\alphab+\chi^{\sharp_2}\cdot\bsy\alphab +\hat\omega\bsy\alphab) \\
&\quad= -[\nablasl_3,\nablasl]\bsy\rho +{}^\star[\nablasl_3,\nablasl]\bsy\sigma -(\nablasl+2\etab)\nablasl_3\bsy\rho +{}^\star(\nablasl+2\etab)\nablasl_3\bsy\sigma -2\nablasl_3\etab\bsy\rho +2\nablasl_3({}^\star\eta)\bsy\sigma \\
&\quad = (\nablasl+2\etab)\divsl\bsy\alphab -{}^\star (\nablasl+2\etab)\curlsl \bsy\alphab-\frac{1}{2}\lp[3\tr \chib -\epsuchi\rp]\lp(-\nablasl\bsy\rho+{}^\star\nablasl\bsy\sigma\rp) \\
&\quad\qquad +  \lp[-2\nablasl_3\etab+(\nablasl+2\etab)\tr\chib-{}^\star(\nablasl+2\etab)\epsuchi\rp]\bsy\rho+  \lp[2\nablasl_3({}^\star\etab)-{}^\star(\nablasl+2\etab)\tr\chib-(\nablasl+2\etab)\epsuchi\rp]\bsy\sigma\\
&\quad\qquad -\Big\{[\nablasl_3,\nablasl]+\frac12 (\tr\chib+\epsuchi{}^\star)\nablasl\Big\}\bsy\rho+{}^\star\Big\{[\nablasl_3,\nablasl]+\frac12 (\tr\chib+\epsuchi{}^\star)\nablasl\Big\}\bsy\sigma\\
&\quad= (\nablasl+2\etab)\divsl\bsy\alphab -{}^\star (\nablasl+2\etab)\curlsl \bsy\alphab-\lp[\frac{3}{2}\tr \chib -\frac{1}{2}\epsuchi{}^\star\rp]\lp(-\nablasl\bsy\rho+{}^\star\nablasl\bsy\sigma-2\etab \bsy\rho +2{}^\star\etab \bsy\sigma\rp) \, .
\end{align*}
To conclude the proof, we note
\begin{gather*}
\divsl\bsy{\alpha}+(\eta+\etab,\bsy{\alpha}) = \Sigma^{-1} \divsl(\Sigma\bsy{\alpha}) \, , \quad\quad \curlsl\bsy{\alpha}+(\eta+\etab)\wedge\bsy{\alpha}=\Sigma^{-1}\curlsl(\Sigma \bsy{\alpha}) \, ,\\
(\nablasl+3\eta+\etab) (\Sigma^{-1} \bsy{\alpha})= \Sigma^{-1} (\nablasl+3\eta+\etab)\bsy{\alpha}-\Sigma^{-2}\nablasl (\Sigma \bsy{\alpha}) = \Sigma^{-1} (\nablasl+2\eta) \bsy{\alpha}
\end{gather*}
and similarly for $\bsy{\alphab}$.
\end{proof}

\subsection{The tensorial Fackerell--Ipser equations}
\label{app:FI-equations}

Similarly to the extremal Maxwell components $\bsy\alpha$ and $\bsy\alphab$,  the middle Maxwell components $\bsy\rho$ and $\bsy\sigma$ satisfy (now coupled) wave equations.

\begin{proposition} For any solution to the Maxwell equations \eqref{eq:del3-sigma}--\eqref{eq:del4-ualpha}, the identities
\begin{align*}
\lp[\nablasl_3\nablasl_4 - \slashed\triangle -(\eta-\etab)\cdot \nablasl\rp](\Sigma \bsy \rho) &= 
-\chib \cdot (\eta \otimes (\Sigma\bsy \alpha)) -\lp[\epsuchi \nablasl_3 +{}^\star(\eta-\etab)\cdot  \nablasl -2\sigma_G\rp] (\Sigma \bsy \sigma) \, ,\\
\lp[\nablasl_3\nablasl_4 - \slashed\triangle - (\eta-\etab)\cdot \nablasl\rp] (\Sigma \bsy \sigma) &=\chib\wedge (\eta \otimes (\Sigma\bsy \alpha))+\lp[\epsuchi \nablasl_3-{}^\star (\eta-\etab)\cdot \nablasl-2\sigma_G\rp](\Sigma \bsy \rho)
\end{align*}
hold. We recall the definition of $\sigma_G$ from Section~\ref{sec:algebraically-special-props}.
\end{proposition}

\begin{proof} We take a $\nablasl_3$-derivative of equation \eqref{eq:del4-rho'} and obtain
\begin{align*}
\nablasl_3\nablasl_4 (\Sigma\bsy \rho)&= [\nablasl_3,\divsl](\Sigma\bsy \alpha)+\divsl \nablasl_3 (\Sigma \bsy \alpha)-\nablasl_3\lp(\epsuchi(\Sigma\bsy \rho)\rp)\\
&=-\frac12\tr\chib\divsl (\Sigma\bsy\alpha) +\frac12\epsuchi\curlsl(\Sigma\bsy\alpha)-\epsuchi\eta \wedge(\Sigma\bsy \alpha) +\frac12\divsl  \lp(\tr\chib\Sigma\bsy\alpha -\epsuchi{}^\star(\Sigma\bsy\alpha) \rp)\\
&\qquad +\divsl\lp(\nablasl (\Sigma\bsy \rho)+{}^\star\nablasl (\Sigma\bsy \sigma)\rp) +\divsl\lp((\eta-\etab)\Sigma\bsy \rho+{}^\star(\eta-\etab)\Sigma\bsy \sigma\rp) \\ 
&\qquad -\nablasl_3\epsuchi(\Sigma\bsy \rho)-\epsuchi \nablasl_3(\Sigma \bsy \sigma)\\
&= -\epsuchi\eta \wedge(\Sigma\bsy \alpha)+\frac12 (\nablasl \tr\chib)\cdot (\Sigma\bsy \alpha)-\frac12\nablasl \epsuchi \wedge (\Sigma \bsy \alpha)\\
&\qquad +\slashed\triangle (\Sigma\bsy \rho) +(\eta-\etab)\cdot[\nablasl(\Sigma \bsy \rho)-{}^\star \nablasl (\Sigma \bsy \sigma)]+(2\curlsl\eta -\nablasl_3\epsuchi) (\Sigma\bsy \sigma)-\epsuchi \nablasl_3(\Sigma \bsy \sigma)\\
&=-\chib \cdot (\eta \otimes (\Sigma\bsy \alpha)) +\slashed\triangle (\Sigma\bsy \rho) +(\eta-\etab)\cdot[\nablasl(\Sigma \bsy \rho)-{}^\star \nablasl (\Sigma \bsy \sigma)]+2\sigma_G\Sigma\bsy \sigma-\epsuchi \nablasl_3(\Sigma \bsy \sigma) \, .
\end{align*}
Similarly, we have
\begin{align*}
\nablasl_3\nablasl_4 (\Sigma\bsy \sigma)&=\chib\wedge (\eta \otimes (\Sigma\bsy \alpha))+\slashed\triangle (\Sigma\bsy \sigma)-{}^\star (\eta-\etab)\cdot [\nablasl(\Sigma \bsy \rho)-{}^\star \nablasl (\Sigma \bsy \sigma)]-2\sigma_G\Sigma\bsy \rho +\epsuchi \nablasl_3(\Sigma\bsy \rho) \, ,
\end{align*}
which concludes the proof.
\end{proof}

It was shown by Fackerell and Ipser \cite{Fackerell1972} that, in the Newman--Penrose formalism, the middle Maxwell components satisfy a \textit{decoupled} wave equation with a complex potential. By introducing $\swei{\upupsilon}{0}$ as defined in \eqref{eq:def-upupsilon}, one obtains the following.

\begin{proposition}[Fackerell--Ipser equation] \label{prop:FI-NP} For any solution to the Maxwell equations \eqref{eq:Maxwell-m}--\eqref{eq:Maxwell-n}, the identity
\begin{align}
\Box_{g} (\kappa^{-1}\swei{\upupsilon}{0})+\frac{2M}{(r-ia\cos\theta)^3}(\kappa^{-1}\swei{\upupsilon}{0})=0 \label{eq:FI-NP}
\end{align}
holds.
\end{proposition}

\subsection{The Maxwell equations for transformed spin-weighted quantities} 
\label{sec:appendix-MaxwellNP}

In light of the definitions \eqref{eq:def-upphi}--\eqref{eq:def-Phi} (originally introduced in \cite{SRTdC2020,SRTdC2023}), it is convenient to recast the Maxwell equations \eqref{eq:Maxwell-m}--\eqref{eq:Maxwell-n} as 
\begin{align}
\sqrt{2}\overline\kappa m (\swei{\upupsilon}{0}) 
%&= \frac{k^2(r^2+a^2)}{\sqrt{2}\Delta}\uL\lp(\frac{\sqrt{r^2+a^2}}{\kappa}\swei{\upphi}{+1}_0\rp)
&= \frac{\kappa}{\sqrt{2(r^2+a^2)}}\lp(\swei{\Phi}{+1}+\frac{a(a+ir\cos\theta)}{\kappa}\swei{\upphi}{+1}_0\rp),\label{eq:Maxwel-m-2}\\
\sqrt{2}\kappa \overline m (\swei{\upupsilon}{0}) 
%&= \frac{k^2(r^2+a^2)}{\sqrt{2}\Delta}L\lp(\frac{\sqrt{r^2+a^2}}{\kappa}\swei{\upphi}{-1}_0\rp)\\
&= \frac{\kappa}{\sqrt{2(r^2+a^2)}}\lp(\swei{\Phi}{-1}-\frac{a(a+ir\cos\theta)}{\kappa}\swei{\upphi}{-1}_0\rp),\label{eq:Maxwel-m-bar-2}\\
\frac{\Sigma}{r^2+a^2}\, e_4^{\rm as}(\swei{\upupsilon}{0}) &= \frac{\kappa^2}{\sqrt{2}(r^2+a^2)}(\sqrt  2 \kappa \overline m +\cot\theta)\lp(\frac{\sqrt{r^2+a^2}}{\kappa}\swei{\upphi}{+1}_0\rp),\label{eq:Maxwel-l-2}\\
\frac{\Delta}{r^2+a^2}\, e_3^{\rm as}(\swei{\upupsilon}{0}) &= \frac{\kappa^2}{\sqrt{2}(r^2+a^2)}(\sqrt  2 \overline\kappa  m +\cot\theta)\lp(\frac{\sqrt{r^2+a^2}}{\kappa}\swei{\upphi}{-1}_0\rp) , \label{eq:Maxwel-n-2}
\end{align}
where $m=\frac{1}{\sqrt{2}}(e_1^{\rm as}+ie_2^{\rm as})$ for the choice of $(e_1^{\rm as},e_2^{\rm as})$ given in \eqref{explicit_e1_e2}.

\phantomsection

\bibliographystyle{halpha-abbrv-rita}
{\small \bibliography{Maxwell}}

\end{document}